\documentclass{article}
\usepackage{bm, amssymb, amsmath, amsthm, accents, algorithmic, algorithm, appendix, breqn, bigints, booktabs, caption,  color, enumitem, graphicx, multirow, mathtools, mathrsfs, subcaption, parskip, tabularx, tikz, verbatim, algorithm,algorithmic,float, hyperref}
\usepackage[authoryear]{natbib}
\usepackage{bm,bbm}
\usepackage{enumitem}
\usepackage{calligra}
\usepackage{calrsfs}
\usepackage[mathcal]{eucal} 
\allowdisplaybreaks

\DeclareMathOperator*{\essinf}{ess\,inf}

\newtheorem{prop}{Proposition}[section]

\newtheorem{lem}{Lemma}[section]
\newtheorem{thm}{Theorem}[section]
\newtheorem{cor}{Corollary}[section]
\newtheorem{defn}{Definition}[section] 
\newtheorem{rem}{Remark}[section]

\newtheorem{assumptionL}{\unskip} 

\newtheorem{assumptionX}{\unskip} 

\newtheorem{assumptionB}{\unskip} 

\newtheorem{assumptionW}{\unskip} 

\newcommand{\beq}{\begin{equation}}
\newcommand{\eeq}{\end{equation}}

\def\eq#1/{(\ref{e:#1})}

\makeatletter
\@addtoreset{equation}{section}
\makeatother
\makeatletter
\@addtoreset{equation}{section}
\makeatother

\newcommand{\red}{\color{red}}

\newcommand{\al}{\alpha}

\newcommand{\la}{\lambda}
\newcommand{\La}{\Lambda}
\newcommand{\ff}{\infty}
\newcommand{\ra}{\rightarrow}

\newcommand{\ep}{\epsilon}

\newcommand{\bE}[1]{\mathbf{E}\left[ #1 \right]}
\newcommand{\bP}[1]{\mathbf{P}\left( #1 \right)}
\newcommand{\var}[1]{\mathbf{Var}\left( #1 \right)}

\DeclarePairedDelimiterX{\inp}[2]{\langle}{\rangle}{#1, #2}

\newcommand{\R}{\mathbb{R}}

\newcommand{\beqq}{\begin{eqnarray}}
\newcommand{\eeqq}{\end{eqnarray}}
\newcommand{\ta}{\theta}

\newcommand{\interior}[1]{{\kern0pt#1}^{\mathrm{o}}}

\def\Real{\mathbb R}

\begin{document}
\title{Sharp Large Deviations and Gibbs Conditioning for Threshold Models in Portfolio Credit Risk}
\author{Fengnan Deng \\ Department of Statistics \\ George Mason University \\ Fairfax, VA 22030 \and Anand N.\ Vidyashankar \\Department of Statistics \\George Mason University \\Fairfax, VA 22030 \and  Jeffrey F. Collamore \\ Department of Mathematical Sciences \\ University of Copenhagen \\
Universitetsparken 5\\ 2100 Copenhagen {\O}, Denmark}
\date{ }
\maketitle
\begin{abstract}
We obtain sharp large deviation estimates for exceedance probabilities in dependent triangular array threshold models with a diverging number of latent factors. The prefactors quantify how latent-factor dependence and tail geometry enter at leading order, yielding three regimes: Gaussian or exponential-power tails produce polylogarithmic refinements of the Bahadur-Rao $n^{-1/2}$ law; regularly varying tails yield index-driven polynomial scaling; and bounded-support (endpoint) cases lead to an $n^{-3/2}$ prefactor. We derive these results through Laplace-Olver asymptotics for exponentially weighted integrals and conditional Bahadur-Rao estimates for the triangular arrays. Using these estimates, we establish a Gibbs conditioning principle in total variation: conditioned on a large exceedance event, the default indicators become asymptotically i.i.d., and the losses are exponentially tilted (with the boundary case handled by an endpoint analysis). 
The result extends to blocks of size $k_n=o(n)$. As illustrations, we obtain second-order approximations for Value-at-Risk and Expected Shortfall, clarifying when portfolios operate in the genuine large-deviation regime. The results provide a transferable set of techniques—localization, curvature, and tilt identification —for sharp rare-event analysis in dependent threshold systems.
\\

\end{abstract}

\setlength{\parindent}{10pt}
\section{Introduction}
Threshold models generate rare-event (large deviations) phenomena across diverse domains. In credit risk, defaults occur when asset values fall below a threshold; in epidemic modeling, infections spread once exposures exceed a critical level; in reliability and environmental models, system failure occurs at critical stress levels.  All these problems reduce to analyzing sums of threshold indicators with dependence driven by latent factors; in credit risk, such factor structures are classical and widely used \cite{Merton1974,DuffieSingleton2003,Lando2004,mcneil2015quantitative,vasicek2002distribution}. We study this in a triangular array framework, motivated by credit portfolios but with methods applicable to many dependent threshold systems. Specifically, we provide sharp prefactors for a dependent portfolio model with a common factor across both light‑tail and heavy‑tail regimes; prior sharp results typically treat i.i.d. light tails, or treat heavy tails without unified prefactors.

Our framework is a triangular array factor model with a diverging number of factors.
We let
\begin{equation} \label{eq:1.1new}
 Y_i^{(n)} =  \sum_{j=1}^{k_n} a_{j}^{(n)} Z_j^{(n)} + b^{(n)} \ep_i^{(n)}, \quad i=1,2, \ldots, n,
\end{equation}
where $\sum_{i=1}^n (a_{i}^{(n)})^2+(b^{(n)})^2=1$, $b^{(n)} \to b \in (0,1)$, and $\sum_{i=1}^{k_n} a_{i}^{(n)}Z_i^{(n)}$ converges in distribution to a limiting random variable ${\cal Z}$ (which potentially could be degenerate).  Let $X_i^{(n)} = {\bf 1}_{\{Y_i^{(n)} \le v\}}$ denote the default indicator. Also, set
\begin{eqnarray} \label{eq:1.2new}
    L_n = \sum_{i=1}^n U_i^{(n)} X_i^{(n)},
    \quad \text{and} \quad U_i^{(n)} \stackrel{d}{\to} U.
\end{eqnarray}
When $k_n=k$ is fixed and the summands are i.i.d.\ and independent of $n$, this reduces to the Vasicek model and its multifactor extensions (see \cite{Merton1974,glasserman2007large,DuffieSingleton2003,Lando2004,mcneil2015quantitative,vasicek2002distribution}). Early large-deviation analyses for credit portfolios either conditioned away factor dependence or worked only at a logarithmic scale; see, e.g., \cite{dembo2004large,bassamboo2008portfolio,MaierWuthrich2009, Giesecke2013,DelsingMandjes2021}. 
Related sharp asymptotics were obtained in the complementary near-critical regime $x_n \uparrow \mu_U$, where probabilities of the form $\mathbf{P}(L_n \geq n x_n)$ were studied \cite{collamore2022sharp}. 
In contrast, our focus is on the large-deviation regime with fixed thresholds $x>\mu_U$, where the dependence induced by common factors yields qualitatively different prefactors; in particular, the sharp prefactors reflect the latent-factor dependence and tail geometry, departing from the classical Bahadur-Rao law.

The asymptotic behavior of $\bP{L_n \ge nx}$ critically depends on the support of the limiting factor distribution $\mathcal{Z}$. Let $\kappa \coloneqq \mathrm{ess\,inf}\, \mathcal{Z}$ If $\kappa = -\infty$, one obtains
\begin{align*}
\mathbf{P}\left(\frac{L_n}{n} \ge x\right) 
\sim n^{-1/2} \, e^{-n\phi_n(-\tilde M_n)} \, e^{-n\Lambda^*_U(x)} \, (H_n(-\tilde M_n))^{-1}, 
\qquad n \to \infty,
\end{align*}
where $\{\tilde M_n\} \uparrow \infty$ and $\phi_n$, $H_n$ are explicitly determined functions. 
In particular, the prefactor deviates from Bahadur–Rao’s classical $n^{-1/2}$ law, and in the Gaussian case exhibits poly-logarithmic corrections. 
For symmetric regularly varying tails, the scaling is instead determined by the relative indices of $\mathcal{Z}$ and $\epsilon$. 
When $\kappa > -\infty$ and $\mathcal{Z}$ is non-degenerate, the sharp asymptotics are different: 
\begin{align*}
\bP{L_n \geq nx} \sim n^{-3/2} \, e^{-n\Lambda^*_U(x;\kappa)}, \qquad n \to \infty,
\end{align*}
while if $\mathcal{Z}$ is degenerate at $\kappa$, the classical Bahadur–Rao $n^{-1/2}$ prefactor reemerges. Thus, the dependence induced by the common factors leaves a visible footprint in the sharp asymptotics, even though it disappears under logarithmic large deviation scaling.

One of our primary methodological contributions is the use of \emph{endpoint (boundary) Laplace expansions} in the sense of Olver \cite{olver1997asymptotics}, which sharpen the classical Laplace method for exponentially weighted (Laplace–type) integrals. On the probabilistic side, we rely on conditional sharp large–deviation estimates for triangular arrays \cite{bovier2014conditional,liao2024geometric} and on recent higher–order refinements for Laplace integrals \cite{fukuda2025higher}. These tools naturally produce the distinct asymptotic regimes described above, with prefactors that depend on the support of $\mathcal Z$ and on the underlying tails of both $\mathcal Z$ and $\{\epsilon_i^{(n)}\}$ (e.g., log–smooth/Gaussian vs.\ regularly varying). Concretely, light-tailed (Gumbel-type) tails yield polylogarithmic corrections, whereas regularly varying tails (Fréchet-type) yield index–driven power-law scalings. Only in the special (degenerate–factor) case where $\mathcal Z$ is constant do we recover the classical Bahadur–Rao prefactor $n^{-1/2}$. We expect these endpoint Laplace methods to be useful more broadly for large deviations with triangular arrays.

We also present a distribution–agnostic formulation by replacing parametric templates with \emph{log–smooth} and \emph{self–neglecting} tail conditions on the right tail of $\epsilon$ and the left tail of $\mathcal Z$. These conditions are standard and widely satisfied within the Gumbel maximum domain of attraction \cite{deHaanFerreira2006,BinghamGoldieTeugels1987}. Combined with our analysis for regularly varying tails, this yields sharp asymptotics under broad distributional assumptions.


We further establish a \emph{Gibbs–type conditional limit} for the joint law of $(U^{(n)}_{1:k},X^{(n)}_{1:k})$ given $\{L_n\ge nx\}$ with $x>\mu_U$. For i.i.d.\ sequences, related conditional limits are classical in large deviations (see \cite{dembo2009large}); recent refinements appear in \cite{AsmussenandGlynn24}. Our threshold setting is more intricate because $X_i^{(n)}=\mathbf 1\{\mathcal Z_n+b^{(n)}\epsilon_i^{(n)}\le v\}$ enters multiplicatively. Our proof combines a single–ratio (Bayes) decomposition with the uniform conditional Bahadur–Rao expansion and the endpoint Laplace analysis. Two regimes arise. If $\operatorname{ess\,inf}\mathcal Z=\kappa=-\infty$ (unbounded left support), then $X_i^{(n)}\Rightarrow 1$ and $U_i^{(n)}$ converges to the \emph{exponentially tilted} law $P_U^{\theta_x}$. In the context of credit risk, this would imply that every obligor defaults in the limit, and that the distribution of the losses $U^{(n)}_i$ are exponentially tilted. If $\operatorname{ess\,inf}\mathcal Z=\kappa>- \infty$ (boundary–dominated case), then $X_i^{(n)}\Rightarrow{\rm Bernoulli}(p_{\kappa})$ with $p_{\kappa}=F_\epsilon((v-\kappa)/b)$, and $U_i^{(n)}$ converges to the one–step tilt $P_U^{\bar\theta_{x,\kappa}}$. In both cases the convergence holds in total variation for each fixed $k$. The argument extends to $k=k_n\to\infty$ with $k_n/n\to 0$.
Finally, while our focus is on threshold losses $L_n=\sum_{i=1}^n U_i^{(n)}X_i^{(n)}$, similar techniques apply to $L_n=\sum_{i=1}^n U_i^{(n)}f(X_i^{(n)})$ for measurable $f$ satisfying mild smoothness and monotonicity conditions.



The rest of the paper is organized as follows: Section \ref{sec:Main_results} contains sharp large deviation and Gibbs conditioning results. Section \ref{sec:preliminary} is devoted to preliminary results and the probability estimates required to prove the main theorems. Section \ref{sec:proof} contains the proofs of the main results, and Section \ref{sec:concluding} is devoted to concluding remarks and further extensions. The appendices, included in the supplementary materials, provide detailed proofs of lemmas and propositions and illustrate how our methods can obtain sharper expressions for Value-at-Risk (VaR) and Expected Shortfall (ES) in factor models. A small numerical study is also included in Appendix~\ref{app:E}.

\section{Main Results} \label{sec:Main_results}

\subsection{Model description and assumptions}
For any random variable $W \in {\mathbb R}$, let $F_W$ denote the distribution function of $W$; let $f_W$ denote its density function (if it exists); and let $\lambda_W$ and $\Lambda_W$
denote its moment generating function and cumulant generating function, respectively, that is,
\[
\lambda_W(\theta) = {\bf E} \left[ e^{\theta W} \right] \quad \text{and} \quad \Lambda_W(\theta) = \log \lambda_W(\theta), \quad \text{for all} \:\: \theta \in {\mathbb R}.
\]
Also, for any function $f: {\mathbb R} \to {\mathbb R}$, let $f^\ast$ denote the convex conjugate (or Fenchel-Legendre transform) of $f$; namely,
$f^\ast(x) = \sup_{\theta \in {\mathbb R}} \left\{ \theta x - f(\theta)\right\},$ for all $x \in {\mathbb R}$; and let ${\mathfrak D}_f$ denote the domain of $f$.  Finally, for any Borel
set $B \subset {\mathbb R}$, let int\,$B$ denote the interior of $B$.

Our objective is to study the tail behavior of a portfolio credit risk model consisting of 
$n$ obligors.  Default occurs for the $i$th obligor if $Y_i^{(n)} \le v$ for some prespecified threshold $v \in {\mathbb R}$, where 
the normalized asset values $\{ Y_i^{(n)}:  i=1,\ldots, n \}$ will be modeled using a high-dimensional factor model, namely,
\begin{equation} \label{eq:2.1new}
 Y_i ^{(n)} =  \sum_{j=1}^{k_n} a_{j}^{(n)} Z_j^{(n)} + b^{(n)} \ep_i^{(n)}, \quad i=1,2, \ldots, n,
\end{equation}
where  $\{ Z_j^{(n)}:  j=1,\ldots, k_n \}$ is a collection of random variables in ${\mathbb R}$ (possibly dependent), $\{ k_n \}$
is a nondecreasing sequence of positive integers, $\{a_j^{(n)}:  j=1,\ldots,k_n\}$ is a sequence taking values in $[0,1)$,
and $b^{(n)} \in (0,1)$.  Moreover, assume that  $\{ \epsilon_i^{(n)}:  i=1,\ldots,n\}  \subset {\mathbb R}$ is an i.i.d.\ sequence,
independent of $\{ Z_j^{(n)}:  j=1,\ldots, k_n \}$ and with a distribution which is independent of $n$.
{\it We assume throughout the article that the nonnegative constants $\{ a_j^{(n)}:  j=1,\ldots,k_n \}$ and $b^{(n)}$ have been 
normalized such that}
\begin{equation} \label{assumption1}
\sum_{j=1}^{k_n} (a_{j}^{(n)})^2 + (b^{(n)})^2 =1, \quad \text{for all }n;
 \tag{$A1$}
\end{equation}
{\it and we further assume that} 
\begin{equation} \label{assumption2}
b := \lim_{n\to\infty}b^{(n)} \in (0,1).
 \tag{$A2$}
\end{equation}

For ease of notation, denote the aggregate common factors by
\begin{equation*}
{\cal Z}_n := a_1^{(n)} Z_1^{(n)} + \cdots + a_{k_n}^{(n)} Z_{k_n}^{(n)}, \quad n \in {\mathbb Z}_+.
\end{equation*}

Next, let $U_i^{(n)}$ denote the loss incurred from the $i^{\text{th}}$ loan (the ``loss given default'', or LGD), and assume that $\{ U_i^{(n)}:  i=1,\ldots, n \}$ is an i.i.d.\ sequence (possibly dependent on $n$). We assume that the loss size is independent of the event of default; thus, $\{ U_i^{(n)}:  i=1,\ldots, n \}$ may be taken to be independent of  $\{ Z_j^{(n)}:  j=1,\ldots, k_n \}$
and of $\{\epsilon_i^{(n)}:  i=1,\ldots,n\}$ (see Appendix \ref{app:F}, Remark \ref{rem:U-dep-X} for the case where $U$ may depend on $X$).
Set $X_i^{(n)} = {\bf 1}_{\{ Y_i^{(n)} \le v \}}$, and denote the total loss by
\begin{equation} \label{loss}
L_n = \sum_{i=1}^n U_i^{(n)} X_i^{(n)},  \quad n=1,2,\ldots.
\end{equation}
Our goal is to investigate the behavior of ${\bf P} \left( L_n \ge nx \right)$ as $n \to \infty$ for different choices of $x \in {\mathbb R}_+$. Finally, as is 
standard in large deviation analysis, we postulate that there exist random variables $U$ and ${\cal Z}$ such that
\begin{equation} \label{assumption3}
U_i^{(n)} \stackrel{d}{\to} U \quad  \text{and} \quad {\cal Z}_n \stackrel{p}{\to} {\cal Z} \quad \text{as} \quad n \to \infty.
\tag{$A3$}
\end{equation}
As in the previous section, set $\kappa = \essinf{\cal Z}$ and  $\hat{X}^{(z)} = {\bf 1}_{\{z + b\epsilon \leq v\}}$,
and define
\[
\Lambda_U(\theta ; z) = \log {\bf E}\left[ e^{ \theta
 U \hat{X}^{(z)} } \right], \quad \theta \in {\mathbb R}, \:z \in {\mathbb R},
\]
where $U$ is identified to be the limiting random variable in \eqref{assumption3}, and $\epsilon$ is an independent copy of $\epsilon_i^{(n)}$. Let \(\Theta:=\{\theta\ge 0:\Lambda_U(\theta)<\infty\}\) and \(\theta_0:=\sup\Theta\in(0,\infty]\).
Fix a compact interval \([0,\theta_\star]\subset(0,\theta_0)\) containing the relevant tilts \(\theta_x\). We now turn to the sharp (non-logarithmic) analysis, for which we assume the following regularity.

\noindent{\it Large deviation regularity:}\vspace*{-.1in}
\begin{assumptionL}\label{asp:L2}
(Uniform smooth convergence on \([0,\theta_\star]\)). 
\(\lambda_{U^{(n)}}(\theta)\to\lambda_U(\theta)\) and \(\lambda'_{U^{(n)}}(\theta)\to\lambda'_U(\theta)\) uniformly for \(\theta\in[0,\theta_\star]\),
and \(b^{(n)}\to b\).
When rates are used in the proofs, we additionally require 
\(n\lvert b^{(n)}-b\rvert\to 0\) and
\(n\sup_{\theta\in[0,\theta_\star]}\{\lvert\lambda_{U^{(n)}}(\theta)-\lambda_U(\theta)\rvert+\lvert\lambda'_{U^{(n)}}(\theta)-\lambda'_U(\theta)\rvert\}\to 0\).

\end{assumptionL}
\begin{assumptionL}\label{asp:L3}
    For $x\in(\mu_U,\infty)$, assume that $\Lambda_{U^{(n)}}'(\theta)=x$ and $\Lambda_{U}^'(\theta)=x$ have unique roots $\theta_x^{(n)}$ and $\theta_x$ respectively in $(0, \theta_\star)$. Assume $U^{(n)}$, $U$ are nonlattice and $\mathbf{E}_{\theta_x}[|U|^3]$ and $\mathbf{E}_{\theta_x^{(n)}}[|U^{(n)}|^3]$ are finite. $\mathbf{E}_{\theta_x}$ and $\mathbf{E}_{\theta_x^{(n)}}$ are under  exponentially tilted distribution $\mathbf{P}_{\theta_x}$ and $\mathbf{P}_{\theta_x^{(n)}}$, where $\mathrm{d}\mathbf{P}_{\theta_x}(u) =\exp\{\theta_x u - \Lambda_{U}(\theta_x)\} \mathrm{d}\mathbf{P}(u)$ and $\mathrm{d}\mathbf{P}_{\theta_x^{(n)}}(u)=\exp\{\theta_x^{(n)} u - \Lambda_{U^{(n)}}(\theta_x^{(n)})\}\mathrm{d}\mathbf{P}(u)$.
\end{assumptionL}
We note here that, to keep assumptions transparent and easy to verify, we state results under \ref{asp:L2}–\ref{asp:L3}; the proofs in fact extend under weaker, localized smoothness/moment conditions around the tilt together with a conditional Berry–Esseen/Bahadur–Rao input for the triangular arrays (see Appendix \ref{app:C}). Before we state our main Theorem, we introduce two important classes of distributions that arise in practice. To this end, recall that a function \( h \) on \( (0, \infty) \) is said to be regularly varying with index \( \alpha \) if
\[
\lim_{x \to \infty} \frac{h(tx)}{h(x)} = t^\alpha \quad \text{for all } t > 0 \text{ and some } \alpha \in \mathbb{R},
\]
or equivalently, \( h(x) = L(x) x^\alpha \) for some slowly varying function \( L \) (i.e., a function that grows more slowly than any power function; for a precise definition, see \cite{feller1970oscillations}). We say that a random variable \( W \) belongs to the \emph{class of symmetric regularly varying distributions with index \( \alpha \)} (which we denote by \( RV_\alpha \)) if
\[
\mathbf{P}(|W|>x) = L_W(x) x^{-\alpha}, \quad x \in \mathbb{R_+},
\]
for some slowly varying function \( L_W \). We assume that $L_W(\cdot)$ is twice continuously differentiable. Also, recall that the distribution function \( F_W \) of a random variable \( W \) belongs to the class of centered exponential power distributions (or generalized normal distributions) if the density is of the form
\[
f_W(x) = \frac{\gamma}{2 b \Gamma(\gamma^{-1})} \exp\left\{-\left(\frac{|x|}{b}\right)^{\gamma}\right\} \coloneqq \beta_W \exp\{-\xi_W |x|^{\gamma}\}, \quad \gamma \in (0,2], \: b \in {\mathbb R}_+,
\]
where $\Gamma(\cdot)$ denotes the gamma function. Note that $\beta_W$ depends on $\xi_W$ and $\gamma$, while $\xi_W$ depends only on $\gamma$.  We denote this distribution by $GN(\beta_W, \xi_W, \gamma).$
When $\gamma=1$, the density reduces to the symmetric Laplace distribution, denoted by \( \text{Lap}(0, b) \), and is given by 
\[
f_W(x) =  \frac{1}{2b} \exp \left( - \frac{|x|}{b} \right), \quad x \in \mathbb{R}.
\]
The class ${\cal C}$ of distributions we consider in our next proposition consists of all centered generalized normal distributions and all symmetric regularly varying distributions (as described by the class $RV_\alpha$ above), where $ \alpha > 0$. We now turn to our main sharp asymptotic results when the distribution of  ${\cal Z}_n$ and $\ep$  belongs to $\mathcal{C}$. In the rest of the paper, let $\psi_\infty=\left(\theta_x\sqrt{\Lambda_U''(\theta_x)}\right)^{-1}$.

\begin{thm}\label{thm:LDP_ndZ}
Assume that conditions {\rm(\ref{assumption1})-(\ref{assumption3})}, {\rm\ref{asp:L2}-\ref{asp:L3}} hold. There exists a sequence of constants $M_n \nearrow \ff$ and a collection of functions
$\{ \phi_n \}$ such that
\begin{eqnarray*}
 \lim_{n \ra \ff} n^{\frac{1}{2}}  e^{n \Lambda^*_U(x)} e^{n \phi_n(-\tilde M_n)} H_n(-\tilde M_n) \bP{L_n \geq n x} = \psi_{\ff},
\end{eqnarray*}
where $\exp(n \phi_n(-\tilde M_n))$ converges to a constant depending only on the distribution of $\ep$ and $\mathcal{Z}_n$, and $H_n(-\tilde M_n)$ diverges to infinity as $n \to \ff$.
\end{thm}

The proof of Theorem \ref{thm:LDP_ndZ} involves using sharp large deviations for conditional distributions of triangular arrays as in \cite{bahadur1960deviations}. The literature on sharp large deviations for conditional distributions is limited and involves using Berry-Esseen type bounds in the conditional central limit theorem on an exponential scale. For results that are in this general area, see \cite{bovier2014conditional} and \cite{liao2024geometric}.

We next turn to identifying $\{\tilde M_n\}$ and characterizing the exact growth rate of $\exp(n\phi_n(-\tilde M_n)) H_n(-\tilde M_n)$ as a function of $n$.
These depend on the distributions of $\ep$ and $\mathcal{Z}_n$, which belong to $\mathcal{C}$.  Note that the behavior of $\phi_n(\cdot)$ is studied in Proposition \ref{prop:CRateFunction}.

\begin{prop}\label{prop:LDP_ndZ}
Assume that conditions {\rm(\ref{assumption1})-(\ref{assumption3})}, {\rm\ref{asp:L2}-\ref{asp:L3}} hold and that the distributions of $\ep$ and $\mathcal{Z}_n$ belong to the class $\mathcal{C}$.

\begin{enumerate}
\item Let $\ep\sim GN(\beta_\epsilon,\xi_\epsilon,\gamma)$ and ${\mathcal{Z}_n}\sim GN(\beta_{\mathcal{Z}_n},\xi_{\mathcal{Z}_n},\gamma)$ and $\gamma\in(0,2]$. Assume $\beta_{\mathcal{Z}_n}\to\beta_{\mathcal{Z}}$ and $(\log n)|\xi_{\mathcal{Z}_n}-\xi_{\mathcal{Z}}|\to0$ as $n\to\ff$. Also, set $c_{\gamma}=\frac{b^\gamma \xi_{\mathcal{Z}}}{\xi_\ep}$,
$\Delta_{\gamma}=c_{\gamma}\left(1+\log\frac{C_x \beta_\ep}{\xi_{\cal{Z}} \gamma b}\right)-\log \beta_{\cal{Z}}$, $C_x=\frac{\lambda_U(\theta_x)-1}{\lambda_U(\theta_x)}$, and
\begin{align*}
    \eta_\gamma=\begin{cases}
        1 & \gamma\in(0,2)\\
        \exp\{-v^2\xi_{\cal{Z}}\} & \gamma=2.
    \end{cases}
\end{align*}
Then choosing $\tilde M_n = b^{(n)} (\frac{\log n}{\xi_\ep})^{\gamma^{-1}}(1+o(1))$
\begin{align*}
\lim_{n\to\ff}\frac{R_{1n}R_{2n}R_{3n}}{e^{n \phi_n(-\tilde M_n)}H_n(-\tilde M_n)} = K_{\gamma} e^{-\Delta_{\gamma}} \eta_{\gamma},
\end{align*}
where $R_{1n}=n^{c_{\gamma}}$, $\log R_{2n}=(\gamma-1)\gamma^{-1}(1-c_{\gamma}) \log(\log n)$, $\log R_{3n}= -v\gamma c_{\gamma} b^{-1} \xi_{\ep}^{\gamma^{-1}} (\log n)^{(\gamma-1)\gamma^{-1}}$.
The constant $K_{\gamma}$ is given by
\begin{align*}
K_{\gamma}=\left(\frac{b^{\gamma}}{\xi_\ep}\right)^{-\frac{(1-\gamma)c_{\gamma}}{\gamma}} \frac{1}{\gamma}\sqrt{\frac{b^{\gamma}}{\xi_\ep \xi_{\cal Z}}} b^{1-\gamma} \xi_\ep^{(\gamma-1)\gamma^{-1}}. 
\end{align*}

\item If $\ep \sim RV_{\alpha_\ep}$, $\mathcal{Z}_n \sim RV_{\alpha_{\mathcal{Z}_n}}$, and $(\log n) |\alpha_{\mathcal{Z}_n}-\alpha_{\mathcal{Z}}|\to0$ and $L_{\mathcal{Z}_n}(n^{\frac{1}{\alpha_\ep}}) \left[L_{\mathcal{Z}}(n^{\frac{1}{\alpha_\ep}})\right]^{-1} \to1$ as $n\to\ff$, then $\tilde M_n = n^{\frac{1}{\alpha_\ep}}(1+o(1))$ and
\begin{eqnarray*}
\lim_{n\to\ff} \frac{n^{\frac{\alpha_{\mathcal{Z}}}{\alpha_\epsilon}} [L_\epsilon(n^{\frac{1}{\alpha_\epsilon}})]^{\frac{\alpha_{\mathcal{Z}}}{\alpha_\epsilon}} [L_{\mathcal{Z}}(n^{\frac{1}{\alpha_\epsilon}})]^{-1}}{e^{n \phi_n(-\tilde M_n)}H_n(-\tilde M_n)} = K_x.
\end{eqnarray*}
In this case, the constant $K_x$ is given by $K_x = e^{\Delta} (\alpha_\epsilon \alpha_{\mathcal{Z}})^{-\frac{1}{2}}$,
where $\Delta = \frac{\alpha_{\mathcal{Z}}}{\alpha_\epsilon}\log\frac{C_x\alpha_\epsilon b^{\alpha_\epsilon}}{\alpha_{\mathcal{Z}}} + \log\alpha_{\mathcal{Z}} - \frac{\alpha_{\mathcal{Z}}}{\alpha_\epsilon}$.
\end{enumerate}
The specific formula of $\tilde M_n$ is given in the proof.
\end{prop}   

\begin{rem} In the above proposition, one can set $\gamma=1$ in the generalized normal distribution to obtain the results for the symmetric Laplace distribution. In this case, a simplification occurs, and the rate is $n^{c_1}$. It is also possible that the distributions of $\mathcal{Z}_n$ and $\ep$ are different. As an illustration, taking $\ep \sim N(0, \sigma^2_{\ep})$
and $\mathcal{Z}_n \sim RV_{\alpha_{\mathcal{Z}_n}}$, and assuming $(\log\log n)|\alpha_{\mathcal{Z}_n}-\alpha_{\mathcal{Z}}|\to0$ and $L_{\mathcal{Z}_n}(\sqrt{\log n})\left[L_{\mathcal{Z}}(\sqrt{\log n})\right]^{-1}\to1$ as $n\to\ff$, then $\tilde M_n= \sqrt{2}b^{(n)}\sigma_\epsilon\sqrt{\log n}(1+o(1))=b^{(n)} \left(\xi_\ep^{-1} \log n\right)^{\frac{1}{2}}(1+o(1))$, and
\begin{eqnarray*}
\lim_{n\to\ff} \frac{e^{n \phi_n(-\tilde M_n)}H_n(-\tilde M_n)}{(\log n)^{\frac{(\alpha_{\mathcal{Z}}+1)}{2}} \left[L_{\mathcal{Z}}\left(\sqrt{\log n}\right)\right]^{-1}} = K_x.
\end{eqnarray*}
The constant $K_x$ in this case is given by $K_x = \alpha_{\mathcal{Z}} (2b^2\sigma_\epsilon^2)^{-\frac{(\alpha_{\mathcal{Z}}+1)}{2}} b\sigma_\epsilon(\alpha_{\mathcal{Z}}+1)^{-\frac{1}{2}}$.
\end{rem}
\begin{rem}[Multi-factor regularly varying tails]
Suppose that the idiosyncratic factors $\epsilon$ and the common factors $Z_j^{(n)}$ both have regularly varying tails with indices $\alpha_\epsilon$ and $ \alpha_Z$, respectively, where $ \alpha_\epsilon$ and $ \alpha_Z$ are Positive, and $ L_\epsilon$ and $ L_Z$ are slowly varying functions. If $Z_j^{(n)}$ are i.i.d.\ with distribution $\mathrm{RV}_{\alpha_Z}$ and weights $\{a^{(n)}_j\}$ satisfy $\sum_{j=1}^{k_n} (a^{(n)}_j)^2+(b^{(n)})^2=1$, then
\[
\mathbf{P}\left(\sum_{j=1}^{k_n} a^{(n)}_j Z_j^{(n)} \le -w\right) \sim c_{\mathcal{Z}}^{\text{eff}}\,w^{-\alpha_{\mathcal{Z}}}L_{\mathcal{Z}}(w),
\qquad 
c_{\mathcal{Z}}^{\text{eff}}:=c_Z\sum_{j=1}^{k_n} |a^{(n)}_j|^{\alpha_Z}.
\]
Following the scaling in Theorem~\ref{thm:LDP_ndZ}, $M_n=n^{1/\alpha_\epsilon}(1+o(1))$, so the sharp asymptotics remain valid with the effective constant $c_{\mathcal{Z}}^{\text{eff}}$ in place of $c_{{Z}}$. 
Thus, the high-dimensional structure enters only through the geometry of the weights.
\end{rem}

Next, we turn to the case when $\kappa >-\ff$. In this situation, the rate depends on the support of $\mathcal{Z}$. Specifically, when the $\mathcal{Z}$ is non-degenerate, the prefactor governing the large deviations is of order $n^{-3/2}$ while when $\mathcal{Z}$ is degenerate at $\kappa$, the $n^{-1/2}$ prefactor predictably reemerges. This is the content of our next Theorem, which does not require membership in $\mathcal{C}$. We need a few additional notations. Let $p_{\kappa}= F_{\ep}(\frac{v-\kappa}{b})$, $q_{\kappa}=\mu_U p_{\kappa}$, and $\psi_\ff(z)=[\bar\theta_{x,z}]^{-1} [\Lambda''(\bar\theta_{x,z};z)]^{-\frac{1}{2}}$, where $\bar\theta_{x,z}$ is the root of $\Lambda'(\theta;z)=x$, $\Lambda(\theta;z) = \log\bE{\exp(\theta UX)} = \log \left[\lambda_{U}(\theta) F_{\epsilon}\left(\frac{v-z}{b}\right) + 1- F_{\epsilon}\left(\frac{v-z}{b}\right)\right]$, and $\Lambda^*(x;z)=\sup_{\theta}[\theta x - \Lambda(\theta;z)]$. The derivatives of $\Lambda(\theta;z)$ are taken with respect to $\theta$.
When $\mathcal{Z}_n$ converges to $\kappa$, the tilts are defined via the generating functions $\Lambda_n(\theta;\kappa) \coloneqq \log\{\la_{U^{(n)}}(\ta)p_{\kappa}+1-p_{\kappa}\}$, and $\Lambda(\theta;\kappa) \coloneqq \log\{\la_{U}(\ta)p_{\kappa}+1-p_{\kappa}\}$. 

\begin{thm} \label{thm:LDP_kappa}
Suppose the supports of $\mathcal{Z}_n$ and $\mathcal{Z}$ are $[z_0,\ff)$. Assume that conditions (\ref{assumption1})-(\ref{assumption3}), {\ref{asp:L2}-\ref{asp:L3}} hold. Then, for any $x > \mu_U$
\begin{eqnarray*}
    \lim_{n \ra \ff} n^{\frac{3}{2}} e^{n \Lambda^*(x;z_0)}  \bP{L_n \geq n x} = \frac{C_{z_0}\psi_\ff(z_0)}{\sqrt{2\pi}},
\end{eqnarray*}
where $C_{z_0}=\frac{f_{\mathcal{Z}}(z_0)}{(\Lambda^*)'(x;z_0)}\in(0,\ff)$, and $(\Lambda^*)'(x;z)$ is derivative of $\Lambda^*(x;z)$ with respect to $z$. If $\mathcal{Z}$ is degenerate at $\kappa$, additionally, assume the following \emph{localization conditions} (i) $\bP{n|\mathcal{Z}_{n}-\kappa|>\eta}\to 0$ as $n\to\ff$ for any fixed $\eta\in(0,\infty)$, 
(ii) for any $x>  q_{\kappa}$, there exists a unique $\theta_{x,n} \in (0, \theta_0)$ satisfying
$\Lambda_{n}'(\theta_{x,n};\kappa)=x$, and (iii) $\Lambda'(\theta;\kappa)=x$ has unique root, $\theta_x \in (0, \theta_0)$ hold. Then, for any $x > q_{\kappa}$
\begin{eqnarray*}
    \lim_{n \ra \ff} n^{\frac{1}{2}} e^{n \Lambda^*(x;\kappa)}  \bP{L_n \geq n x} = \frac{\psi_\ff(\kappa)}{\sqrt{2\pi}}.
\end{eqnarray*}
\end{thm}
\begin{rem}
When $\mathcal{Z}=\kappa$, and $n=o(k_n)$, then a sufficient condition for (i) to hold is that $\bP{|\mathcal{Z}_n-\kappa|> \eta} \le \exp(-k_n c_{\eta})$ for some positive constant $c_{\eta}$. Alternatively, if $\{Z_j^{(n)}, j \ge 1, n \ge 1\}$ are i.i.d. and $a_j^{(n)} \sim k_n^{-1/2}$  and $n=o(\sqrt{k_n})$, then (i) prevails by the Marcinkiewicz–-Zygmund law of large numbers (\cite{chow2012probability}).
\end{rem}
\textbf{General tail framework beyond class $\mathcal{C}$.}
We now describe a structural set of conditions on the right tail of $\ep$ and the left tail of $\mathcal{Z}$ under which the sharp large deviation bounds of Theorems~\ref{thm:LDP_ndZ} and \ref{thm:LDP_kappa} continue to hold. Let $C^2(\Real)$ denote the class of functions $f:\Real \to\Real$ that are twice continuously differentiable. Also, $f(\cdot)$ is said to be self-neglecting if $(f(x))^{-1}f(x+o(f(x))$ converges to one as $x \to \ff$. Some useful facts about self-neglecting functions are included in Appendix \ref{app:D}. Throughout the rest of this section, set $g_\ep(t):=1-F_\ep(t)$, write the left tail of the density of $\mathcal{Z}_n$, for large $w >0$ as
$f_{\mathcal{Z}_n}(-w)=c_{\mathcal{Z}_n}(w)\,e^{-R_{\mathcal{Z}_n}(w)}$, where $R_{\mathcal{Z}_n}(\cdot) \in C^2(\Real)$ and set
$r_{\mathcal{Z}_n}(w):=R'_{\mathcal{Z}_n}(w)$. Also, recall from Proposition \ref{prop:CMGF} and \ref{prop:theta_nx} that 
\( C_x := \lim_{n\to\ff} C_n(\theta_{x,n}) = \dfrac{(\lambda_U(\theta_x)-1)}{\lambda_U(\theta_x)}. \)

\begin{assumptionX}\label{asp:D1}
(Log-smooth tail condition for $\epsilon$ (Gumbel-type)). There exist $Q_\ep,c_\ep\in C^2(\R)$ such that, for all sufficiently large $t$,
\[
g_\ep(t)=c_\ep(t)\,e^{-Q_\ep(t)},\qquad Q'_\ep(t)\uparrow\ff,\qquad
c_\ep(t):=\frac{1}{Q'_\ep(t)}\ \text{satisfies}\ \frac{c_\ep\big(t+o(c_\ep(t))\big)}{c_\ep(t)}\to 1,
\]
and $Q''_\ep(t)=o\big(Q'_\ep(t)^2\big)$ as $t\to\ff$.
\end{assumptionX}

\begin{assumptionX}\label{asp:D2}
(Log-smooth left-tail condition for $\mathcal{Z}_n$). There exists $R_{\mathcal{Z}_n}\in C^2(\R)$ and a slowly varying $c_{\mathcal{Z}_n}$ such that, for all sufficiently large $w>0$,
\[
f_{\mathcal{Z}_n}(-w)=c_{\mathcal{Z}_n}(w)\,e^{-R_{\mathcal{Z}_n}(w)},\qquad
r_{\mathcal{Z}_n}(w):=R'_{\mathcal{Z}_n}(w)\uparrow\ff,\qquad R_{\mathcal{Z}_n}''(w)=o\big(r_{\mathcal{Z}_n}(w)^2\big),
\]
and $c_{\mathcal{Z}_n}$ varies slowly on windows of width $o\big(1/r_{\mathcal{Z}_n}(w)\big)$.
\end{assumptionX}

We present several examples covered by the conditions \ref{asp:D1}–\ref{asp:D2} and relationship with Gumbel maximum domain of attraction in Appendix \ref{app:D}. We need additional conditions that allow a large deviation analysis of $L_n$ under log-smooth assumptions. 

\begin{assumptionB}\label{asp:B1}
(Balance condition). There exist a sequence $M_n \in (a_n, b_n)$, where $a_n, b_n \to \ff$, such that
\[
R_{1n}\ \coloneqq \frac{\frac{n\,C_x}{b^{(n)}}\,f_\epsilon\big(\frac{v+M_n}{b^{(n)}}\big)}{r_{\mathcal{Z}_n}(M_n)}  \to 1 \quad \text{and} \quad
R_{2n}\ \coloneqq \frac{nC_x f_{\ep}(\frac{v+M_n}{b^{(n)}})Q_{\ep}'(\frac{v+M_n}{b^{(n)}})}{(b^{(n)}r_{\mathcal{Z}_n}(M_n))^2} =\Theta(1).
\]
\end{assumptionB}

\begin{assumptionW}\label{asp:W1}
(Window interiority). $\min\{M_n-a_n,\ b_n-M_n\}\cdot r_{\mathcal{Z}_n}(M_n)\to\ff$.
\end{assumptionW}

\begin{assumptionW}\label{asp:W2}
(Uniform conditional Bahadur-Rao on the bracket). There exists $ 0 <C<\infty$ such that 
\[
\sup_{z\in(-b_n,-a_n)}|r_n(x,z)|\le C n^{-1/2},\quad
0<c\le \theta_{x,n}(z)\,\sigma_{x,n}(z)\le C\ \ \forall z\in(-b_n,-a_n).
\]
\end{assumptionW}

\begin{assumptionW}\label{asp:W3}
(Local flatness and curvature on the Laplace window). Let $t_n:=(v+M_n)/b^{(n)}$. There exists $\eta>0$ such that, uniformly for $|z+M_n|\le \eta/r_{\mathcal{Z}_n}(M_n)$,
\[
\frac{r_{\mathcal{Z}_n}(-z)}{r_{\mathcal{Z}_n}(M_n)}=1+o(1),\quad
\frac{r_\epsilon((v-z)/b^{(n)})}{r_\epsilon(t_n)}=1+o(1),
\]
and slow/tilt multipliers vary by $1+o(1)$. Moreover, $\exists A_1,A_2>0$ such that
\[
-\,A_2\,n\,f_\epsilon(t_n)\,r_\epsilon(t_n)\ \le\ (n\,\widetilde h_n)''(z)\ \le\
-\,A_1\,n\,f_\epsilon(t_n)\,r_\epsilon(t_n)
\]
uniformly on the window.
\end{assumptionW}

\smallskip
We now state our main result under the general hypotheses when $\kappa=-\ff$.
\begin{thm}\label{thm:LDP_gen_ndZ}
Assume that the conditions {\rm(\ref{assumption1})-(\ref{assumption3})}, {\rm\ref{asp:L2}-\ref{asp:L3}} hold and $\kappa=-\ff$. Assume,  additionally that {\rm\ref{asp:D1}-\ref{asp:D2},\ref{asp:B1}, \ref{asp:W1}-\ref{asp:W3}} hold. Then for every $x>\mu_U$, there exist $\{\tilde M_n\}\uparrow\ff$ and functions $\{\phi_n\}$ and a scale $H_n(\cdot)$ (as in Theorem~\ref{thm:LDP_ndZ}) such that
\begin{eqnarray*}
 \lim_{n \ra \ff} n^{\frac{1}{2}}\,e^{n \Lambda^*_U(x)}\,e^{n \phi_n(-\tilde M_n)}\,H_n(-\tilde M_n)\ \bP{L_n \geq n x}
 \ =\ \psi_{\ff}.
\end{eqnarray*}
Under class $\mathcal{C}$ (GN/RV), the factor $e^{n \phi_n(-\tilde M_n)}H_n(-\tilde M_n)$ reduces to the explicit expressions in Proposition~\ref{prop:LDP_ndZ}.
\end{thm}

\begin{rem}[Scope and reduction to $\mathcal{C}$]
When \textup{\ref{asp:D1}}–\textup{\ref{asp:D2}} specialize to the generalized-normal classes, Theorems~\ref{thm:LDP_gen_ndZ}
recover Theorems~\ref{thm:LDP_ndZ}
, and $e^{n \phi_n(-\tilde M_n)}H_n(-\tilde M_n)$ simplifies as in Proposition~\ref{prop:LDP_ndZ}. Similar detailed results for regularly varying and Laplace cases are also available in Proposition~\ref{prop:LDP_ndZ}.
\end{rem}

\subsection{Gibbs conditioning in total variation}
In this section, we derive the conditional distribution of $U_1^{(n)}$ and $X_1^{(n)}$ given $n^{-1}L_n \geq x$, for $x > \mu_U.$ We assume that for each $n\geq1$, $\{U_j^{(n)}: j \ge 1\}$ are i.i.d. random variables.
\begin{thm}\label{thm:gibbs}
Assume that the conditions {\rm(\ref{assumption1})-(\ref{assumption3})}, {\rm\ref{asp:L2}-\ref{asp:L3}} 
hold.
\begin{enumerate}
\item \label{gibbs:1} Let $\mathcal{Z}$ be non-degenerate with support $\mathbb{R}$ ($\kappa=-\infty$, unbounded to the left). 
Then for any $x>\mu_U$,
\[
\sup_{B\in\mathcal{B}(\mathbb{R})}
\Big|
\mathbf{P}\big(U^{(n)}_1\in B,\,X^{(n)}_1=1\,\big|\,L_n\ge n x\big)
- \mathbf{P}_{\theta_x}(B)
\Big|\;\longrightarrow\;0,
\]
where $\mathbf{P}_{\theta_x}(B):=\lambda_U(\theta_x)^{-1}\int_B e^{\theta_x y}\,d\mathbf{P}_U(y)$ and $\mathbf{P}_U$ is the law of $U$.

\item \label{gibbs:2} For each fixed $k\ge 1$, the conditional law of $(U^{(n)}_j,X^{(n)}_j)_{j=1}^k$ given $\{L_n\ge n x\}$ 
converges in total variation to the product of $k$ i.i.d.\ copies of $(U,X)$, with $X\equiv 1$ and $U\sim \mathbf{P}_{\theta_x}$; i.e.
\[
\left\|\,
\mathcal{L}\left((U^{(n)}_1,X^{(n)}_1),\ldots,(U^{(n)}_k,X^{(n)}_k)\,\middle|\,L_n\ge n x\right)
- (\mathbf{P}_{\theta_x}\otimes\delta_1)^{\otimes k}
\,\right\|_{\mathrm{TV}}
\longrightarrow 0.
\]
Here $\delta_1$ denotes the Dirac measure at the point 1.

\item Let $\mathcal{Z}$ be non-degenerate with support $[\kappa,\infty)$ and $\kappa>- \infty$, and set $p_\kappa:=F_\epsilon((v-\kappa)/b)\in(0,1)$.
Then the conclusions of parts (1)–(2) hold with the limiting law of $(U,X)$ given by
\[
\mathbf{P}(X=1)=p_\kappa,\quad \mathbf{P}(X=0)=1-p_\kappa,\quad U\,|\,\{X=1\}\sim \mathbf{P}_{\bar\theta_{x,\kappa}},\quad U\,|\,\{X=0\}\sim \mathbf{P}_U,
\]
where $\bar\theta_{x,\kappa}$ solves $\partial_\theta \Lambda(\theta;\kappa)=x$ and $\mathbf{P}_{\bar\theta_{x,\kappa}}(B)=\lambda_U(\bar\theta_{x,\kappa})^{-1}\int_B e^{\bar\theta_{x,\kappa}y}\,d\mathbf{P}_U(y)$ with $\partial_\theta$ denoting differentiation with respect to $\theta$.
In the degenerate boundary case $Z\equiv\kappa$, the same conclusion holds under the localization assumptions of Theorem \ref{thm:LDP_kappa}.
\end{enumerate}
\end{thm}
\noindent For a variant of the Theorem when $U$ may depend on $X$, see Appendix \ref{app:F}.
\begin{rem}
The above Theorem shows that conditioned on the large deviation event, the process evolves as if every obligor has defaulted, and the distribution of LGD evolves as a shifted distribution. In the non-degenerate case, the shift is independent of $\{\mathcal{Z}_n\}$ and hence $\mathcal{Z}$, whereas in the degenerate case, the shift depends on the limit $\mathcal{Z}=\kappa$. Additionally, even though $X_1^{(n)}, \cdots, X_k^{(n)}$ are not independent, however, conditioned on $L_n \geq nx$ they are asymptotically independent.
\end{rem}
\begin{rem}[Inheritance under log-smooth/boundary hypotheses]\label{rem:gibbs-beyond-C}
Under the additional hypotheses of Theorems \ref{thm:LDP_gen_ndZ} (log-smooth, $\kappa=-\ff$) 
the conclusions of Theorem \ref{thm:gibbs} continue to hold. The proof is identical, with the normalization/localization from Theorems \ref{thm:LDP_gen_ndZ}
replacing the class~$\mathcal{C}$ prefactor controls.
\end{rem}

\section{Preliminary results} \label{sec:preliminary}

This section collects the analytic tools used in the proofs of Theorems \ref{thm:LDP_ndZ}-\ref{thm:gibbs}. 
We establish convexity, monotonicity, and stability properties for the conditional cumulant generating function 
$\Lambda_n(\theta;z)$ and its Legendre transform, and we localize the tilting parameters $\theta_{x,n}(z)$ needed for Laplace evaluations. 
Two probabilistic inputs underlie the sharp prefactors: a \emph{conditional central limit theorem} and a \emph{conditional Bahadur-Rao} (BR) estimate for the triangular arrays. For completeness, the conditional CLT and the conditional BR bound are stated in Appendix \ref{app:C} (Proposition \ref{thm:CCLT} and Theorem \ref{thm:CBahadurRao}), and are invoked here after we verify the regularity requirements in Propositions \ref{prop:CMGF}-\ref{prop:CRateFunction}.

Next, we study the properties of $\Lambda_n(\theta;z)$ as a function of $\theta$ for every $z$ and $z$ for every $\theta$. We recall the definition of $\Lambda_n(\theta;z)$ is the logarithmic moment generating function of $U^{(n)}X^{(n)}$ conditioned on $\mathcal{Z}_n=z$. That is,
\begin{eqnarray*} 
\Lambda_n(\theta;z)=\log\left[ \lambda_{U^{(n)}}(\theta)F_{\epsilon}(\frac{v-z}{b^{(n)}}) + 1- F_{\epsilon}(\frac{v-z}{b^{(n)}}) \right].
\end{eqnarray*}
We define that $\Lambda(\theta;z)$ is the logarithmic moment generating function of $UX$ conditioned on $\mathcal{Z}=z$, that is
\begin{eqnarray*} 
\Lambda(\theta;z)=\log\left[ \lambda_{U}(\theta)F_{\epsilon}(\frac{v-z}{b}) + 1- F_{\epsilon}(\frac{v-z}{b}) \right].
\end{eqnarray*}

\begin{prop} \label{prop:CMGF} 
Under assumptions {\rm(\ref{assumption1})-(\ref{assumption3})}, and \rm{\ref{asp:L2}}, the following hold:
\begin{enumerate}
\item Fix $z \in \Real$. Then,
\begin{enumerate}
\item $\Lambda_n(\theta;z)$ is convex and differentiable in $\theta.$ 
\item The derivative, $\frac{\partial}{\partial \theta}\Lambda_n(\theta;z)$, is strictly increasing in $\theta$.
\item $\lim\limits_{\theta\to0}\frac{\partial}{\partial \theta}\Lambda_n(\theta;z) = F_{\epsilon}(\frac{v-z}{b^{(n)}})\mathbf{E}[U^{(n)}]$, $\lim\limits_{\theta\to\infty}\frac{\partial}{\partial \theta}\Lambda_n(\theta;z) = \lim\limits_{\theta\to\infty}\frac{\lambda_{U^{(n)}}'(\theta)}{\lambda_{U^{(n)}}(\theta)}$.
The limit of the ratio, as $\theta\to\ff$, will be infinite if $U^{(n)}$ is unbounded. Else, it will equal the upper bound of $U^{(n)}$.
\end{enumerate}
\item Fix $\theta\in[0,\theta_0]$. Then
\begin{enumerate}
\item $\Lambda_n(\theta;z)$ is differentiable in $z$. 
\item $\Lambda_n(\theta;z)$ is strictly decreasing in $z$.
\item $\lim\limits_{z\to\infty}\Lambda_n(\theta;z)=0$, $\lim\limits_{z\to-\infty}\Lambda_n(\theta;z)=\log \lambda_{U^{(n)}}(\theta) = \Lambda_{U^{(n)}}(\theta)>0$.
\end{enumerate}
\item For $z=-M_n$, where $M_n\to\infty$ as $n\to\infty$, write $g_n(M_n)=1-F_\epsilon(\frac{v+M_n}{b^{(n)}})$ and $C_n(\theta)=\dfrac{\lambda_{U^{(n)}}(\theta)-1}{\lambda_{U^{(n)}}(\theta)}$, then for large $n$,
    \begin{align*}
        \Lambda_n(\theta;-M_n)=\Lambda_{U^{(n)}}(\theta) - g_n(M_n)C_n(\theta)(1+o(1)).
    \end{align*}
\item For any fixed $z\in\Real$ and $\theta\in[0,\theta_0]$, $\lim\limits_{n\to\infty}\Lambda_n(\theta;z)=\Lambda(\theta;z)$.
\end{enumerate}
\end{prop}
Our next result is concerned with the path properties of $\theta_{x,n}(z)$ defined as a root of the equation
\[
\frac{\partial}{\partial \theta}\Lambda_n(\theta_{x,n}(z);z)=x.
\]
\begin{prop}
\label{prop:theta_nx}
Under the assumptions {\rm(\ref{assumption1})-(\ref{assumption3})}, and {\rm\ref{asp:L2}-\ref{asp:L3}}, the function $\theta_{x,n}(z)$ is differentiable in $z$, strictly increasing, and satisfies
\[
\lim_{z\to -\infty}\theta_{x,n}(z)=\theta_{x,n}>0,\quad \lim_{z\to \infty}\theta_{x,n}(z)=\infty,\quad\text{and}~ \lim_{n\to \infty}\theta_{x,n}(z) = \theta_{x}(z)
\]
where $\theta_{x,n}$ is the unique solution of $\Lambda_{U^{(n)}}'(\theta)=x$ and $\theta_{x}(z)$ is the unique solution of $\frac{\partial}{\partial\theta}\Lambda(\theta;z)=x$. Moreover, the derivative is given by:
\[
\frac{d\theta_{x,n}(z)}{dz}=-\frac{\frac{\partial^2}{\partial z\partial\theta}\Lambda_n(\theta;z)}{\frac{\partial^2}{\partial\theta^2}\Lambda_n(\theta;z)},~~\text{where}
\]
the RHS is evaluated at $\theta=\theta_{x,n}(z)$.
\end{prop}
An explicit formula for the above is provided in Proposition \ref{app:prop_lam_n} in Appendix  \ref{app:A}. We now turn to the properties of the conditional rate function $$\Lambda_n^*(x;z)=\sup_{\theta\in[0,\theta_0]}[\theta x- \Lambda_n(\theta; z)],$$ the Legendre-Fenchel transform of $\Lambda_n(\theta; z)$. 

\begin{prop} \label{prop:CRateFunction} 
Under the conditions {\rm(\ref{assumption1})-(\ref{assumption3})}, and {\rm\ref{asp:L2}-\ref{asp:L3}}, for each fixed \( n \ge 1 \) and each fixed \( x > 0 \) the following hold:
\begin{enumerate}
\item[(i)] The function \( z \mapsto \Lambda_n^*(x;z) \) is strictly increasing.
\item[(ii)] \(\lim_{z \to \infty}\Lambda_n^*(x;z)=\infty.\)
\item[(iii)] If \(\Lambda_U(\theta)=\log\mathbf{E}[e^{\theta U}]\) and its Fenchel-Legendre transform is \(\Lambda_U^*(x)\), then
\[
\lim_{z\to -\infty}\Lambda_n^*(x;z)=\Lambda_U^*(x)
\]
uniformly in \( n \).
\item[(iv)] Under additional condition {\rm\ref{asp:L2}}, then for any fixed $z$, as $n\to\ff$, $n|\Lambda^*_{U^{(n)}}(x)-\Lambda^*_{U}(x)|\to0$, $n|\theta_{x,n}-\theta_{x}|\to0$, and $n|\Lambda_n^*(x;z)-\Lambda^*(x;z)|\to0$.
\item[(v)] For $z=-M_n$, where $M_n\to\infty$ as $n\to\infty$, recalling $g_n(M_n)=1-F_\epsilon(\frac{v+M_n}{b^{(n)}})$ and $C_n(\theta)=\frac{\lambda_{U^{(n)}}(\theta)-1}{\lambda_{U^{(n)}}(\theta)}$, then as $n\to\ff$,
\begin{align*}
\phi_n(-M_n)=  \Lambda_n^*(x;-M_n)-\Lambda_{U^{(n)}}^*(x)=  g_n(M_n)C_n(\theta_{x,n})(1+o(1)),
\end{align*}
where $\theta_{x,n}=\arg\sup_\theta \{\theta x -\log \lambda_{U^{(n)}}(\theta)\}$.
\end{enumerate} 
\end{prop}

\section{Proofs}\label{sec:proof}
In this section we provide the proofs of the main results stated in Section \ref{sec:Main_results}. The proofs require additional lemmas whose proofs are provided in Appendix A.
\subsection{Proof of Theorem \ref{thm:LDP_ndZ} and Proposition \ref{prop:LDP_ndZ}} \label{sec:5.1}

Here and below, $C, C_1, C_2,\dots$ denote finite positive constants whose values may change from line to line.

\begin{proof}
First, by conditioning with respect to (w.r.t.) $\mathcal{Z}_n$ note that 
\begin{align*}
    \bP{L_n>nx} 
    =& \int_{-\infty}^{\infty} \bP{L_n>nx|\mathcal{Z}_n=z} f_{\mathcal{Z}_n} (z)\mathrm{d}z.
\end{align*}
Let $z_0 \in \Real$ be fixed (to be chosen later) and express $\bP{L_n \geq nx}$ as integral over the ranges $(-\ff, z_0)$ and $(z_0, \ff)$ and denote them as $T_{1n}$ and $T_{2n}$. We first study $T_{2n}$. Since $\bP{L_n\geq nx | \mathcal{Z}_n=z}$ is non-increasing in $z$ and $x > \mu_U$, using Cram\'er's large deviation upper bound and Proposition \ref{prop:CRateFunction} (iv) it follows that 
\begin{eqnarray*}
    T_{2n} &\leq& \bP{L_n\geq nx | \mathcal{Z}_n=z_0} \int_{z_0}^{\ff}f_{\mathcal{Z}_n}(z) \mathrm{d}z \leq C_1 e^{-n\Lambda_n^*(x;z_0)} \leq C_2 e^{-n\Lambda^*(x;z_0)}.
\end{eqnarray*}
Next, we study $T_{1n}$. To this end, applying conditional Bahadur-Rao Theorem (see Theorem \ref{thm:CBahadurRao} in the appendix), and setting $\psi_n(z)=[\theta_{x,n}(z)\sigma_{x,n}(z)]^{-1}$ and $\phi_n(z)=\Lambda^*_n(x;z)-\Lambda^*_{U^{(n)}}(x)$, it follows that
\begin{align*}
    T_{1n} =& \int_{-\infty}^{z_0} n^{-\frac{1}{2}}e^{-n \Lambda^*_{n}(x;z)}\frac{1}{\sqrt{2\pi}\theta_{x,n}(z)\sigma_{x,n}(z)}(1+r_n(z)) f_{\mathcal{Z}_n} (z)\mathrm{d}z\\
    =& e^{-n[\Lambda^*_{U^{(n)}}(x)-\Lambda^*_{U}(x)]} \frac{1}{\sqrt{2\pi}} n^{-\frac{1}{2}}e^{-n\Lambda^*_{U}(x)}\int_{-\infty}^{z_0} e^{-n \phi_n(z)}\psi_n(z)(1+r_n(z)) f_{\mathcal{Z}_n} (z)\mathrm{d}z.
\end{align*}
Since $0<\inf_{n,z\in(-\ff,z_0)}\psi_n(z)\leq \sup_{n,z\in(-\ff,z_0)}\psi_n(z)<\ff$ and $\sup_{n,z\in(-\ff,z_0)}|r_n(z)|\to0$ as $n\to\ff$ by Theorem \ref{thm:CBahadurRao}, it is sufficient to study the asymptotic behavior of
\begin{align*}
    \int_{-\infty}^{z_0} e^{-n \phi_n(z)}\psi_n(z) f_{\mathcal{Z}_n} (z)\mathrm{d}z = \int_{-\infty}^{z_0} e^{n \tilde h_n(z)}\psi_n(z) \mathrm{d}z,
\end{align*}
where $\tilde h_n(z)=-\phi_n(z)+ \frac{1}{n}\log f_{\mathcal{Z}_n}(z)$. We first show that there exists a unique $\{\tilde M_n\}$ diverging to infinity satisfying $\tilde h_n'(-\tilde M_n)=0$. Fix $z_0$ such that $\tilde h_n'(z_0)<0$. Notice that from properties of $\phi_n(\cdot)$ in Proposition \ref{prop:CRateFunction} and $\ep,\mathcal{Z}_n\in\mathcal{C}$, there exists $\delta_n$ such that $\delta_n\to\ff$ and $\tilde h_n''(z)<0$ for $z\in(-\delta_n,z_0)$ and $\tilde h_n'(z)>0$ for $z\in(-\ff,-\delta_n)$. Hence there is a unique solution to $\tilde h_n'(z)=0$ on the interval $(-\ff,z_0)$. 

{\bf Generalized Normal Case:} We begin with the proof when $\ep\sim GN(\beta_\epsilon,\xi_\epsilon,\gamma)$ and ${\mathcal{Z}_n}\sim GN(\beta_{\mathcal{Z}_n},\xi_{\mathcal{Z}_n},\gamma)$ and $\gamma\in(0,2]$. Notice that  $f_{\epsilon}(z)= \beta_\epsilon e^{-\xi_\ep |z|^{\gamma}}$ and as $ n \to \ff$, $f_{\mathcal{Z}_n}(z) = \beta_{\mathcal{Z}_n} e^{-\xi_{\mathcal{Z}_n} |z|^{\gamma}}$ converges to $f_{\mathcal{Z}}(z) = \beta_{\mathcal{Z}} e^{-\xi_{\mathcal{Z}} |z|^{\gamma}}$, where $\gamma\in(0,2]$. In this case, as $z\to\ff$, $1-F_{\epsilon}(z) = \frac{\beta_\ep}{\gamma\xi_\ep}z^{1-\gamma}e^{-\xi_\ep z^{\gamma}} (1+o(1))$. By standard calculation, it is easy to see that $\tilde M_n = M_n(1+o(1))$ where $M_n=b^{(n)}\left(\frac{\log n}{\xi_\ep}\right)^{\frac{1}{\gamma}}$. Fix $\beta\in(0,1)$ and decompose the integral into three parts, we obtain
\begin{align} \label{eq:gen_1}
     \int_{-\infty}^{z_0} e^{n \tilde h_n(z)}\psi_n(z) \mathrm{d}z =& \left(\int_{-\ff}^{-M_n(1+\beta)} + \int_{-M_n(1+\beta)}^{-M_n(1-\beta)} + \int_{-M_n(1-\beta)}^{z_0}\right) e^{n \tilde h_n(z)}\psi_n(z) \mathrm{d}z = J_{1n} + J_{2n} + J_{3n}.
\end{align}
We begin with the study of $J_{2n}$; notice that, by setting $z=-M_n t$ it reduces to 
\begin{align*}
    J_{2n} = M_n \int_{1-\beta}^{1+\beta} e^{n h_n(t)} \psi_n(-t M_n) \mathrm{d}t,
\end{align*}
where $h_n(t)=\tilde h_n(-t M_n)$ and its formula is provided in Lemma \ref{pf:lem_Gen_1}, wherein one uses Proposition \ref{prop:CRateFunction} (v). Let  $t_{0,n}$ be the root of $h_n^{\prime}(t)=0$, identified in Lemma  \ref{pf:lem_Gen_1}. Notice that $\tilde M_n=t_{0,n}M_n$. In fact, we will show in that Lemma that $t_{0,n}=1+o(1)$, as $n \ra \ff$. Now, applying the Laplace method (see \cite{olver1997asymptotics,fukuda2025higher}) and that $\psi_n(-t_{0,n} M_n)\to\psi_\ff=\theta_x\sqrt{\Lambda_U''(\theta_x)}$, it follows that
\begin{align} \label{eq:gen_2}
\lim_{n\to\infty}\frac{ J_{2n}}{M_n e^{n h_n(t_{0,n})}\sqrt{\frac{2\pi}{-nh_n''(t_{0,n})}}}=\psi_{\ff}.
\end{align}
Observe that the denominator is the same as $\sqrt{2\pi} \exp(-n \phi_n(-\tilde{M}_n) [H_n(-\tilde{M}_n)]^{-1}$, where
\begin{align*}
    [H_n(-\tilde{M}_n)]^{-1} = M_n f_{\mathcal{Z}_n}(-\tilde M_n)[n|h_n''(t_{0,n})|]^{-\frac{1}{2}} = f_{\mathcal{Z}_n}(-\tilde M_n)[n |\tilde h_n''(-\tilde M_n)|]^{-\frac{1}{2}}.
\end{align*}
We notice here that $\psi_n(-t_{0,n} M_n)\to\psi_\ff=\left(\theta_x\sqrt{\Lambda_U''(\theta_x)}\right)^{-1}$. In Lemma \ref{pf:lem_Gen_2} below, we will establish the precise behavior ({\it i.e.}in $n$) of the denominator. Specifically, we will establish that
\begin{align*}
    \lim_{n\to\ff} R_{1n}R_{2n}R_{3n}\exp(-n \phi_n(-\tilde{M}_n) [H_n(-\tilde{M}_n)]^{-1} = K_{\gamma} e^{-\Delta_{\gamma}}\eta_{\gamma},
\end{align*}
where $R_{1n}=n^{c_{\gamma}}$, $\log R_{2n}=(\gamma-1)\gamma^{-1}(1-c_{\gamma}) \log(\log n)$, $\log R_{3n}= -v\gamma c_{\gamma} b^{-1} \xi_{\ep}^{\gamma^{-1}} (\log n)^{(\gamma-1)\gamma^{-1}}$, $c_{\gamma}=\frac{b^\gamma \xi_{\mathcal{Z}}}{\xi_\ep}$, $\Delta_{\gamma}=c_{\gamma}\left(1+\log\frac{C_x \beta_\ep}{\xi_{\cal{Z}} \gamma b}\right)-\log \beta_{\cal{Z}}$, $C_x = \frac{\lambda_U(\theta_x)-1}{\lambda_U(\theta_x)}$, and $\eta_{2}=\exp\{-v^2\xi_{\mathcal{Z}}\}$ and $\eta_{\gamma}=1$ for $\gamma\in(0,2)$. The constant $K_{\gamma}$ is given by
\begin{align*}
K_{\gamma}=\left(\frac{b^{\gamma}}{\xi_\ep}\right)^{-\frac{(1-\gamma)c_{\gamma}}{\gamma}} \frac{1}{\gamma}\sqrt{\frac{b^{\gamma}}{\xi_\ep \xi_{\cal Z}}} b^{1-\gamma} \xi_\ep^{(\gamma-1)\gamma^{-1}}. 
\end{align*}
Turning to $J_{1n}$ and $J_{3n}$, we will show in Lemma \ref{pf:lem_Gen_3} that $J_{2n}^{-1}(J_{1n}+J_{3n})$ converges to 0. Combining this with equation (\ref{eq:gen_1}) and (\ref{eq:gen_2}), the theorem for the generalized normal distribution follows.

{\bf Regularly varying case:} We turn to the case $\mathcal{Z}_n \sim RV_{\alpha_{\mathcal{Z}_n}}$ and $\ep\sim RV_{\alpha_\ep}$. Recall the notation implies they have symmetric regularly varying tails; that is, as $|z|\to\infty$, $f_{\epsilon}(z) \sim \alpha_\epsilon |z|^{-\alpha_\epsilon-1}L_\epsilon(z)$ and $f_{\mathcal{Z}_n}(z)\sim \alpha_{\mathcal{Z}_n} |z|^{-\alpha_{\mathcal{Z}_n}-1}L_{\mathcal{Z}_n}(z)$. Also, as $z\to\infty$, $1-F_{\epsilon}(z)\sim |z|^{-\alpha_\epsilon}L_\epsilon(z)$, $1-F_{\mathcal{Z}_n}(z)\sim |z|^{-\alpha_{\mathcal{Z}_n}}L_{\mathcal{Z}_n}(z)$; for $z\to-\infty$, $F_{\epsilon}(z)\sim |z|^{-\alpha_\epsilon}L_\epsilon(z)$, $F_{\mathcal{Z}_n}(z)\sim |z|^{-\alpha_{\mathcal{Z}_n}}L_{\mathcal{Z}_n}(z)$. By standard calculation, we observe that $\tilde M_n= M_n^{1+o(1)}$ where $M_n=n^{\frac{1}{\alpha_\epsilon}}$. Next, fix $\beta\in(0,1)$ and, as before, decompose the integral into three parts, namely,
\begin{align} \label{eq:rv_1}
     \int_{-\infty}^{z_0} e^{n \tilde h_n(z)}\psi_n(z) \mathrm{d}z =& \left(\int_{-\ff}^{-M_n n^{\beta}} + \int_{-M_n n^{\beta}}^{-M_n n^{-\beta}} + \int_{-M_n n^{-\beta}}^{z_0}\right) e^{n \tilde h_n(z)}\psi_n(z) \mathrm{d}z = J_{1n}' + J_{2n}' + J_{3n}'.
\end{align}
We begin with the study of $J_{2n}$, Notice that, by setting $z=-M_n^{t}$ it reduces to 
\begin{align*}
    J_{2n}' = \log M_n \int_{1-\beta\alpha_\epsilon}^{1+\beta\alpha_\epsilon} e^{n h_n(t)} \psi_n(-M_n^{t}) \mathrm{d}t,
\end{align*}
where $h_n(t)=\tilde h_n(-M_n^t)+\frac{t\log M_n}{n}$ and its formula is provided in Lemma \ref{pf:lem_RR_1}, wherein one uses Proposition \ref{prop:CRateFunction} (v). Let  $t_{0,n}$ be the root of $h_n^{\prime}(t)=0$, identified in Lemma  \ref{pf:lem_RR_1}. Notice that $\tilde M_n=M_n^{t_0,n}$. In fact, we will show in that Lemma that $t_{0,n}=1+o(1)$, as $n \ra \ff$. Now, applying the Laplace method and that $\psi_n(- M_n^{t_{0,n}})\to\psi_\ff=\theta_x\sqrt{\Lambda_U''(\theta_x)}$, it follows that
\begin{align} \label{eq:rv_2}
\lim_{n\to\infty}\frac{ J_{2n}'}{(\log M_n) e^{n h_n(t_{0,n})}\sqrt{\frac{2\pi}{-nh_n''(t_{0,n})}}}=\psi_{\ff}.
\end{align}
Observe that the denominator is the same as $\sqrt{2\pi} \exp(-n \phi_n(-\tilde{M}_n) [H_n(-\tilde{M}_n)]^{-1}$, where
\begin{align*}
    [H_n(-\tilde{M}_n)]^{-1} = (\log M_n) f_{\mathcal{Z}_n}(-\tilde M_n) M_n^{t_{0,n}} [n|h_n''(t_{0,n})|]^{-\frac{1}{2}} = f_{\mathcal{Z}_n}(-\tilde M_n)[n |\tilde h_n''(-\tilde M_n)|]^{-\frac{1}{2}}.
\end{align*}
We notice here that $\psi_n(-t_{0,n} M_n)\to\psi_\ff=\left(\theta_x\sqrt{\Lambda_U''(\theta_x)}\right)^{-1}$. In Lemma \ref{pf:lem_RR_2} below, we will establish that
\begin{align*}
    \lim_{n\to\ff} n^{\frac{\alpha_{\mathcal{Z}}}{\alpha_\epsilon}} [L_\epsilon(n^{\frac{1}{\alpha_\epsilon}})]^{\frac{\alpha_{\mathcal{Z}}}{\alpha_\epsilon}} [L_{\mathcal{Z}}(n^{\frac{1}{\alpha_\epsilon}})]^{-1} \exp(-n \phi_n(-\tilde{M}_n) [H_n(-\tilde{M}_n)]^{-1} = K_x,
\end{align*}
where $K_x = e^{\Delta} (\alpha_\epsilon \alpha_{\mathcal{Z}})^{-\frac{1}{2}}$, $\Delta = \frac{\alpha_{\mathcal{Z}}}{\alpha_\epsilon}\log\frac{C_x\alpha_\epsilon b^{\alpha_\epsilon}}{\alpha_{\mathcal{Z}}} + \log\alpha_{\mathcal{Z}} - \frac{\alpha_{\mathcal{Z}}}{\alpha_\epsilon}$, and $C_x = \frac{\lambda_U(\theta_x)-1}{\lambda_U(\theta_x)}$.
Turning to $J_{1n}'$ and $J_{3n}'$, we will show in Lemma \ref{pf:lem_RR_3} that $(J_{2n}')^{-1}(J_{1n}'+J_{3n}')$ converges to 0. Combining this with equation (\ref{eq:rv_1}) and (\ref{eq:rv_2}), the theorem for the regularly varying tail follows.

\end{proof}

Next, we state Lemmas \ref{pf:lem_Gen_1}-\ref{pf:lem_RR_3} and provide their proofs in Appendix \ref{app:lem_Gen_1}-\ref{app:lem_RR_3}, respectively. Our first Lemma provides a formula for $h_n(\cdot)$ and describes the asymptotic behavior of the root of $h_n'(t)=0$.

\begin{lem} \label{pf:lem_Gen_1}
Under the conditions of Theorem \ref{thm:LDP_ndZ} and Proposition \ref{prop:LDP_ndZ}
$$h_n(t)=\frac{1}{n}\left[ - n C_n(\theta_{x,n})\cdot \frac{\beta_\ep}{\gamma\xi_\ep}(\frac{v+t M_n}{b^{(n)}})^{1-\gamma}e^{-\xi_\ep (\frac{v + t M_n}{b^{(n)}})^{\gamma}} (1+o(1)) + (\log \beta_{\mathcal{Z}_n} - \xi_{\mathcal{Z}_n} (t M_n)^{\gamma})(1+o(1)) \right],$$ 
where $M_n=b^{(n)}\left(\frac{\log n}{\xi_\ep}\right)^{\frac{1}{\gamma}}$ and $C_n(\theta_{x,n})$ is defined in Proposition \ref{prop:CRateFunction}. Also, the solution of $h_n'(t)=0$, namely $t_{0,n}$ satisfies 
\begin{align*}
t_{0,n}^{\gamma} = 1 + \frac{(1-\gamma)(b^{(n)})^{\gamma}}{\xi_\ep} \frac{\log M_n}{M_n^\gamma} -v\gamma t^{\gamma-1} \frac{1}{M_n} +
\frac{1}{M_n^\gamma}\left[ \frac{(b^{(n)})^{\gamma}}{\xi_\ep}\log\frac{C_n(\theta_{x,n}) \beta_\ep}{\xi_{\mathcal{Z}_n} \gamma b^{(n)}} -\frac{\log\eta_{\gamma,n}}{\xi_{\mathcal{Z}_n}} \right](1+o(1)) \to1,
\end{align*}
where $\eta_{2,n}=\exp\{-v^2\xi_{\mathcal{Z}_n}\}$ and $\eta_{\gamma,n}=1$ for $\gamma\in(0,2)$. Set $\eta_{2}=\lim\limits_{n\to\ff}\eta_{2,n}=\exp\{-v^2\xi_{\mathcal{Z}}\}$ and $\eta_{\gamma}=1$ for $\gamma\in(0,2)$. Then, as $n\to\infty$,
\begin{align*}
n h_n(t_{0,n})=& -\xi_{\mathcal{Z}_n} M_n^{\gamma} - \frac{(1-\gamma)(b^{(n)})^{\gamma}\xi_{\mathcal{Z}_n}}{\xi_\ep} \log M_n + v\gamma\xi_{\mathcal{Z}_n} M_n^{\gamma-1} - (\Delta_{\gamma} - \log\eta_\gamma) (1+o(1))\\
n h_n''(t_{0,n}) =& -\xi_\ep\xi_{\mathcal{Z}_n}(\gamma M_n)^2 M_n^{2(\gamma-1)} (\frac{1}{b^{(n)}})^{\gamma} (1+o(1)),
\end{align*}
where $\Delta_{\gamma}=\left(\frac{b^{\gamma}\xi_{\mathcal{Z}}}{\xi_\ep}\log\frac{C_x \beta_\ep}{\xi_{\mathcal{Z}} \gamma b} +\frac{\xi_{\mathcal{Z}}}{\xi_\ep}b^{\gamma}-\log \beta_{\mathcal{Z}} \right)$, and $C_x=\lim\limits_{n\to\ff}C_n(\theta_{x,n}) = \frac{\lambda_U(\theta_x)-1}{\lambda_U(\theta_x)}$.
\end{lem}

Our next Lemma provides a precise behavior of $\exp\{-n \phi_n(-\tilde{M}_n)\} [H_n(-\tilde{M}_n)]^{-1}$
as a function of $n$.
\begin{lem}\label{pf:lem_Gen_2}
Under conditions of Theorem \ref{thm:LDP_ndZ} and Proposition \ref{prop:LDP_ndZ}, $\ep\sim GN(\beta_\epsilon,\xi_\epsilon,\gamma)$ and ${\mathcal{Z}_n}\sim GN(\beta_{\mathcal{Z}_n},\xi_{\mathcal{Z}_n},\gamma)$ and $\gamma\in(0,2]$, setting $c_{\gamma}=\frac{b^\gamma \xi_{\mathcal{Z}}}{\xi_\ep}$, the following holds:
\begin{align}\label{eq:lem_gen_1}
    \lim_{n\to\ff} \frac{R_{1n}R_{2n}R_{3n}}{e^{n \phi_n(-\tilde{M}_n)}H_n(-\tilde{M}_n)}  = K_{\gamma} e^{-\Delta_{\gamma}}\eta_{\gamma},
\end{align}
where $R_{1n}=n^{c_{\gamma}}$, $\log R_{2n}=(\gamma-1)\gamma^{-1}(1-c_{\gamma}) \log(\log n)$, $\log R_{3n}= -v\gamma c_{\gamma} b^{-1} \xi_{\ep}^{\gamma^{-1}} (\log n)^{(\gamma-1)\gamma^{-1}}$, and $\eta_{2}=\exp\{-v^2\xi_{\mathcal{Z}}\}$, $\eta_{\gamma}=1$ for $\gamma\in(0,2)$ and $\Delta_{\gamma}$, $\eta_{\gamma}$ as in Lemma \ref{pf:lem_Gen_1}. The constant $K_{\gamma}$ is given by
\begin{align*}
K_{\gamma}=\left(\frac{b^{\gamma}}{\xi_\ep}\right)^{-\frac{(1-\gamma)c_{\gamma}}{\gamma}} \frac{1}{\gamma}\sqrt{\frac{b^{\gamma}}{\xi_\ep \xi_{\cal Z}}} b^{1-\gamma} \xi_\ep^{(\gamma-1)\gamma^{-1}}. 
\end{align*}
\end{lem}

Our next Lemma establishes that the large deviation probability is governed by the behavior of $J_{2n}$. 

\begin{lem}\label{pf:lem_Gen_3}
Let $R_{1n}$, $R_{2n}$, $R_{3n}$ be as in Lemma \ref{pf:lem_Gen_2} and $J_{1n}$, $J_{2n}$, $J_{3n}$ be as defined in equation (\ref{eq:gen_1}). Then,
\begin{align*}
    \lim_{n\to\ff} R_{1n}R_{2n}R_{3n}(J_{1n}+J_{3n}) =0.
\end{align*}
\end{lem}

Our next Lemma is an analogue of Lemma \ref{pf:lem_Gen_1} for the regularly varying tails.

\begin{lem}\label{pf:lem_RR_1}
Under conditions of Theorem \ref{thm:LDP_ndZ} and Proposition \ref{prop:LDP_ndZ}, and 
$$h_n(t)=\frac{1}{n}\left[-n C_n(\theta_{x,n})\left(\frac{v+ M_n^{t}}{b}\right)^{-\alpha_\epsilon} L_\epsilon(M_n^t)(1+o(1)) + (\log L_{\mathcal{Z}_n}(M_n^t) -\alpha_{\mathcal{Z}_n} t \log M_n + \log\alpha_{\mathcal{Z}_n})(1+o(1)) \right],$$ 
where $M_n = n^{{\frac{1}{\alpha_\epsilon}}}$, and $C_n(\theta_{x,n})$ is defined in Proposition \ref{prop:CRateFunction}, then the solution of $h_n'(t)=0$ is $t_{0,n}$ and as $n\to\infty$,
\begin{align*}
    t_{0,n} = 1+\frac{\log L_\epsilon(M_n)}{\alpha_\epsilon \log M_n} + \frac{\log\frac{C_n(\theta_{x,n})\alpha_\epsilon (b^{(n)})^{\alpha_\epsilon}}{\alpha_{\mathcal{Z}_n}}}{\alpha_\epsilon \log M_n} (1+o(1))  \to 1.
\end{align*}
Furthermore, as $n\to\infty$,
\begin{align*}
    n h_n(t_{0,n}) =& -\frac{\alpha_{\mathcal{Z}_n}}{\alpha_\epsilon}\log n + \log L_{\mathcal{Z}_n}(M_n) + \frac{\alpha_{\mathcal{Z}_n}}{\alpha_\epsilon}\log L_\epsilon(M_n) + \Delta +o(1) , \\
    n h_n''(t_{0,n}) =& -(\log M_n)^2 \alpha_\epsilon \alpha_{\mathcal{Z}_n} (1+o(1)) ,
\end{align*}
where $\Delta=\frac{\alpha_{\mathcal{Z}}}{\alpha_\epsilon}\log\frac{C_x\alpha_\epsilon b^{\alpha_\epsilon}}{\alpha_{\mathcal{Z}}} + \log\alpha_{\mathcal{Z}} - \frac{\alpha_{\mathcal{Z}}}{\alpha_\epsilon}$, and $C_x=\lim_{n\to\ff}C_n(\theta_{x,n}) = \frac{\lambda_U(\theta_x)-1}{\lambda_U(\theta_x)}$.
\end{lem}

Our next Lemma provides a precise asymptotic behavior of $\exp(-n \phi_n(-\tilde M_n)) H_n(-\tilde M_n)$ in terms of $n$. We observe here that, as in the generalized normal case, the dominant term is a power of $n$. A key difference is that the model parameter $b$ is lost in the rate, while it plays an important role in the generalized normal case.

\begin{lem}\label{pf:lem_RR_2}
Under conditions of Theorem \ref{thm:LDP_ndZ} and Proposition \ref{prop:LDP_ndZ}, $\ep \sim RV_{\alpha_\ep}$, $\mathcal{Z}_n \sim RV_{\alpha_{\mathcal{Z}_n}}$, then the following holds:
\begin{align*}
    \lim_{n\to\ff} n^{\frac{\alpha_{\mathcal{Z}}}{\alpha_\epsilon}} [L_\epsilon(n^{\frac{1}{\alpha_\epsilon}})]^{\frac{\alpha_{\mathcal{Z}}}{\alpha_\epsilon}} [L_{\mathcal{Z}}(n^{\frac{1}{\alpha_\epsilon}})]^{-1} \exp(-n \phi_n(-\tilde{M}_n) [H_n(-\tilde{M}_n)]^{-1} = K_x,
\end{align*}
where $K_x = e^{\Delta} (\alpha_\epsilon \alpha_{\mathcal{Z}})^{-\frac{1}{2}}$, $\Delta$ is defined in Lemma \ref{pf:lem_RR_1}.
\end{lem}

The next lemma shows that $J_{2n}'$ governs the large deviation probability.

\begin{lem}\label{pf:lem_RR_3}
Let $J_{1n}'$, $J_{2n}'$, $J_{3n}'$ be as defined in equation (\ref{eq:rv_1}). Then,
\begin{align*}
    \lim_{n\to\ff} n^{\frac{\alpha_{\mathcal{Z}}}{\alpha_\epsilon}} [L_\epsilon(n^{\frac{1}{\alpha_\epsilon}})]^{\frac{\alpha_{\mathcal{Z}}}{\alpha_\epsilon}} [L_{\mathcal{Z}}(n^{\frac{1}{\alpha_\epsilon}})]^{-1} (J_{1n}'+J_{3n}') =0.
\end{align*}
\end{lem}

\subsection{Proof of Theorem \ref{thm:LDP_kappa}}{\label{sec:5.2}}
We first study the non-degenerate $\mathcal{Z}$ with support $[z_0,\ff)$, and then do the degenerate case. As before, here and below, $C, C_1, C_2,\dots$ denote finite positive constants whose values may change from line to line.

\subsubsection{Proof of the lower bounded support case}

\begin{proof}
Denote the CDF of $\mathcal{Z}_n$ by $F_{\mathcal{Z}_n}(z)$. Applying conditional expectation and the conditional sharp large deviations, Theorem \ref{thm:CBahadurRao}, and denoting $\psi_n(z)=\frac{1}{\theta_{x,n}(z)\sigma_{x,n}(z)}$, which is continuous in $z$, we obtain
\begin{align*}
    \bP{L_n>nx} =& \int_{-\infty}^{\infty} P(L_n>nx|\mathcal{Z}_{n}=z) \mathrm{d}F_{\mathcal{Z}_n}(z) = n^{-\frac{1}{2}}\frac{1}{\sqrt{2\pi}}\int_{z_0}^{\infty} e^{-n \Lambda^*_{n}(x;z)}\psi_n(z)(1+r_n(z)) \mathrm{d}F_{\mathcal{Z}_n}(z).
\end{align*}
Let
\begin{align*}
& D_n(z_0) = \frac{\partial}{\partial z}\Lambda^*_{n}(x;z)\Big|_{z=z_0}
= \frac{d\theta_{x,n}(z)}{dz}\Big|_{z=z_0} -\frac{\partial}{\partial z}\Lambda_n(\theta;z)\Big|_{\theta=\theta_{x,n}(z), z=z_0}.
\end{align*}
By Proposition \ref{prop:theta_nx}, the above is same as
\begin{align*}
-\frac{\frac{\partial^2}{\partial z\partial\theta}\Lambda_n(\theta;z)}{\frac{\partial^2}{\partial\theta^2}\Lambda_n(\theta;z)}\Big|_{\theta=\theta_{x,n}(z),z=z_0} -\frac{\partial}{\partial z}\Lambda_n(\theta;z)\Big|_{\theta=\theta_{x,n}(z),z=z_0}.
\end{align*}
Now plugging in the formula in the Proposition \ref{app:prop_lam_n}, and taking $\theta=\theta_{x,n}(z_0)$ and $z=z_0$, we obtain the expression for $D_n(z_0)$.
Also by Proposition \ref{prop:CRateFunction}, we observe that $\Lambda_n^*(x;z)$ is strictly increasing in $z$, and hence $D_n(z_0)>0$ for all $n$. Furthermore, using the Proposition \ref{prop:theta_nx} (the convergence of $\theta_{x,n}(z)$) and the assumption {\ref{asp:L2}}, we obtain $D_n(z_0)\to D(z_0)>0$, where the expression for $D(z_0)$ is obtained by replacing $U^{(n)}$ by $U$, $b^{(n)}$ by $b$, and $\theta_{x,n}(z_0)$ by $\theta_{x}(z_0)$.
By Laplace method,
\begin{align*}
    \lim_{n\to\ff} \frac{\int_{z_0}^{\infty} e^{-n \Lambda^*_{n}(x;z)}\psi_n(z)(1+r_n(z)) f_{\mathcal{Z}_n}(z) \mathrm{d}z}{e^{-n\Lambda^*_{n}(x;z_0)}\psi_n(z_0)(1+r_n(z_0)) f_{\mathcal{Z}_n}(z_0) [n D_n(z_0)]^{-1}} =1.
\end{align*}
Notice that as $n\to\ff$, $f_{\mathcal{Z}_n}(z_0)\to f_{\mathcal{Z}}(z_0)$, $\psi_n(z_0)\to\psi_{\ff}(z_0)=[\theta_x(z_0)]^{-1} [\Lambda''(\theta_{x}(z_0);z_0)]^{-\frac{1}{2}}$, where the derivatives of $\Lambda(\theta;z_0)$ is taken with respect to $\theta$, and $(1+r_n(z_0))\to1$, and $D_n(z_0)\to D(z_0)$. And using Proposition \ref{prop:CRateFunction} that $n|\Lambda^*_{n}(x;z_0)-\Lambda^*(x;z_0)|\to0$ as $n\to\ff$, we establish that
\begin{align*}
    \lim_{n\to\ff} n^{\frac{3}{2}}e^{-n\Lambda^*(x;z_0)} \bP{L_n\geq nx}= \frac{C_{z_0}}{\sqrt{2\pi}\theta_{x}(z_0)\sqrt{\Lambda''(\theta_{x}(z_0);z_0)}}=\frac{C_{z_0}\psi_{\ff}(z_0)}{\sqrt{2\pi}},
\end{align*}
where $C_{z_0}=\frac{f_{\mathcal{Z}}(z_0)}{D(z_0)}=\frac{f_{\mathcal{Z}}(z_0)}{(\Lambda^*)'(x;z_0)}$, and $(\Lambda^*)'(x;z)$ is derivative of $\Lambda^*(x;z)$ with respect to $z$.
\end{proof}

\subsubsection{Proof of the degenerate case}

\begin{proof}
Denote the CDF of $\mathcal{Z}_n$ by $F_{\mathcal{Z}_n}(z)$. Applying conditional expectation and the conditional sharp large deviations, Theorem \ref{thm:CBahadurRao}, and denoting $\psi_n(z)=\frac{1}{\theta_{x,n}(z)\sigma_{x,n}(z)}$, which is continuous in $z$, we obtain
\begin{align*}
    \bP{L_n>nx} =& \int_{-\infty}^{\infty} \mathbf{P}(L_n>nx|\mathcal{Z}_{n}=z) \mathrm{d}F_{\mathcal{Z}_n}(z) = n^{-\frac{1}{2}}\frac{1}{\sqrt{2\pi}}\int_{-\infty}^{\infty} e^{-n \Lambda^*_{n}(x;z)}\psi_n(z)(1+r_n(z)) \mathrm{d}F_{\mathcal{Z}_n}(z)\\
    =& n^{-\frac{1}{2}}\frac{\psi_{\ff}(\kappa)}{\sqrt{2\pi}}e^{-n \Lambda^*(x;\kappa)}\int_{-\infty}^{\infty} e^{-n (\Lambda^*_{n}(x;z)-\Lambda^*(x;\kappa))}\frac{\psi_n(z)}{\psi_{\ff}(\kappa)}(1+r_n(z)) \mathrm{d}F_{\mathcal{Z}_n}(z).
\end{align*}
Let $\psi_{\ff}(\kappa)\coloneqq \lim\limits_{n\to\ff}\psi_n(\kappa)=\lim\limits_{n\to\ff}[\theta_{x,n}(\kappa)]^{-1} [\Lambda_n''(\theta_{x,n}(\kappa);\kappa)]^{-\frac{1}{2}} = [\bar\theta_{x,\kappa}]^{-1} [\Lambda''(\bar\theta_{x,\kappa};\kappa)]^{-\frac{1}{2}}$. The derivatives of $\Lambda_n(\theta;\kappa)$ and $\Lambda(\theta;\kappa)$ are taken with respect to $\theta$. The limit exists from Assumption {\bf\ref{assumption3}}, Proposition \ref{prop:CMGF}, Proposition \ref{prop:theta_nx}, and Proposition \ref{prop:CRateFunction}. For arbitrary $\eta>0$ and fixed large positive $z_0>\kappa$, we decompose the integral on the RHS into four parts: 
\begin{align*}
    \int_{-\infty}^{\infty} e^{-n (\Lambda^*_{n}(x;z)-\Lambda^*(x;\kappa))}\frac{\psi_n(z)}{\psi_{\ff}(\kappa)}(1+r_n(z)) \mathrm{d}F_{\mathcal{Z}_n}(z) = J_{1n} + J_{2n} + J_{3n} + J_{4n}, \quad \text{where}
\end{align*}
$J_{1n}, J_{2n}, J_{3n}, ~\text{and} ~J_{4n}$ are integrals over the ranges $(-\infty,\kappa-\frac{\eta}{n})$, $(\kappa-\frac{\eta}{n},\kappa+\frac{\eta}{n})$, $(\kappa+\frac{\eta}{n},z_0)$, and $(z_0,\ff)$. 
Notice that $0<\sup\limits_{n\geq1,z\leq z_0}\frac{\psi_n(z)}{\psi_{\ff}(\kappa)}(1+r_n(z))\leq C_1$ for a finite positive constant $C_1$, and as $z\leq z_0$, $\left|\frac{\partial}{\partial z}\Lambda^*_{n}(x;z)\right|<C_2$ for a finite positive constant $C_2$. We will show that $J_{2n}\to1$, $J_{1n}\to0$, $J_{3n}\to0$, and $J_{4n}\to0$ as $n\to\infty$.
The proofs of these results rely on the following claim whose proof is based on the differentiability of $\Lambda^*_{n}(x;z)$ and Proposition \ref{prop:CRateFunction}.

\noindent\textbf{Claim:} Under conditions of the theorem, $|\Lambda^*_{n}(x;z)-\Lambda^*(x;\kappa)|\leq C_2 |z-\kappa| + o(\frac{1}{n})$ for $z\leq z_0$, as $n \to \ff$.

\noindent{\bf Proof of the Claim:} By the differentiability of $\Lambda^*_{n}(x;z)$ in $z$, that is, $\left|\frac{\partial}{\partial z}\Lambda^*_{n}(x;z)\right|<C_2$ for $z\leq z_0$, and applying the mean value theorem, we obtain
\begin{align*}
    |\Lambda^*_{n}(x;z)-\Lambda^*(x;\kappa)|\leq |\Lambda^*_{n}(x;z)-\Lambda^*_{n}(x;\kappa)| + |\Lambda^*_{n}(x;\kappa)-\Lambda^*(x;\kappa)| \leq C_2|z-\kappa|+ |\Lambda^*_{n}(x;\kappa)-\Lambda^*(x;\kappa)|.
\end{align*}
From Proposition \ref{prop:CRateFunction} that $n|\Lambda^*_{n}(x;\kappa)-\Lambda^*(x;\kappa)|\to0$ as $n\to\ff$, then
\begin{align*}
    |\Lambda^*_{n}(x;z)-\Lambda^*(x;\kappa)| \leq C_2|z-\kappa|+ o(\frac{1}{n}).
\end{align*}
This completes the proof the claim. We now return to the proof of the theorem. We start with $J_{2n}$. From the claim that there exists a sequence $\{c_n\}$ such that as $n\to\ff$, $c_n\to0$ and 
\begin{align*}
    \lim_{n\to\ff} J_{2n} \leq& \lim_{n\to\ff} \int_{\kappa-\frac{\eta}{n}}^{\kappa+\frac{\eta}{n}} e^{nC_2|z-\kappa|+c_n}\frac{\psi_n(z)}{\psi_{\ff}(\kappa)}(1+r_n(z)) \mathrm{d}F_{\mathcal{Z}_n}(z)\\
    \leq& \lim_{n\to\ff} e^{C_2\eta+c_n} \lim_{n\to\ff}\sup_{z\in[\kappa-\frac{\eta}{n},\kappa+\frac{\eta}{n}]}\frac{\psi_n(z)}{\psi_{\ff}(\kappa)} = e^{C_2\eta}.
\end{align*}
Similarly, the lower bound of limit of $J_{2n}$ can be obtained as follows:
\begin{align*}
    \lim_{n\to\ff} J_{2n} \geq& \lim_{n\to\ff} \int_{\kappa-\frac{\eta}{n}}^{\kappa+\frac{\eta}{n}} e^{-nC_2|z-\kappa|+c_n}\frac{\psi_n(z)}{\psi_{\ff}(\kappa)}(1+r_n(z)) \mathrm{d}F_{\mathcal{Z}_n}(z)\\
    \geq& \lim_{n\to\ff} e^{-C_2\eta+c_n} \lim_{n\to\ff}\inf_{z\in[\kappa-\frac{\eta}{n},\kappa+\frac{\eta}{n}]}\frac{\psi_n(z)}{\psi_{\ff}(\kappa)} = e^{-C_2\eta}.
\end{align*}
Combining the lower bound and upper bound of $J_{2n}$, and letting $\eta\to0$, we obtain $J_{2n}\to1$ as $n\to\ff$. Now turning to $J_{1n}$, using the claim and the uniform upper bound, $C_1$, of $\frac{\psi_n(z)}{\psi_{\ff}(\kappa)}(1+r_n(z))$ for  $n\geq1$ and $z\leq z_0$, we decompose $J_{1n}$ as follows:
\begin{align*}
    J_{1n} =& \sum_{i=1}^{\infty} \int_{\kappa-\frac{(i+1)\eta}{n}}^{\kappa-\frac{i\eta}{n}} e^{-n (\Lambda^*_{n}(x;z)-\Lambda^*(x;\kappa))}\frac{\psi_n(z)}{\psi_{\ff}(\kappa)}(1+r_n(z)) \mathrm{d}F_{\mathcal{Z}_n}(z)\\
    \leq& C_1 \sum_{i=1}^{\infty} \int_{\kappa-\frac{(i+1)\eta}{n}}^{\kappa-\frac{i\eta}{n}} e^{nC_2|z-\kappa|+c_n} \mathrm{d}F_{\mathcal{Z}_n}(z) \leq C_1 \sum_{i=1}^{\infty} e^{C_2(i+1)\eta + c_n} P\left(n|\mathcal{Z}_n-\kappa|>i\eta\right).
\end{align*}
By the assumption that $\mathbf{P}(n|\mathcal{Z}_n-\kappa|>\eta)\to 0$ for any fixed $\eta\in(0,\infty)$,  there exists $\{\gamma_n>0\}$ such that $\gamma_n\to\infty$ and $P\left(n|\mathcal{Z}_n-\kappa|>i\eta\right)\leq e^{-\gamma_n i\eta}$. Therefore,
\begin{align*}
    J_{1n} \leq& C_1 \sum_{i=1}^{\infty} e^{C_2(i+1)\eta + c_n} e^{-\gamma_n i\eta}
    = C_1 e^{\eta + c_n} \sum_{i=1}^{\infty} e^{(C_2-\gamma_n)i\eta}.
\end{align*}
Let $h_n=e^{(C_2-\gamma_n)\eta}$, then as $n\to\ff$, $h_n\to0$ and 
\begin{align*}
    J_{1n} = C_1 e^{\eta + c_n} \sum_{i=1}^{\infty} (h_n)^i = C_1 e^{\eta + c_n} \frac{h_n}{1-h_n}\to0.
\end{align*}
Using similar methods, one can show $J_{3n}\to0$. Turning to $J_{4n}$, notice that $\bP{L_n>nx|\mathcal{Z}_n=z}$ is bounded above by $\bP{L_n>nx|\mathcal{Z}_n=z_0}$ and $\Lambda^*_{n}(x;z)-\Lambda^*(x;\kappa)>C_3>0$ for all $z>z_0$; hence as $n\to\ff$
\begin{align*}
    J_{4n}\leq& e^{-nC_3}\frac{\psi_n(z_0)}{\psi_{\ff}(\kappa)}(1+r_n(z_0)) \to 0.
\end{align*}
Combining these results, it follows that $n^{\frac{1}{2}}e^{n \Lambda^*(x;\kappa)} \bP{L_n>nx}$ converges to $\psi_{\ff}(\kappa) (2\pi)^{-1/2}$ as $n\to\ff$.
\end{proof}

\subsection{Proof of Theorem \ref{thm:LDP_gen_ndZ}}\label{sec:pf-thm23}
Recall
\[
\phi_n(z):=\Lambda_n^*(x;z)-\Lambda_{U^{(n)}}^*(x),\qquad
\tilde h_n(z):=-\phi_n(z)+\frac{1}{n}\log f_{\mathcal{Z}_n}(z),
\]
and
\[
\psi_n(z):=\frac{1}{\theta_{x,n}(z)\,\sigma_{x,n}(z)},\qquad
\psi_\infty:=\frac{1}{\theta_x\sqrt{\Lambda_U''(\theta_x)}}.
\]
Let $-M_n$ be the balance \ref{asp:B1} locator and $z_n^*:=-\tilde M_n$ the (true) saddle.

\begin{lem}[Localization with fixed cutoff under log-smooth tails]\label{lem:loc-logsmooth}
Assume \textup{\ref{asp:D1}}, \textup{\ref{asp:D2}}, \textup{\ref{asp:B1}}, \textup{\ref{asp:W1}-\ref{asp:W3}} and fix $z_0\in\mathbb{R}$.
Write
\[
\bP{L_n\ge nx}
=\int_{-\infty}^{\infty}\bP{L_n\ge nx\mid \mathcal{Z}_n=z}\,f_{\mathcal{Z}_n}(z)\,dz
=:T_{1n}+T_{2n},
\]
where $T_{1n}$ integrates over $(-\infty,z_0]$ and $T_{2n}$ over $(z_0,\infty)$.
Then for all sufficiently large $n$ (so that $-\tilde M_n<z_0$) the following hold:
\begin{enumerate}[label=(\alph*)]
\item[\textup{(1)}] \textbf{Right tail is negligible:} there exist constants $C_1,C_2<\infty$ (independent of $n$) with
\[
T_{2n}\ \le\ C_1\,e^{-n\Lambda_n^*(x;z_0)}
\ \le\ C_2\,e^{-n\Lambda^*(x;z_0)}
\ =\ o\Big(e^{-n\Lambda^*_U(x)}\,e^{-n\phi_n(-\tilde M_n)}\,[H_n(-\tilde M_n)]^{-1}\Big).
\]
\item[\textup{(2)}] \textbf{\(J\)-block decomposition on $(-\infty,z_0]$:} Fix any $\beta\in(0,1)$ and decompose
\[
T_{1n}\ =\ \frac{1}{\sqrt{2\pi}} n^{-\frac{1}{2}}e^{-n\Lambda^*_{U}(x)} (J_{1n}+J_{2n}+J_{3n}) (1+o(1), \quad\text{as }n\to\ff,
\]
with integration ranges $(-\infty,-(1+\beta)M_n]$, $[-(1+\beta)M_n,\,-(1-\beta)M_n]$, and $ (-(1-\beta)M_n,\ z_0]$, respectively, applied to the integrand from Theorem \ref{thm:CBahadurRao}.
Then for every $\epsilon>0$ there exist $A\ge1$ and $N$ such that for all $n\ge N$,
\[
\frac{J_{1n}+J_{3n}}{J_{2n}}\ \le\ \epsilon.
\]
\item[\textup{(3)}] \textbf{Laplace window around the true saddle:} letting $z_n^*:=-\tilde M_n$ and $w_n:=(n|\tilde h_n''(z_n^*)|)^{-1/2}$, one has
\[
J_{2n}
=\int_{|z-z_n^*|\le A w_n} e^{n\tilde h_n(z)}\,\psi_n(z)\{1+r_n(x,z)\}\,dz\ \{1+o(1)\}
=\sqrt{2\pi}\psi_\infty\,e^{n\tilde h_n(z_n^*)}\,[n|\tilde h_n''(z_n^*)|]^{-1/2}\{1+o(1)\},
\]
uniformly in $n$.
\end{enumerate}
\end{lem}
\begin{proof}
    See Appendix \ref{app:loc-logsmooth}.
\end{proof}

\begin{rem}{\label{rem:log-smooth_proof}}
This is the log-smooth analogue of our GN/RV setup in Section \ref{sec:5.1}: first fix $z_0$ and split $\bP{L_n\ge nx}=T_{1n}+T_{2n}$; then analyze $T_{1n}$ via the $J_{1n} + J_{2n} + J_{3n}$ decomposition, where $J_{2n}$ carries the mass and the side blocks are negligible. Part (1) uses that $z\mapsto \Lambda_n^*(x;z)$ is increasing, so $T_{2n}$ is exponentially dominated by $e^{-n\Lambda^*(x;z_0)}$; parts (2)–(3) replicate the role of Lemmas \ref{pf:lem_Gen_3} and \ref{pf:lem_RR_3} with \textup{\ref{asp:D1}}, \textup{\ref{asp:D2}}, \textup{\ref{asp:B1}}, {\rm\ref{asp:W1}-\ref{asp:W3}} replacing the GN/RV tail specifics. See eq. (\ref{eq:gen_1}) and the surrounding discussion. 
\end{rem}

\begin{proof}[Proof of Theorem \ref{thm:LDP_gen_ndZ}]
Fix $x>\mu_U$ and a cutoff $z_0\in\R$. By conditional Bahadur-Rao (Theorem \ref{thm:CBahadurRao}),
\[
\bP{L_n\ge nx}=\frac{e^{-n\Lambda_U^*(x)}}{\sqrt{2\pi n}}
\int_{-\infty}^{\infty} e^{-n\phi_n(z)}\,\psi_n(z)\,\{1+r_n(x,z)\}\,f_{\mathcal{Z}_n}(z)\,dz.
\]
Set $\tilde h_n(z):=-\phi_n(z)+n^{-1}\log f_{\mathcal{Z}_n}(z)$ and split the integral at $z_0$ as in Lemma~\ref{lem:loc-logsmooth}.
By Lemma~\ref{lem:saddle-consistency}, $z_n^*:=-\tilde M_n$ is the unique maximizer of $\tilde h_n$ and $\tilde M_n=M_n(1+o(1))$.
Part~(1) of Lemma~\ref{lem:loc-logsmooth} gives $T_{2n}=o\big(n^{-\frac{1}{2}}e^{-n\Lambda_U^*(x)}e^{-n\phi_n(-\tilde M_n)}[H_n(-\tilde M_n)]^{-1}\big)$.
Parts~(2)–(3) yield
\[
T_{1n}
=\psi_\infty  n^{-\frac{1}{2}}e^{-n\Lambda^*_{U}(x)} e^{n\tilde h_n(z_n^*)}\,[n|\tilde h_n''(z_n^*)|]^{-1/2}\{1+o(1)\}.
\]
Recalling $e^{n\tilde h_n(z_n^*)}=e^{-n\phi_n(-\tilde M_n)}\,f_{\mathcal{Z}_n}(-\tilde M_n)$ and
$[H_n(-\tilde M_n)]^{-1}:=f_{\mathcal{Z}_n}(-\tilde M_n)[n|\tilde h_n''(-\tilde M_n)|]^{-1/2}$,
we conclude
\[
\bP{L_n\ge nx}\sim \psi_\infty n^{-\frac{1}{2}}\,
e^{-n\Lambda_U^*(x)}\,e^{-n\phi_n(-\tilde M_n)}\,[H_n(-\tilde M_n)]^{-1},
\]
which is the statement of Theorem \ref{thm:LDP_gen_ndZ}.
\end{proof}

\subsection{Proof of Theorem \ref{thm:gibbs}}

\begin{proof} 
For $B\in\mathcal{B}(\Real)$, define the measure $\nu_n(B)=\bP{U_1^{(n)}\in B, X_1^{(n)}=1| L_n \ge nx}$ and  $\nu_{\theta_x}(B)=\mathbf{P}_{\theta_x}(U\in B)$.  Notice that, for any fixed $Q>0$
\begin{align}{\label{pf:keyineq}}
    \sup_{B \in \mathcal{B}(\Real)}\left| \nu_n(B)-\nu_{\theta_x}(B) \right| \leq \sup_{B_1 \subset [-Q, Q]}\left| \nu_n(B_1)-\nu_{\theta_x}(B_1) \right| + \nu_n(|U|>Q)+\nu_{\theta_x}(|U|>Q).
\end{align}
Hence, it is sufficient to show that the RHS converges to zero as $n\to\ff$. We will first consider the case $\kappa=-\ff$ and begin by verifying that the first term converges to zero. To this end, using the definition of conditional probability,
\begin{align}\label{pf:eq_2.3.1}
    \bP{U_1^{(n)}\in B_1, X_1^{(n)}=1 | L_n>nx} 
    =& \frac{\bP{ L_n>nx, U_1^{(n)}\in B_1, X_1^{(n)}=1}}{\bP{ L_n>nx}},
\end{align}
and setting $A_n(x)=(2\pi)^{-\frac{1}{2}}\exp(n\phi_n(-\tilde M_n))H_n(-\tilde M_n) (\psi_\ff)^{-1}$, it follows from Theorem \ref{thm:LDP_ndZ} that
\begin{align*}
    \lim_{n\to\ff} \sqrt{2\pi} n^{\frac{1}{2}}e^{n\Lambda_U^*(x)} A_n(x) \bP{L_n>nx}=1.
\end{align*}
Turning to the numerator of (\ref{pf:eq_2.3.1}), notice that
\begin{align}\label{pf:eq_2.3.2}
    \bP{ L_n>nx, U_1^{(n)}\in B_1, X_1^{(n)}=1}
    = \int_{-\ff}^{\ff} \bP{ L_n>nx, U_1^{(n)}\in B_1, X_1^{(n)}=1|\mathcal{Z}_n=z} \mathrm{d}F_{\mathcal{Z}_n}(z),
\end{align}
where $F_{\mathcal{Z}_n}(\cdot)$ is the CDF of $\mathcal{Z}_n$. Now, using the definition of conditional probability, the RHS of the above equation is the same as
\begin{eqnarray*}
\int_{-\ff}^{\ff}\int_{B_1} \bP{ L_n>nx|X_1^{(n)}=1, U_1^{(n)}=u, \mathcal{Z}_n=z} F_{\ep}\left(\frac{v-z}{b^{(n)}}\right) \mathrm{d}F_{U^{(n)}}(u)\mathrm{d}F_{\mathcal{Z}_n}(z).
\end{eqnarray*}
Next, setting $x_n(u)=\frac{n}{n-1}x-\frac{u}{n-1}$  and interchanging the order of integration (justified using Tonelli's Theorem), it follows that the LHS of equation (\ref{pf:eq_2.3.2}) equals 
\begin{align} \label{pf:eq_2.3.3}
    \int_{B_1} \int_{-\ff}^{\ff} \bP{ L_{n-1}>(n-1)x_n(u) | \mathcal{Z}_n=z} F_{\ep}\left(\frac{v-z}{b^{(n)}}\right) \mathrm{d}F_{\mathcal{Z}_n}(z)\mathrm{d} F_{U^{(n)}}(u).
\end{align}
We now study the inner integral in the above equation. Set $B_*=[-Q, Q]$. We proceed as in the proof of Theorem \ref{thm:LDP_ndZ} and express it as a sum of $T_{1n}(u)$ and $T_{2n}(u)$, where we emphasize the dependence on $u$. Now, using Theorem \ref{thm:CBahadurRao} and Proposition \ref{prop:CRateFunction} it follows that $\sup_{u \in B_*}T_{2n}(u) [T_{1n}(u)]^{-1}$ converges to zero. Next, turning to $T_{1n}(u)$, as before, we express it as a product of a prefactor term and the integrals $J_{1n}(u)$, $J_{2n}(u)$ , and $J_{3n}(u)$ as in equation (\ref{eq:gen_1}). Next, observe again using the boundedness of $B_*$ that
\begin{align*}
 \lim_{n \ra \ff}  \sup_{u \in B_*} \frac{J_{1n}(u) +J_{3n}(u)}{J_{2n}(u)} =0,
\end{align*}
and $J_{2n}(u)=A_{n-1}(x_n(u))+O_u(\frac{1}{n})$, where the subscript $u$ emphasizes dependence on $u$. Also, using the boundedness of $B_*$, continuity of $J_{in}(u)$ in $u$, and  assumption \ref{asp:L2}, it follows that
\begin{align*}
\lim_{n \to \ff}\sup_{u \in B_*}|\exp(n|\Lambda_{U^{(n)}}^*(x_n(u))-\Lambda_U^*(x_n(u))|-1|=0.
\end{align*}
Hence, the inner integral in (\ref{pf:eq_2.3.3}) can be expressed as 
\begin{align*}
    \left[\sqrt{2\pi} (n-1)^{\frac{1}{2}}e^{(n-1)\Lambda_U^*(x_n(u))} [A_{n-1}(x_n(u))+O_u(\frac{1}{n})](1+o(1))\right]^{-1},
\end{align*}
as $n\to\ff$, where the $O_u(\frac{1}{n})$ comes from the Laplace method similar to the proof of Theorem \ref{thm:LDP_ndZ}. 
Next, expanding $\Lambda_U^*(x_n(u))$ around $x$, it follows that
\begin{align}
&\Lambda_U^*(x_n(u)) = \Lambda_U^*(x) + (x_n(u)-x) \Lambda_U^{*'}(x) + \frac{\Lambda_U^{*''}(x_n^*)}{2}(x_n(u)-x)^2 \nonumber\\
=& \Lambda_U^*(x) + \left(\frac{1}{n-1}x-\frac{u}{n-1}\right) \Lambda_U^{*'}(x) + O(n^{-2})
= \theta_x x - \Lambda_U(\theta_x) + \frac{(x-u)}{n-1}\theta_{x} + O(n^{-2}). \label{pf:Tayexp}
\end{align}
Thus, the ratio in (\ref{pf:eq_2.3.1}) can be expressed as:
\begin{align*}
   \int_{B_1} \left[\frac{n}{n-1}\right]^{\frac{1}{2}} \left[\frac{e^{n\Lambda_U^*(x)}}{e^{(n-1)\Lambda_U^*(x_n(u))} }\right] \left[\frac{ A_n(x) + O(\frac{1}{n})}{A_{n-1}(x_n(u))+O_u(\frac{1}{n})}\right] (1+o(1)) \mathrm{d} F_{U^{(n)}}(u),
\end{align*}
which, using the boundedness of $B_1$, reduces to
\begin{align}{\label{pf:nrexp}}
   \int_{B_1} \left[\frac{n}{n-1}\right]^{\frac{1}{2}} \left[e^{\theta_x u-\Lambda_U(\theta_x)}\right] \left[\frac{ A_n(x) + O(\frac{1}{n})}{A_{n-1}(x_n(u))+O_u(\frac{1}{n})}\right] (1+o(1)) \mathrm{d} F_{U^{(n)}}(u),
\end{align}
where $O_u(n^{-1})$ term is uniform in $u \in B_*$. Since the third term inside the integral converges to one uniformly for $ u \in [-Q, Q]$, using the bounded continuity of $e^{\theta_xu-\Lambda_U(\theta_x)}$ on $B_{*}$, it follows, using the  $U^{(n)} \Rightarrow U$ and Helly-Bray Theorem that 
\begin{align*}
   \lim_{n \ra \ff} \left| \int_{B_*}e^{\theta_xu-\Lambda_U(\theta_x)} \mathrm{d}(F_{U^{(n)}}(u)-F_{U}(u)) \right|=0.
\end{align*}
Hence, $|\nu_n(B_1)-\nu_{\theta_x}(B_1)|$ converges to zero as $n \to \ff$. Next, using Remark \ref{rem:B2-apply} in Appendix \ref{app:B}, it follows that
\begin{align*}
    \sup_{B_1 \subset B_*}|\nu_n(B_1)-\nu_{\theta_x}(B_1)|  \to 0 ~\text{as}~ n \to \ff. 
\end{align*}
Turning to the second term on the RHS of (\ref{pf:keyineq}), we proceed as before and notice, using the definition of conditional probability, that
\begin{align*}
\nu_n(U^{(n)}>Q) = \frac{\bP{U^{(n)}>Q, X_1^{(n)}=1, L_n \geq nx}}{\bP{L_n \geq nx}}
\end{align*}
Next, proceeding as in the derivation of (\ref{pf:eq_2.3.3}), the numerator of the above expression can be expressed as
\begin{align*}
 \int_{Q}^{\ff} \int_{-\ff}^{\ff} \bP{ L_{n-1}>(n-1)x_n(u) | \mathcal{Z}_n=z} F_{\ep}\left(\frac{v-z}{b^{(n)}}\right) \mathrm{d}F_{\mathcal{Z}_n}(z)\mathrm{d} F_{U^{(n)}}(u).
\end{align*}
Next, using the rates from Theorem \ref{thm:LDP_ndZ} one obtains,  similar to the verification of (\ref{pf:nrexp}), that
\begin{align}{\label{pf:nrexp1}}
   \nu_n(U^{(n)} >Q) =\int_{Q}^{\ff} \left[\frac{n}{n-1}\right]^{\frac{1}{2}} \left[e^{\theta_x u-\Lambda_U(\theta_x)}\right] \left[\frac{ A_n(x) + O(\frac{1}{n})}{A_{n-1}(x_n(u))+O_u(\frac{1}{n})}\right] (1+o(1)) \mathrm{d} F_{U^{(n)}}(u).
\end{align}
Now using the expression (\ref{pf:Tayexp}), it can be seen (see Lemma \ref{lem:prefactor-stability} in Appendix \ref{app:B}) that there exist positive finite constants $c_1$ and $c_2$ (independent of $Q, n$) such that
\begin{align*}
   \left[\frac{ A_n(x) + O(\frac{1}{n})}{A_{n-1}(x_n(u))+O_u(\frac{1}{n})}\right] \le 1+\frac{c_1 u}{n} +\frac{c_2 u^2}{n^2}.
\end{align*}
Plugging into (\ref{pf:nrexp1}) and using Assumptions \ref{asp:L2} and \ref{asp:L3}, it follows using exponential Markov inequality (see Lemma \ref{lem:uniform-tilted-tail} in Appendix \ref{app:B}) that
\begin{align*}
\sup_{n \ge 1}\nu_n(U^{(n)} >Q) \le C e^{-Qt},
\end{align*}
where $t >0$ is such that
$\ta_x+t < \ta_0$. Turning to the third term on the RHS of \ref{pf:keyineq}, we notice that, using \ref{asp:L2} and \ref{asp:L3}, since $U$ has exponential moments under $\mathbf{P}_{\theta_x}$, by the exponential Markov inequality it follows that, for $t >0$ such that
$\ta_x+t < \ta_0$
\begin{align}{\label{pf:tailineq}}
\nu_{\ta_x}(U >Q) \le \frac{\la(\ta_x+t)}{\la(\ta_x)}\exp(-Qt).
\end{align}
Since $Q$ is arbitrary, the proof follows. The proof when $\kappa >-\ff$ and $\mathcal{Z}$ is non-degenerate is similar. Finally, proof of (ii) follows by iterating the above remove-one-term argument for the numerator with $O(n^{-1})$ cumulative error, yielding i.i.d. tilted limits.
\end{proof}

\begin{rem}[Growing blocks]\label{rem:kn}
The total-variation Gibbs limit described in the above Theorem extends from fixed $k$ to $k=k_n\to\infty$ with $k_n/n\to0$. 
Indeed, in the Bayes ratio the numerator involves $y_{n,k}:=x-\frac{1}{n}\sum_{i=1}^{k_n}U_i^{(n)}X_i^{(n)}$, and the perturbation satisfies $|y_{n,k}-x|=O_{\mathbb P}(k_n/n)$. 
Uniform conditional BR/BE bounds and the smoothness of $y\mapsto(\theta_{y,n}(z),\sigma_{y,n}(z))$ then give the same product limit measure as for fixed $k$, with an error $o(1)$ in total variation.
\end{rem}

\section{Concluding Remarks} \label{sec:concluding}

We obtain sharp large deviation estimates for the total loss $L_n$ in threshold factor models with a diverging number of common factors. Unlike the classical Bahadur-Rao prefactor $n^{-1/2}$, our results exhibit logarithmic and polynomial corrections whose form depends explicitly on the tails of the factor distribution $Z$ and the idiosyncratic noise $\epsilon$. We also established a Gibbs conditioning principle in total variation for the conditional law under rare-loss events, identifying the tilted limit law and clarifying when dependence disappears at the conditioning scale. The analysis rests on Laplace-Olver asymptotics for exponential integrals, conditional Bahadur-Rao bounds for the triangular arrays, and a localization of the saddle dictated by the tail geometry.

Beyond the generalized normal and regularly varying classes, we formulated log-smooth conditions that place the model in the Gumbel maximum domain of attraction, thereby unifying light- and heavy-tail regimes under a single toolkit. As an illustration, we derived second-order expansions for Value-at-Risk and Expected Shortfall and indicated when these measures are genuinely in the large-deviation regime.

\emph{Extensions.} An important next step is the $d$-dimensional setting, where the loss vector has multiple components or types. In this case, sharp large deviations require identification of the dominating point(s) of the multidimensional rate function. The prefactor then depends on the local curvature of the rate function at these points, and multiple competing saddles may contribute. These issues introduce genuine analytical challenges beyond the one-dimensional setting. Work in this direction is ongoing.
\newpage

\appendix

\section{Appendix}\label{app:A}

This appendix records the conditional c.g.f.\ $\Lambda_n(\theta;z)$ and its derivatives, and provides the proof details for Lemmas~\ref{pf:lem_Gen_1}–\ref{lem:loc-logsmooth} used in Section~\ref{sec:proof}.
\begin{prop} \label{app:prop_lam_n}
$\Lambda_n(\theta;z)$ is the logarithmic moment generating function of $U^{(n)}X^{(n)}$ conditioned on $\mathcal{Z}_n=z$. That is,
\begin{eqnarray*} 
\Lambda_n(\theta;z)=\log\left[ \lambda_{U^{(n)}}(\theta)F_{\epsilon}(\frac{v-z}{b^{(n)}}) + 1- F_{\epsilon}(\frac{v-z}{b^{(n)}}) \right].
\end{eqnarray*}
The first and second partial derivatives are given by:
\begin{align*}
\dfrac{\partial}{\partial \theta}\Lambda_n(\theta;z)=& \frac{F_{\epsilon}(\frac{v-z}{b^{(n)}})\lambda_{U^{(n)}}'(\theta)}{F_{\epsilon}(\frac{v-z}{b^{(n)}})\lambda_{U^{(n)}}(\theta)+1-F_{\epsilon}(\frac{v-z}{b^{(n)}})}\\
\dfrac{\partial}{\partial z}\Lambda_n(\theta;z)=& \dfrac{\frac{1}{b^{(n)}}f_{\epsilon}(\frac{v-z}{b^{(n)}})(1-\lambda_{U^{(n)}}(\theta))}{F_{\epsilon}(\frac{v-z}{b^{(n)}})\lambda_{U^{(n)}}(\theta)+1-F_{\epsilon}(\frac{v-z}{b^{(n)}})}\\
\dfrac{\partial^2}{\partial z\partial\theta}\Lambda_n(\theta;z) &=
\frac{-f_{\ep}\left(\frac{v-z}{b^{(n)}}\right)\frac{1}{b^{(n)}}\mathbf{E}[U^{(n)} e^{\theta U^{(n)}}]}{\left[F_{\ep}\left(\frac{v-z}{b^{(n)}}\right)\mathbf{E}[e^{\theta U^{(n)}}]+1 - F_{\ep}\left(\frac{v-z}{b^{(n)}}\right)\right]^2}\\
\frac{\partial^2}{\partial\theta^2}\Lambda_n(\theta;z) &=
\frac{F_{\ep}\left(\frac{v-z}{b^{(n)}}\right)\left[\mathbf{E}[U^2 e^{\theta U^{(n)}}]\left(F_{\ep}\left(\frac{v-z}{b^{(n)}}\right)\mathbf{E}[e^{\theta U^{(n)}}]+1 - F_{\ep}\left(\frac{v-z}{b^{(n)}}\right)\right)-\left(\mathbf{E}[U^{(n)} e^{\theta U^{(n)}}]\right)^2 F_{\ep}\left(\frac{v-z}{b^{(n)}}\right)\right]}{\left[F_{\ep}\left(\frac{v-z}{b^{(n)}}\right)\mathbf{E}[e^{\theta U^{(n)}}]+1 - F_{\ep}\left(\frac{v-z}{b^{(n)}}\right)\right]^2}
\end{align*}
\end{prop}

\subsection{Proof of Lemma \ref{pf:lem_Gen_1}} \label{app:lem_Gen_1}
Applying Proposition \ref{prop:CRateFunction}, it follows that
\begin{align*}
    h_n(t)=& \frac{1}{n}[-n\phi_n(-t M_n)+ \log f_{\mathcal{Z}_n}(-t M_n)]\\
    =&\frac{1}{n}\left[ - n C_n(\theta_{x,n})\cdot \frac{\beta_\ep}{\gamma\xi_\ep}(\frac{v+t M_n}{b^{(n)}})^{1-\gamma}e^{-\xi_\ep (\frac{v + t M_n}{b^{(n)}})^{\gamma}} (1+o(1)) + (\log \beta_{\mathcal{Z}_n} - \xi_{\mathcal{Z}_n} (t M_n)^{\gamma})(1+o(1)) \right]\\
    h_n'(t) =& \frac{1}{n}\left[ - n C_n(\theta_{x,n})\cdot \frac{\beta_\ep}{\gamma\xi_\ep} e^{-\xi_\ep (\frac{v + t M_n}{b^{(n)}})^{\gamma}} \frac{M_n}{b^{(n)}}\left[(1-\gamma)(\frac{v+t M_n}{b^{(n)}})^{-\gamma} -\xi_\ep \gamma \right] (1+o(1)) -\xi_{\mathcal{Z}_n} \gamma M_n^{\gamma} t^{\gamma-1} (1+o(1)) \right] \\
    =& \frac{1}{n}\left[ n C_n(\theta_{x,n})\cdot \beta_\ep e^{-\xi_\ep (\frac{v + t M_n}{b^{(n)}})^{\gamma}} \frac{M_n}{b^{(n)}} (1+o(1)) -\xi_{\mathcal{Z}_n} \gamma M_n^{\gamma} t^{\gamma-1} (1+o(1)) \right]\\
    h_n''(t) =& \frac{1}{n}\left[ n C_n(\theta_{x,n})\cdot \beta_\ep \frac{M_n}{b^{(n)}} e^{-\xi_\ep (\frac{v + t M_n}{b^{(n)}})^{\gamma}}\left(-\xi_\ep\gamma (\frac{v + t M_n}{b^{(n)}})^{\gamma-1} \frac{M_n}{b^{(n)}} \right) - \xi_{\mathcal{Z}_n} \gamma M_n^{\gamma}(\gamma-1)t^{\gamma-2} \right]
\end{align*}
In the $h_n''(t)$ term, the first term behaves like $M_n^{\gamma} n M_n e^{-M_n^\gamma}\sim M_n^{\gamma} M_n$ since $M_n^{\gamma}\sim \log n$. The second term behaves like $M_n^\gamma$. Hence the first term is the dominant term and is negative when $t>0$.
As $n\to\ff$, $o(1)$ term can be dropped. For $\gamma\in(0,2)$,
\begin{align*}
    &h_n'(t_{0,n})=0 \\
    \Longleftrightarrow&  n C_n(\theta_{x,n})\cdot \beta_\ep e^{-\xi_\ep (\frac{v + t_{0,n} M_n}{b^{(n)}})^{\gamma}} \frac{M_n}{b^{(n)}} (1+o(1))  = \xi_{\mathcal{Z}_n} \gamma M_n^{\gamma} t_{0,n}^{\gamma-1}\\
    \Longleftrightarrow& \log\frac{C_n(\theta_{x,n}) \beta_\ep}{\xi_{\mathcal{Z}_n} \gamma b^{(n)}} + \log n + \log M_n -\xi_\ep (\frac{v + t_{0,n} M_n}{b^{(n)}})^{\gamma} = \gamma\log M_n + (\gamma-1)\log t_{0,n}\\
    & \text{Notice that $(\frac{t_{0,n}M_n+v}{b^{(n)}})^\gamma=(\frac{t_{0,n}M_n}{b^{(n)}})^{\gamma} + \gamma \frac{v}{b^{(n)}} (\frac{t_{0,n}M_n}{b^{(n)}})^{\gamma-1}(1+O(M_n^{\gamma-2}))$ as $n\to\ff$}\\
    \Longleftrightarrow& \log\frac{C_n(\theta_{x,n}) \beta_\ep}{\xi_{\mathcal{Z}_n} \gamma b^{(n)}} + \log n + \log M_n -\xi_\ep (\frac{t_{0,n} M_n}{b^{(n)}})^{\gamma} - \frac{v\gamma\xi_\ep}{b^{(n)}} (\frac{t_{0,n}M_n}{b^{(n)}})^{\gamma-1}(1+O(M_n^{\gamma-2})) = \gamma\log M_n + (\gamma-1)\log t_{0,n}\\
    \Longleftrightarrow& (t_{0,n}M_n)^\gamma = \frac{(b^{(n)})^{\gamma}}{\xi_\ep}\log n + (1-\gamma)\frac{(b^{(n)})^{\gamma}}{\xi_\ep} \log M_n - v\gamma (t_{0,n} M_n)^{\gamma-1}(1+O(M_n^{\gamma-2})) + \\
    & \quad\quad \frac{(b^{(n)})^{\gamma}}{\xi_\ep}\log\frac{C_n(\theta_{x,n}) \beta_\ep}{\xi_{\mathcal{Z}_n} \gamma b^{(n)}} + (1-\gamma)\frac{(b^{(n)})^{\gamma}}{\xi_\ep}\log t_{0,n}\\
    \Longleftrightarrow& t_{0,n}^{\gamma} = \frac{(b^{(n)})^{\gamma}}{\xi_\ep}\frac{\log n}{M_n^\gamma} + \frac{(1-\gamma)(b^{(n)})^{\gamma}}{\xi_\ep} \frac{\log M_n}{M_n^\gamma} -v\gamma t_{0,n}^{\gamma-1} \frac{1}{M_n} (1+O(M_n^{\gamma-2})) +\\
    & \quad\quad \frac{1}{M_n^\gamma}\left[ \frac{(b^{(n)})^{\gamma}}{\xi_\ep}\log\frac{C_n(\theta_{x,n}) \beta_\ep}{\xi_{\mathcal{Z}_n} \gamma b^{(n)}} + (1-\gamma)\frac{(b^{(n)})^{\gamma}}{\xi_\ep}\log t_{0,n} \right]\\
    \Longleftrightarrow& t_{0,n}^{\gamma} = 1 + \frac{(1-\gamma)(b^{(n)})^{\gamma}}{\xi_\ep} \frac{\log M_n}{M_n^\gamma} -v\gamma t_{0,n}^{\gamma-1} \frac{1}{M_n} (1+O(M_n^{\gamma-2})) +\\
     & \quad\quad \frac{1}{M_n^\gamma}\left[ \frac{(b^{(n)})^{\gamma}}{\xi_\ep}\log\frac{C_n(\theta_{x,n}) \beta_\ep}{\xi_{\mathcal{Z}_n} \gamma b^{(n)}} + (1-\gamma)\frac{(b^{(n)})^{\gamma}}{\xi_\ep}\log t_{0,n} \right].
\end{align*}
And for $\gamma=2$, it follows by similar methods that
\begin{align*}
    t_{0,n}^2 = 1 - \frac{2 v}{M_n} t_{0,n} - \frac{(b^{(n)})^2\log M_n}{\xi_\ep M_n^2} + \frac{1}{M_n^2}\left(\frac{(b^{(n)})^{2}}{\xi_\ep}\log\frac{C_n(\theta_{x,n}) \beta_\ep}{2\xi_{\mathcal{Z}_n} b^{(n)}} - \frac{(b^{(n)})^{2}}{\xi_\ep}\log t_{0,n}  - v^2 \right) (1+o(1))
\end{align*}
Notice that when $\gamma\in(0,1)$, $t_{0,n}=1+O\left(\frac{\log M_n}{M_n^\gamma}\right)$; when $\gamma=1$, $t_{0,n}=1-\frac{v}{M_n}+O\left(\frac{1}{M_n^2}\right)$; when $\gamma\in(1,2)$, $t_{0,n}=1-\frac{v}{M_n}+O\left(\frac{\log M_n}{M_n^\gamma}\right)$; when $\gamma=2$, $t_{0,n}= 1-\frac{v}{M_n} + O(\frac{\log M_n}{M_n^2})$. And for all three cases, $\frac{t_{0,n}^{\gamma-1}}{M_n}=\frac{1}{M_n}+o\left(\frac{1}{M_n^{\gamma}}\right)$, hence
\begin{align*}
    t_{0,n}^{\gamma} = 1 + \frac{(1-\gamma)(b^{(n)})^{\gamma}}{\xi_\ep} \frac{\log M_n}{M_n^\gamma} -\frac{v\gamma }{M_n} + \frac{1}{M_n^\gamma}\left[ \frac{(b^{(n)})^{\gamma}}{\xi_\ep}\log\frac{C_n(\theta_{x,n}) \beta_\ep}{\xi_{\mathcal{Z}_n} \gamma b^{(n)}} -\frac{\log\eta_{\gamma,n}}{\xi_{\mathcal{Z}_n}} \right](1+o(1)),
\end{align*}
where $\eta_{2,n}=\exp\{-v^2\xi_{\mathcal{Z}_n}\}$ and $\eta_{\gamma,n}=1$ for $\gamma\in(0,2)$.
Since $t_{0,n}$ is the root of $h_n'(t)=0$, that is 
$$n C_n(\theta_{x,n})\cdot \beta_\ep e^{-\xi_\ep (\frac{v + t_{0,n} M_n}{b^{(n)}})^{\gamma}} \frac{M_n}{b^{(n)}} (1+o(1))  = \xi_{\mathcal{Z}_n} \gamma M_n^{\gamma} t_{0,n}^{\gamma-1}.$$
\begin{align*}
    &n h_n(t_{0,n})\\
    =& - \xi_{\mathcal{Z}_n} \gamma M_n^{\gamma} t^{\gamma-1} \frac{b^{(n)}}{M_n}\frac{1}{\gamma\xi_\ep} \left(\frac{v+t_{0,n} M_n}{b^{(n)}}\right)^{1-\gamma} +\log\beta_{\mathcal{Z}_n} - \xi_{\mathcal{Z}_n} (t_{0,n} M_n)^\gamma \\ 
    =& -\frac{\xi_{\mathcal{Z}_n}}{\xi_\ep} b^{(n)} (t_{0,n} M_n)^{\gamma-1}\left[ \left(\frac{t_{0,n} M_n}{b^{(n)}}\right)^{1-\gamma} + (1-\gamma)\left(\frac{t_{0,n} M_n}{b^{(n)}}\right)^{-\gamma} \left(\frac{v}{b^{(n)}}\right)(1+o(1)) \right] +\log\beta_{\mathcal{Z}_n} - \xi_{\mathcal{Z}_n} (t_{0,n} M_n)^\gamma \\
    =& -\frac{\xi_{\mathcal{Z}_n}}{\xi_\ep}(b^{(n)})^{\gamma}(1+o(1)) + \log\beta_{\mathcal{Z}_n} - \xi_{\mathcal{Z}_n} (t_{0,n} M_n)^\gamma \\
    =& -\xi_{\mathcal{Z}_n} M_n^{\gamma} - \frac{(1-\gamma)(b^{(n)})^{\gamma}\xi_{\mathcal{Z}_n}}{\xi_\ep} \log M_n + v\gamma\xi_{\mathcal{Z}_n} M_n^{\gamma-1} -\\
    & \left(\underbrace{\frac{(b^{(n)})^{\gamma}\xi_{\mathcal{Z}_n}}{\xi_\ep}\log\frac{C_n(\theta_{x,n}) \beta_\ep}{\xi_{\mathcal{Z}_n} \gamma b^{(n)}} +\frac{\xi_{\mathcal{Z}_n}}{\xi_\ep}(b^{(n)})^{\gamma}-\log \beta_{\mathcal{Z}_n}}_{\Delta_\gamma^{(n)}} - \log\eta_{\gamma,n} \right)  (1+o(1)).
\end{align*}
\begin{align*}
    &n h_n''(t_{0,n}) \\
    =&  -\gamma\xi_{\mathcal{Z}_n} M_n^{\gamma}t_{0,n}^{\gamma-1} \cdot \xi_\ep\gamma\left(\frac{v+t_{0,n} M_n}{b^{(n)}}\right)^{\gamma-1} \frac{M_n}{b^{(n)}}(1+o(1))\\
    =& -\xi_\ep\xi_{\mathcal{Z}_n} (\gamma M_n)^2 (t_{0,n} M_n)^{\gamma-1} \frac{1}{b^{(n)}} \left[ \left(\frac{t_{0,n} M_n}{b^{(n)}}\right)^{\gamma-1} + (\gamma-1)\left(\frac{t_{0,n} M_n}{b^{(n)}}\right)^{\gamma-2} \left(\frac{v}{b^{(n)}}\right)(1+o(1)) \right]\\
    =&  -\xi_\ep\xi_{\mathcal{Z}_n}(\gamma M_n)^2 M_n^{2(\gamma-1)} (\frac{1}{b^{(n)}})^{\gamma} (1+o(1)).
\end{align*}
Using the convergence of $b^{(n)}$ from Assumption (\ref{assumption2}), convergence of $C_n(\cdot)$ from \ref{asp:L2}, and convergence of $\theta_{x,n}$ from Proposition \ref{prop:theta_nx}, we obtain $\Delta_\gamma^{(n)}\to\Delta_\gamma$, where $\Delta_\gamma=\left(\frac{b^{\gamma}\xi_{\mathcal{Z}}}{\xi_\ep}\log\frac{C_x \beta_\ep}{\xi_{\mathcal{Z}} \gamma b} +\frac{\xi_{\mathcal{Z}}}{\xi_\ep}b^{\gamma}-\log \beta_{\mathcal{Z}} \right)$, $\eta_{\gamma,n}\to\eta_{\gamma}$ where $\eta_{2}=\exp\{-v^2\xi_{\mathcal{Z}}\}$ and $\eta_{\gamma}=1$ for $\gamma\in(0,2)$, and $C_x=\lim_{n\to\ff}C_n(\theta_{x,n}) = \frac{\lambda_U(\theta_x)-1}{\lambda_U(\theta_x)}$. This completes the proof.

\subsection{Proof of Lemma \ref{pf:lem_Gen_2}}\label{app:lem_Gen_2}
Since $\tilde M_n = t_{0,n}M_n$, it follows that
\begin{align*}
    \exp\{-n \phi_n(-\tilde{M}_n)\} [H_n(-\tilde{M}_n)]^{-1} = M_n e^{n h_n(t_{0,n})}\sqrt{\frac{1}{n|h_n''(t_{0,n})|}}.
\end{align*}    
Now, applying Lemma \ref{pf:lem_Gen_1} and using algebraic manipulations, the RHS of the above expression reduces to $[M_n\tilde{R}_{1n}\tilde{R}_{2n}\tilde{R}_{3n}]^{-1} Q_n(x) + o(1)$, as $n \to \ff$, where for $j=1,2,3$, $\tilde{R}_{jn}$ is same as $R_{jn}$ where $b^{(n)}$ and $\mathcal{Z}_n$ are replaced by $b$ and $\mathcal{Z}$, and $Q_n(x)$ is a constant. It follows from \ref{asp:L2}  and $(\log n) |\xi_{\mathcal{Z}_n}-\xi_{\mathcal{Z}}|\to0$ as $n\to\ff$ that $Q_n(x)$ converges to the RHS of (\ref{eq:lem_gen_1}). 

\subsection{Proof of Lemma \ref{pf:lem_Gen_3}}\label{app:lem_Gen_3}
Since $0<\inf_{n,z\in(-\ff,z_0)}\psi_n(z)\leq \sup_{n,z\in(-\ff,z_0)}\psi_n(z)<\ff$, there exist finite positive constants $C_1,C_2$ such that
\begin{align*}
    J_{1n}=&\int_{-\infty}^{-M_n(1+\beta)} e^{-n\phi_n(z)} \psi_n(z) f_{\mathcal{Z}_n}(z) \mathrm{d}z \leq C_1  F_{\mathcal{Z}_n}(-M_n(1+\beta))\\
    \leq& C_2 (\log n)^{\frac{1-\gamma}{\gamma}} n^{-\frac{b^\gamma \xi_Z}{\xi_\ep}(1+\beta)^{\gamma}} = o(R_{1n}R_{2n}R_{3n}), \quad \text{as} ~ n \to \ff.
\end{align*}
Turning to $J_{3n}$, notice that $\phi_n(z)$ is increasing in $z$. Hence, there exist finite positive constants $C_1,C_2,C_3$ such that
\begin{align*}
    J_{3n}=&\int_{-M_n(1-\beta)}^{\infty} e^{-n\phi_n(z)} \psi_n(z) f_{\mathcal{Z}_n}(z) \mathrm{d}z \leq C_1 e^{-n\phi_n(-M_n(1-\beta))}\\
    =& C_1 \exp\{ -\frac{C_2}{[M_n(1-\beta)+v]^{\gamma-1}} n \exp\{-\frac{\xi_\ep}{b^{\gamma}}[v+M_n(1-\beta)]^{\gamma}\} \}\\
    \leq& C_1 \exp\{ -C_3 (\log n)^{(1-\gamma)\gamma^{-1}} \cdot n \cdot n^{-(1-\beta)^\gamma} \} \\
    =&  C_1 \exp\{ -C_3 n^{1-(1-\beta)^\gamma} \cdot (\log n)^{(1-\gamma)\gamma^{-1}} \} = o(R_{1n}R_{2n}R_{3n}), \quad \text{as}~ n \ra \ff,
\end{align*}
completing the proof of the Lemma.

\subsection{Proof of Lemma \ref{pf:lem_RR_1}} \label{app:lem_RR_1}

Applying Proposition \ref{prop:CRateFunction}, it follows that
\begin{align*}
    h_n(t)=& \frac{1}{n}[-n\phi_n(-M_n^{t})+ \log f_{\mathcal{Z}_n}(-M_n^{t}) + t \log M_n]\\
    =& \frac{1}{n}\left[-n C_n(\theta_{x,n})\left(\frac{v+ M_n^{t}}{b^{(n)}}\right)^{-\alpha_\epsilon} L_\epsilon(M_n^t)(1+o(1)) + (\log L_{\mathcal{Z}_n}(M_n^t) -\alpha_{\mathcal{Z}_n} t \log M_n + \log\alpha_{\mathcal{Z}_n})(1+o(1)) \right] \\
    h_n'(t) =& \frac{1}{n} \left[ n C_n(\theta_{x,n})\alpha_\epsilon\left(\frac{v+ M_n^{t}}{b^{(n)}}\right)^{-\alpha_\epsilon-1}\frac{1}{b^{(n)}}M_n^t(\log M_n) L_\epsilon(M_n^t) \right.\\
    & \left. - n C_n(\theta_{x,n})\left(\frac{v+ M_n^{t}}{b^{(n)}}\right)^{-\alpha_\epsilon}L_\epsilon'(M_n^t)M_n^t(\log M_n) + \frac{L_{\mathcal{Z}_n}'(M_n^t)M_n^t(\log M_n)}{L_{\mathcal{Z}_n}(M_n^t)} - \alpha_{\mathcal{Z}_n} \log M_n \right] \\
    =& \frac{1}{n} \left\{ n C_n(\theta_{x,n})\alpha_\epsilon\left(\frac{v+ M_n^{t}}{b^{(n)}}\right)^{-\alpha_\epsilon-1}\frac{1}{b^{(n)}}M_n^t(\log M_n) L_\epsilon(M_n^t) \left[1-\frac{b^{(n)}}{\alpha_\epsilon}\left(\frac{v+ M_n^{t}}{b^{(n)}}\right)\frac{L_\epsilon'(M_n^t)}{L_\epsilon(M_n^t)}\right] \right.\\
    & \left.+ \frac{L_{\mathcal{Z}_n}'(M_n^t)M_n^t(\log M_n)}{L_{\mathcal{Z}_n}(M_n^t)} - \alpha_{\mathcal{Z}_n} \log M_n \right\} \\
    & \text{notice that $\frac{f_{\mathcal{Z}_n}'(z)}{f_{\mathcal{Z}_n}(z)}\approx (\alpha_{\mathcal{Z}_n}+1)|z|^{-1}$ and $\lim_{x\to\infty}\frac{x L_\epsilon'(x)}{L_\epsilon(x)}=0$,}\\
    \approx& \frac{1}{n} \left[n C_n(\theta_{x,n})\alpha_\epsilon\left(\frac{v+ M_n^{t}}{b^{(n)}}\right)^{-\alpha_\epsilon-1}\frac{1}{b^{(n)}}M_n^t(\log M_n) L_\epsilon(M_n^t) + \frac{L_{\mathcal{Z}_n}'(M_n^t)M_n^t(\log M_n)}{L_{\mathcal{Z}_n}(M_n^t)} - \alpha_{\mathcal{Z}_n} \log M_n \right] \\
    =& \frac{\log M_n}{n} \left[n C_n(\theta_{x,n})\alpha_\epsilon\left(\frac{v+ M_n^{t}}{b^{(n)}}\right)^{-\alpha_\epsilon-1}\frac{1}{b^{(n)}}M_n^t L_\epsilon(M_n^t) + \frac{L_{\mathcal{Z}_n}'(M_n^t)M_n^t}{L_{\mathcal{Z}_n}(M_n^t)} - \alpha_{\mathcal{Z}_n} \right]\\
    \approx& \frac{\log M_n}{n} \left[n C_n(\theta_{x,n})\alpha_\epsilon\left(\frac{v+ M_n^{t}}{b^{(n)}}\right)^{-\alpha_\epsilon-1}\frac{1}{b^{(n)}}M_n^t L_\epsilon(M_n^t) - \alpha_{\mathcal{Z}_n} \right] \quad \text{(By $\lim_{x\to\infty}\frac{x L_\epsilon'(x)}{L_\epsilon(x)}=0$ and $t>0$)}\\
    h_n''(t) \approx & \frac{\log M_n}{n} \frac{\mathrm{d}}{\mathrm{d}t}\left[n C_n(\theta_{x,n})\alpha_\epsilon\left(\frac{M_n^{t}}{b^{(n)}}\right)^{-\alpha_\epsilon} L_\epsilon(M_n^t) - \alpha_{\mathcal{Z}_n} \right] \quad\text{(approximation from equation (\ref{h'_RV_2}))}\\
    \approx& \frac{(\log M_n)^2 M_n^t \alpha_\epsilon}{n b^{(n)}} \left[-n C_n(\theta_{x,n})\alpha_\epsilon \left(\frac{M_n^{t}}{b^{(n)}}\right)^{-\alpha_\epsilon-1} L_\epsilon(M_n^t)  \right]\\
    =& \frac{(\log M_n)^2 \alpha_\epsilon}{n} \left[-n C_n(\theta_{x,n})\alpha_\epsilon \left(\frac{M_n^{t}}{b^{(n)}}\right)^{-\alpha_\epsilon} L_\epsilon(M_n^t)  \right] <0 \quad \text{(By $\lim_{x\to\infty}\frac{x L_\epsilon'(x)}{L_\epsilon(x)}=0$)}
\end{align*}
\begin{align}
    h_n'(t_{0,n})=0 \Longleftrightarrow&  n C_n(\theta_{x,n})\alpha_\epsilon\left(\frac{v+ M_n^{t_{0,n}}}{b^{(n)}}\right)^{-\alpha_\epsilon-1}\frac{1}{b^{(n)}}M_n^{t_{0,n}} L_\epsilon(M_n^{t_{0,n}}) = \alpha_{\mathcal{Z}_n} \label{h'_RV_1}\\
    \Longleftrightarrow&  n C_n(\theta_{x,n})\alpha_\epsilon \left(\frac{M_n^{t_{0,n}}}{b^{(n)}}\right)^{-\alpha_\epsilon-1} \left(1+\frac{v}{M_n^{t_{0,n}}}\right)^{-\alpha_\epsilon-1}\frac{1}{b^{(n)}}M_n^{t_{0,n}} L_\epsilon(M_n^{t_{0,n}}) = \alpha_{\mathcal{Z}_n} \nonumber\\
    \Longleftrightarrow&  n C_n(\theta_{x,n})\alpha_\epsilon \left(\frac{M_n^{t_{0,n}}}{b^{(n)}}\right)^{-\alpha_\epsilon} \left(1+\frac{v}{M_n^{t_{0,n}}}\right)^{-\alpha_\epsilon-1} L_\epsilon(M_n^{t_{0,n}}) = \alpha_{\mathcal{Z}_n} \nonumber\\
    \Longleftrightarrow&  n C_n(\theta_{x,n})\alpha_\epsilon \left(\frac{M_n^{t_{0,n}}}{b^{(n)}}\right)^{-\alpha_\epsilon} \left(1-\frac{(\alpha_\epsilon+1)v}{M_n^{t_{0,n}}}+ \frac{(\alpha_\epsilon+1)(\alpha_\epsilon+2)v^2}{2 M_n^{2t_{0,n}}} + \cdots \right) L_\epsilon(M_n^{t_{0,n}}) = \alpha_{\mathcal{Z}_n} \nonumber\\
    \Longleftrightarrow&  n C_n(\theta_{x,n})\alpha_\epsilon \left(\frac{M_n^{t_{0,n}}}{b^{(n)}}\right)^{-\alpha_\epsilon} L_\epsilon(M_n^{t_{0,n}}) = \alpha_{\mathcal{Z}_n} \quad\text{(for large $n$)} \label{h'_RV_2}\\
    \Longleftrightarrow&  t_{0,n} = \frac{1}{\alpha_\epsilon \log M_n}\left[ \log n + \log L_\epsilon(M_n^{t_{0,n}}) + \log\frac{C_n(\theta_{x,n})\alpha_\epsilon (b^{(n)})^{\alpha_\epsilon}}{\alpha_{\mathcal{Z}_n}} \right]\\
    \quad\approx&1+\frac{\log L_\epsilon(M_n)}{\alpha_\epsilon \log M_n} + \frac{\log\frac{C_n(\theta_{x,n})\alpha_\epsilon (b^{(n)})^{\alpha_\epsilon}}{\alpha_{\mathcal{Z}_n}}}{\alpha_\epsilon \log M_n} \nonumber
\end{align}

Since $t_{0,n}$ is the root of $h'(t)=0$, and then by equation (\ref{h'_RV_1})
\begin{align*}
    n\phi_n(-M_n^{t_{0,n}}) =& n C_n(\theta_{x,n})\left(\frac{v+ M_n^{t_{0,n}}}{b^{(n)}}\right)^{-\alpha_\epsilon} L_\epsilon(M_n^{t_{0,n}}) = \frac{\alpha_{\mathcal{Z}_n} b^{(n)}}{\alpha_\epsilon M_n^{t_{0,n}}}\left(\frac{v+ M_n^{t_{0,n}}}{b^{(n)}}\right) = \frac{v \alpha_{\mathcal{Z}_n}}{\alpha_\epsilon M_n^{t_{0,n}}} + \frac{\alpha_{\mathcal{Z}_n}}{\alpha_\epsilon}.
\end{align*}  
\begin{align*}
    \log M_n^{t_{0,n}} =& \frac{\log n}{\alpha_\epsilon} + \frac{\log\frac{C_n(\theta_{x,n})\alpha_\epsilon (b^{(n)})^{\alpha_\epsilon}}{\alpha_{\mathcal{Z}_n}}}{\alpha_\epsilon} + \frac{\log L_\epsilon(M_n^{t_{0,n}})}{\alpha_\epsilon} \quad\text{(by  (\ref{h'_RV_2}))}
\end{align*}  
\begin{align*}
    h_n(t_{0,n}) =&\frac{1}{n}\left[ -n\phi_n(-M_n^{t_{0,n}}) + \log L_{\mathcal{Z}_n}(M_n^t) -\alpha_{\mathcal{Z}_n} t \log M_n + \log\alpha_{\mathcal{Z}_n} \right]\\
    =& \frac{1}{n}\left[ -\frac{v \alpha_{\mathcal{Z}_n}}{\alpha_\epsilon M_n^{t_{0,n}}} - \frac{\alpha_{\mathcal{Z}_n}}{\alpha_\epsilon} + \log L_{\mathcal{Z}_n}(M_n^{t_{0,n}}) -\alpha_{\mathcal{Z}_n}(\frac{\log n}{\alpha_\epsilon} + \frac{\log\frac{C_n(\theta_{x,n})\alpha_\epsilon (b^{(n)})^{\alpha_\epsilon}}{\alpha_{\mathcal{Z}_n}}}{\alpha_\epsilon} + \frac{\log L_\epsilon(M_n^{t_{0,n}})}{\alpha_\epsilon}) + \log\alpha_{\mathcal{Z}_n}\right] \\
    =& \frac{1}{n}\left[ - \frac{\alpha_{\mathcal{Z}_n}}{\alpha_\epsilon} + \log L_{\mathcal{Z}_n}(M_n) -\frac{\alpha_{\mathcal{Z}_n}}{\alpha_\epsilon}\log n + \frac{\alpha_{\mathcal{Z}_n}}{\alpha_\epsilon}\log\frac{C_n(\theta_{x,n})\alpha_\epsilon (b^{(n)})^{\alpha_\epsilon}}{\alpha_{\mathcal{Z}_n}}+ \frac{\alpha_{\mathcal{Z}_n}}{\alpha_\epsilon}\log L_\epsilon(M_n) + \log\alpha_{\mathcal{Z}_n} +o(1)\right]\\
    =& \frac{1}{n}\left[ -\frac{\alpha_{\mathcal{Z}_n}}{\alpha_\epsilon}\log n + \log L_{\mathcal{Z}_n}(M_n) + \frac{\alpha_{\mathcal{Z}_n}}{\alpha_\epsilon}\log L_\epsilon(M_n) + \underbrace{\frac{\alpha_{\mathcal{Z}_n}}{\alpha_\epsilon}\log\frac{C_n(\theta_{x,n})\alpha_\epsilon (b^{(n)})^{\alpha_\epsilon}}{\alpha_{\mathcal{Z}_n}} + \log\alpha_{\mathcal{Z}_n} - \frac{\alpha_{\mathcal{Z}_n}}{\alpha_\epsilon}}_{\Delta_n} +o(1)\right]\\
    h_n''(t_{0,n}) =& \frac{(\log M_n)^2 \alpha_\epsilon}{n} \left[-n C_n(\theta_{x,n})\alpha_\epsilon \left(\frac{M_n^{t_{0,n}}}{b^{(n)}}\right)^{-\alpha_\epsilon} L_\epsilon(M_n^{t_{0,n}}) +o(1) \right] = -\frac{(\log M_n)^2 \alpha_\epsilon \alpha_{\mathcal{Z}_n}+o(1)}{n} \quad\text{(by  (\ref{h'_RV_2}))}.
\end{align*}
Using the convergence of $b^{(n)}$ from Assumption (\ref{assumption3}), convergence of $C_n(\cdot)$ from \ref{asp:L2}, convergence of $\alpha_{\mathcal{Z}_n}$ from assumption of the proposition, and convergence of $\theta_{x,n}$ from Proposition \ref{prop:theta_nx}, we obtain $\Delta_n\to\Delta$, where $\Delta=\frac{\alpha_{\mathcal{Z}}}{\alpha_\epsilon}\log\frac{C_x\alpha_\epsilon b^{\alpha_\epsilon}}{\alpha_{\mathcal{Z}}} + \log\alpha_{\mathcal{Z}} - \frac{\alpha_{\mathcal{Z}}}{\alpha_\epsilon}$, and $C_x=\lim_{n\to\ff}C_n(\theta_{x,n}) = \frac{\lambda_U(\theta_x)-1}{\lambda_U(\theta_x)}$. This completes the proof.

\subsection{Proof of Lemma \ref{pf:lem_RR_2}} \label{app:lem_RR_2}
Notice that $\tilde M_n = M_n^{t_{0,n}}$, applying Lemma \ref{pf:lem_RR_1}, we obtain
\begin{align*}
    \exp(-n \phi_n(-\tilde{M}_n) [H_n(-\tilde{M}_n)]^{-1} =& (\log M_n) e^{n h_n(t_{0,n})}\sqrt{\frac{1}{n|h_n''(t_{0,n})|}}\\
    =& n^{-\frac{\alpha_{\mathcal{Z}_n}}{\alpha_\epsilon}} [L_\epsilon(n^{\frac{1}{\alpha_\epsilon}})]^{-\frac{\alpha_{\mathcal{Z}_n}}{\alpha_\epsilon}} L_{\mathcal{Z}_n}(n^{\frac{1}{\alpha_\epsilon}}) e^{\Delta}\left(\frac{1}{\alpha_\epsilon \alpha_{\mathcal{Z}_n}}\right)^{\frac{1}{2}}(1+o(1)).
\end{align*}
Using the assumptions that $(\log n) |\alpha_{\mathcal{Z}_n}-\alpha_{\mathcal{Z}}|\to0$ and $L_{\mathcal{Z}_n}(n^{\frac{1}{\alpha_\ep}}) \left[L_{\mathcal{Z}}(n^{\frac{1}{\alpha_\ep}})\right]^{-1} \to1$ as $n\to\ff$, we can replace $\mathcal{Z}_n$ by $\mathcal{Z}$ and $L_{\mathcal{Z}_n}(n^{\frac{1}{\alpha_\ep}})$ by $L_{\mathcal{Z}}(n^{\frac{1}{\alpha_\ep}})$. We complete the proof by Setting $K_x = e^{\Delta} (\alpha_\epsilon \alpha_{\mathcal{Z}})^{-\frac{1}{2}}$. 

\subsection{Proof of Lemma \ref{pf:lem_RR_3}} \label{app:lem_RR_3}
Since $0<\inf_{n,z\in(-\ff,z_0)}\psi_n(z)\leq \sup_{n,z\in(-\ff,z_0)}\psi_n(z)<\ff$, there exist finite positive constants $C_1,C_2$ such that as $n\to\ff$,
\begin{align*}
    J_{1n}'=&\int_{-\infty}^{-M_n n^{\beta}} e^{-n\phi_n(z)} \psi_n(z) f_{\mathcal{Z}_n}(z) \mathrm{d}z \leq C_1  F_{\mathcal{Z}_n}(-M_n n^{\beta})\\
    \leq& C_2 n ^{-\frac{\alpha_Z}{\alpha_\epsilon}-\beta\alpha_Z} =  o\left((n^{-\frac{\alpha_Z}{\alpha_\epsilon}} [L_\epsilon(n^{\frac{1}{\alpha_\epsilon}})]^{-\frac{\alpha_Z}{\alpha_\epsilon}} L_Z(n^{\frac{1}{\alpha_\epsilon}})\right).
\end{align*}
Turning to $J_{3n}$, notice that $\phi_n(z)$ is increasing in $z$. Hence,  there exist finite positive constants $C_1,C_2,C_3$ such that as $n\to\ff$,
\begin{align*}
    J_{3n}'=&\int_{-M_n n^{-\beta}}^{\infty} e^{-n\phi_n(z)} \psi_n(z) f_{\mathcal{Z}_n}(z) \mathrm{d}z \leq C_1 e^{-n\phi_n(-M_n n^{-\beta})}\\
    \leq&  C_2 \exp\{ -n C_x \left( \frac{v+n^{\frac{1}{\alpha_\epsilon}-\beta}}{b} \right)^{-\alpha_\epsilon} L_\epsilon(n^{\frac{1}{\alpha_\epsilon}-\beta}) \}\\
    \leq& C_3 \exp\{ -n^{\beta\alpha_\epsilon}   L_\epsilon(n^{\frac{1}{\alpha_\epsilon}-\beta}) \} =  o\left(n^{-\frac{\alpha_Z}{\alpha_\epsilon}} [L_\epsilon(n^{\frac{1}{\alpha_\epsilon}})]^{-\frac{\alpha_Z}{\alpha_\epsilon}} L_Z(n^{\frac{1}{\alpha_\epsilon}})\right),
\end{align*}
completing the proof of the Lemma.

\subsection{Proof of Lemma \ref{lem:loc-logsmooth}} \label{app:loc-logsmooth}
The proof is similar
to the argument used in Theorem~\ref{thm:LDP_ndZ} (see Remark~\ref{rem:log-smooth_proof}). We provide a brief sketch. By the envelope theorem,
\[
\partial_z\Lambda^*(x,z)\;=\;-\partial_z\Lambda(\theta_x(z),z)
\;=\;-\frac{\lambda_U(\theta_x(z))-1}{\lambda_U(\theta_x(z))\,p(z)+1-p(z)}\,p'(z),
\]
with $p(z)=F_\epsilon((v-z)/b)$ and $p'(z)=-b^{-1}f_\epsilon((v-z)/b)<0$. Since $\theta_x(z)>0$
and $\lambda_U(\theta)>1$ for $\theta>0$, we get $\partial_z\Lambda^*(x,z)>0$, i.e.\ strict increase in $z$.
In the unbounded-left regime ($\kappa=-\infty$), $\Lambda^*(x,z)\downarrow\Lambda_U^*(x)$ as $z\to-\infty$.

\subsection{Statement and proof of Lemma \ref{lem:saddle-consistency}}\label{sec:A4-saddle-consistency}

\begin{lem}[Consistency of $M_n$ and $\tilde M_n$]\label{lem:saddle-consistency}
Assume \textup{\ref{asp:D1}}, \textup{\ref{asp:D2}}, \textup{\ref{asp:B1}}, and {\rm\ref{asp:W1}-\ref{asp:W3}}. Let $M_n$ be the balance scale from \textup{\ref{asp:B1}} and 
$z_n^*=-\tilde M_n$ the maximizer of $\tilde h_n(z):= -\phi_n(z) + n^{-1}\log f_{\mathcal{Z}_n}(z)$. Then
\[
\tilde M_n = M_n\{1+o(1)\},\qquad r_{\mathcal{Z}}(M_n)\big(\tilde M_n - M_n\big) \to 0,
\]
hence $z_n^*=-\tilde M_n = -M_n + o\big(1/r_{\mathcal{Z}}(M_n)\big)$.
\end{lem}

\begin{proof}
Differentiate $\tilde h_n$:
\[
\tilde h_n'(z) = -\phi_n'(z) + \tfrac{1}{n}\,\frac{f'_{\mathcal{Z}_n}(z)}{f_{\mathcal{Z}_n}(z)}.
\]
Evaluating at $z=-M_n$ and using the balance ratio in \textup{\ref{asp:B1}} gives 
$\tilde h_n'(-M_n)=o(r_{\mathcal{Z}}(M_n))$. By the log-smooth assumptions \textup{\ref{asp:D1}}–\textup{\ref{asp:D2}}, and balance condition \ref{asp:B1}, we also have 
$\tilde h_n''(-M_n)\sim -\,r_{\mathcal{Z}}(M_n)^2<0$. By the mean value theorem applied to the full exponent, there exists $\xi_n\in(-\tilde M_n,-M_n)$ such that
\[
0=(n\tilde h_n)'(-\tilde M_n)
=(n\tilde h_n)'(-M_n)+(n\tilde h_n)''(\xi_n)\,(\tilde M_n-M_n).
\]
Hence
\[
\tilde M_n-M_n
=-\frac{(n\tilde h_n)'(-M_n)}{(n\tilde h_n)''(\xi_n)}.
\]
From \textup{\ref{asp:B1}} we have $(n\tilde h_n)'(-M_n)=o \big(r_{\mathcal{Z}}(M_n)\big)$, and by \textup{\ref{asp:D1}}–\textup{\ref{asp:D2}} the curvature satisfies
$(n\tilde h_n)''(\xi_n)\sim -\,n\,r_{\mathcal{Z}}(M_n)^2$ uniformly along $(-\tilde M_n,-M_n)$, so
\[
r_{\mathcal{Z}}(M_n)\,(\tilde M_n-M_n)\ \longrightarrow\ 0.
\]
Finally, since $M_n\uparrow\infty$ by \textup{\ref{asp:B1}} and $r_{\mathcal{Z}}(w)\uparrow\infty$ as $w\to\infty$ by \textup{\ref{asp:D2}}, we have
$r_{\mathcal{Z}}(M_n)M_n\to\infty$; therefore
\[
\frac{\tilde M_n-M_n}{M_n}
=o\Big(\frac{1}{r_{\mathcal{Z}}(M_n)M_n}\Big)\ \longrightarrow\ 0,
\]
i.e.\ $\tilde M_n = M_n\{1+o(1)\}$. This completes the proof.

\end{proof}

\begin{rem}
The proof repeats the displacement/window arguments already used for the GN and RV cases: see 
Lemma~5.1 and Lemma~5.4 (Newton step around the approximate maximizer; curvature of order $-r_Z^2$), 
now applied to $\tilde h_n$.
\end{rem}

\section{Appendix}\label{app:B}

\subsection{Proof of Proposition \ref{prop:CMGF}}

\begin{enumerate}
    \item Let $z$ be fixed. Then
    \begin{enumerate}
        \item $\Lambda_n(\theta;z)$ is the cumulant generating function which is finite using the assumption {\ref{asp:L2}}. The convexity follows from H\"older's inequality and differentiability in $\theta$ follows from the differentiability of $\lambda_{U^{(n)}}(\theta)$.
        
        \item Since $\Lambda_n(\theta;z)$ is strictly convex and differentiable, $\frac{\partial^2}{\partial \theta^2}\Lambda_n(\theta;z)>0$; hence $\frac{\partial}{\partial \theta}\Lambda_n(\theta;z)$ is strictly increasing in $\theta$.
        \item The formula for the derivative, $\frac{\partial}{\partial \theta}\Lambda_n(\theta;z)$, is in Appendix (Proposition \ref{app:prop_lam_n}).
    \end{enumerate}
    \item Let $\theta$ be fixed.
    \begin{enumerate}
        \item By model assumption,  $\epsilon$ is a random variable with continuous density, the differentiability follows by the differentiability of $F_{\ep}(\cdot)$.
        \item Since $U^{(n)}>0$, we have $\lambda_{U^{(n)}}(\theta)>0$; hence $\frac{\partial}{\partial z}\Lambda_n(\theta;z)<0$
        \item Directly from the formula of $\Lambda_n(\theta;z)$.
    \end{enumerate}
\item Let $g_n(M_n)=1-F_{\epsilon}(\frac{v+M_n}{b^{(n)}})$ and $C_n(\theta)=\dfrac{\la_{U^{(n)}}(\theta)-1}{\la_{U^{(n)}}(\theta)}$. First note that
\begin{align*}
    \Lambda_n(\theta;-M_n) =& \log\left[ \la_{U^{(n)}}(\theta)F_{\epsilon}(\frac{v+M_n}{b^{(n)}}) + 1- F_{\epsilon}(\frac{v+M_n}{b^{(n)}}) \right]\\
    =& \log\left[ \lambda_{U^{(n)}}(\theta)\right] - g_n(M_n)C_n(\theta) - \frac{1}{2}[g_n(M_n)]^2[C_n^*(\theta)]^2,
\end{align*}
where $C_n^*(\theta)$ is a point between $\dfrac{\la_{U^{(n)}}(\theta)-1}{\la_{U^{(n)}}(\theta)}$ and $\dfrac{\la_{U^{(n)}}(\theta)-1}{\la_{U^{(n)}}(\theta)+ g_n(M_n)(1-\la_{U^{(n)}}(\theta))}.$ Hence it follows that
\begin{align*}
\Lambda_n(\theta;-M_n) = \Lambda_{U^{(n)}}(\theta) + \delta_n(\theta,M_n),
\end{align*}
where $\delta_n(\theta,M_n) = -g_n(M_n)C_n(\theta) - \frac{1}{2}[g_n(M_n)]^2[C_n^*(\theta)]^2$. $|C_n^*(\theta)|<\infty,|C_n(\theta)|<\infty$ uniformly for any $n,\theta$. 

\item The proof follows from the assumption (\ref{assumption1})-(\ref{assumption3}) and {\ref{asp:L2}}.
\end{enumerate}

\subsection{Proof of Proposition \ref{prop:theta_nx}}

Let
\[
G_n(\theta,z)=\frac{\partial}{\partial\theta}\Lambda_n(\theta;z)-x.
\]
Since $\frac{\partial^2}{\partial\theta^2}\Lambda_n(\theta;z)>0$, by the implicit function theorem, $\theta_{n,x}(z)$ is differentiable, and 
\[
\frac{d\theta_{n,x}(z)}{dz}=-\frac{\frac{\partial^2}{\partial z\partial\theta}\Lambda_n(\theta;z)}{\frac{\partial^2}{\partial\theta^2}\Lambda_n(\theta;z)}.
\]
Using the  calculations of Proposition \ref{app:prop_lam_n} in the Appendix, the numerator is positive since $f_{\epsilon}\left(\frac{v-z}{b_n}\right)>0$, $\mathbf{E}[U e^{\theta U}]>0$, and $1 - F_{\epsilon}\left(\frac{v-z}{b_n}\right)>0$. The denominator is positive due to the strict convexity of $\Lambda_n(\theta;z)$ in $\theta$, guaranteed by the positivity of the variance term (variance of $U$ under $P_\theta$, where $\frac{d P_{\theta}}{d P}=e^{\theta U - \log \la_{U}(\theta)})$. Hence, 
\[
\frac{d\theta_{n,x}(z)}{dz}>0,
\]
proving strict monotonicity. Turning to the boundary limits, as $z\to-\infty$, we have $\Lambda_n(\theta;z)\to \Lambda_U(\theta)$, hence $\theta_{n,x}(z)\to\theta_x$ with $\Lambda_U'(\theta_x)=x$. As $z\to\infty$, we see $F_{\epsilon}((v-z)/b_n)\to 0$, forcing $\theta_{n,x}(z)\to\infty$ to maintain equality with $x$. Finally, again using 
the implicit function theorem and the continuous differentiability of $\Lambda_n(\theta;z)$, it follows that $\theta_{n,x}(z)\to\theta_x(z)$ as $n\to\ff$. This completes the proof.

\subsection{Proof of Proposition \ref{prop:CRateFunction}}

\textbf{(i)} For fixed \(\theta\), the function \(F_{\epsilon}((v-z)/b_n)\) is strictly decreasing in \( z \). Hence, \(\Lambda_n(\theta;z)\) is strictly decreasing in \( z \). Thus, \(\theta x-\Lambda_n(\theta;z)\) is strictly increasing in \( z \), and taking supremum preserves strict monotonicity.

\noindent\textbf{(ii)} For each fixed \(\theta'\), we have
\[
\lim_{z \to \infty}[\theta' x-\Lambda_n(\theta';z)]=\theta' x.
\]
Thus,
\[
\liminf_{z\to\infty}\Lambda_n^*(x;z)=\liminf_{z\to\infty}\sup_{\theta}\{\theta x-\Lambda_n(\theta;z)\}\ge\theta' x.
\]
Taking supremum over all \(\theta'\), since \( x>0 \), gives the result.

\noindent\textbf{(iii)} Define \(g_n(z)=1-F_{\epsilon}((v-z)/b_n)\). Then
\[
|\Lambda_n(\theta;z)-\Lambda_U(\theta)|=\left|\log  \left[1+g_n(z)\,\dfrac{1-\mathbf{E}[e^{\theta U}]}{\mathbf{E}[e^{\theta U}]})\right]\right|.
\]
Set \(C_{\theta}=|(1-\mathbf{E}[e^{\theta U}])/\mathbf{E}[e^{\theta U}]|\). Choose \(z_{\epsilon}\) sufficiently negative so that for all \(z\le z_{\epsilon}\):
\[
|g_n(z)C_{\theta}|<\delta=1-e^{-\epsilon},
\]
uniformly for all \(n\). Thus, $|\Lambda_n(\theta;z)-\Lambda_U(\theta)|<\epsilon,
$
uniformly in \(n\), proving uniform convergence as \( z\to -\infty \). Taking supremum over \(\theta\) and using convergence of $\theta_{n}(x,z)$ completes the proof.

\noindent\textbf{(iv)}  In the following, the derivatives of $\Lambda_n(\theta;\kappa)$ and $\Lambda(\theta;\kappa)$ are taken with respect to $\theta$. First notice that, $n|\Lambda_n(\theta;\kappa)-\Lambda(\theta;\kappa)|$ is bounded by
\begin{align*} 
nF_{\epsilon}\left(\frac{v-\kappa}{b^{(n)}}\right)|\la_{U^{(n)}}(\theta)-\la_{U}(\theta)| + n \la_{U}(\theta)\left|F_{\epsilon}\left(\frac{v-\kappa}{b^{(n)}}\right)- F_{\epsilon}\left(\frac{v-\kappa}{b}\right)\right| + n\left|F_{\epsilon}\left(\frac{v-\kappa}{b}\right)- F_{\epsilon}\left(\frac{v-\kappa}{b^{(n)}}\right)\right|,
\end{align*}
which is bounded above by
\begin{align*}
nF_{\epsilon}\left(\frac{v-\kappa}{b^{(n)}}\right)|\la_{U^{(n)}}(\theta)-\la_{U}(\theta)| + n C  [b^{(n)}-b] \to 0,
\end{align*}
using Assumption {\ref{asp:L2}}. Next, turning to the derivatives, $n|\Lambda_n'(\theta;\kappa)-\Lambda'(\theta,\kappa)|$ is bounded by
\begin{align*}
nF_{\epsilon}\left(\frac{v-\kappa}{b^{(n)}}\right)|\la_{U^{(n)}}'(\theta)-\la_{U}'(\theta)| + n \la_{U}'(\theta) \left|F_{\epsilon}\left(\frac{v-\kappa}{b^{(n)}}\right)- F_{\epsilon}\left(\frac{v-\kappa}{b}\right)\right|,
\end{align*}
which is bounded above by
\begin{align*}
C_1 n |\la_{U^{(n)}}'(\theta)-\la_{U}'(\theta)| + C_1 n |b^{(n)}-b| \to 0
\end{align*}
where the convergence follows using {\ref{asp:L2}}. Next, using
\begin{align*}
    n\left| \frac{\Lambda_n'(\theta;\kappa)}{\Lambda_n(\theta;\kappa)} - \frac{\Lambda'(\theta;\kappa)}{\Lambda(\theta;\kappa)} \right|
    =& \frac{n|\Lambda_n'(\theta;\kappa)-\Lambda'(\theta;\kappa)|}{\La_n(\theta;\kappa)} - \frac{n\La'(\theta;\kappa)|\La_n(\theta;\kappa)-\La(\theta;\kappa)|}{\La_n(\theta;\kappa)\Lambda(\theta;\kappa)} \to 0,
\end{align*}
uniformly in $\theta$ by {\ref{asp:L2}}. Let $h_n(\theta) = \frac{\La_n'(\theta;\kappa)}{\La_n(\theta;\kappa)}$ and $h(\theta) = \frac{\La'(\theta;\kappa)}{\La(\theta;\kappa)}$, and $\theta_{x,n}$ be the root of $h_n(\theta)=x$ while $\theta_{x}$ is the root of $h(\theta)=x$. Using the mean value theorem, $h'(\theta)\in(0,\infty)$, and $\theta_{x,n}\to\theta_{x}$, it follows that
\begin{align*}
    n|\theta_{x,n}-\theta_{x}| = \frac{n}{h'(\theta_n^*)}|h(\theta_{x,n})-h(\theta_{x})| \leq C n |h(\theta_{x,n})-h(\theta_{x})|
    = C n |h(\theta_{x,n})-h_n(\theta_{x,n})| \to 0,
\end{align*}
using the uniform convergence above. Therefore,
\begin{align*}
    n|\La^*_{n}(x,\kappa)-\La^*(x,\kappa)| =& x n |\theta_{x,n}-\theta_{x}| - n|h(\theta_{n,x})-h(\theta_{x})| \to 0.
\end{align*}

\noindent\textbf{(v)} Let $ \Delta_\theta=\theta_{x,n}(-M_n)-\theta_{x,n}$, and $\delta_n(\theta,M_n)=\Lambda_n(\theta;-M_n)-\Lambda_{U^{(n)}}(\theta) = -g_n(M_n)C_n(\theta) - \frac{1}{2}[g_n(M_n)]^2[C_n^*(\theta)]^2 = -g_n(M_n)C_n(\theta)(1+o(1))$ using Proposition \ref{prop:CMGF} (iii), where we have suppressed the dependence of $\Delta_{\theta}$ on $n$ and $x$. Notice that, 
\begin{eqnarray*}
\Lambda_n^*(x;-M_n) = (\theta_{x,n}+\Delta_\theta)x - \Lambda_{U}(\theta_{x,n}+\Delta_\theta) - \delta_n(\theta_{x,n}+\Delta_\theta,M_n).
\end{eqnarray*}
Now using a second-order Taylor expansion, the above can be expressed as
\begin{align*}
    (\theta_{x,n}+\Delta_\theta)x - [\Lambda_{U}(\theta_{x,n})+ \Delta_\theta\Lambda_{U}'(\theta_{x,n})+ \frac{1}{2}\Lambda_{U}'(\theta_x^*)\Delta_\theta^2] - [\delta_n(\theta_{x,n},M_n) + \Delta_\theta\delta_n'(\theta_{x,n},M_n) + \frac{1}{2}\delta_n''(\theta_x^{**},M_n)\Delta_\theta^2],
\end{align*}
which reduces to
$ x \theta_{x,n} - \Lambda_{U}(\theta_{x,n}) -\delta_n(\theta_{x,n},M_n) - B_n$, where $B_n = \left[\frac{1}{2}\Lambda_{U}'(\theta_x^*)\Delta_\theta^2 + \Delta_\theta\delta_n'(\theta_x,M_n) + \frac{1}{2}\delta_n''(\theta_x^{**},M_n)\Delta_\theta^2\right]$,  $\theta_x^*\in[\theta_{x,n},\theta_{x,n}+\Delta_\theta]$, and $\theta_x^{**}\in[\theta_{x,n},\theta_{x,n}+\Delta_\theta]$. Next, since $\left[ \theta x - \Lambda_{U}(\theta) - \delta_n(\theta,M_n)\right]$ attains its supremum at $ \Delta_\theta+ \theta_{x,n}$, it follows that $\Delta_\theta$ can be solved by taking derivative of $\Lambda_n^*(x;-M_n)$ with respect to $x$ and setting it to be zero; hence,
\begin{align*}
\Delta_\theta= -\frac{\delta_n'(\theta_x,M_n)}{\Lambda_{U}'(\theta_x^*)+\delta_n''(\theta_x^{**},M_n)} =  O(g_n(M_n)),
\end{align*}    
since $\delta_n(\theta_x,M_n) = O(g_n(M_n))$,  $\delta_n'(\theta_x,M_n) = O(g_n(M_n))$,  and $\delta_n''(\theta_x,M_n) = O(g_n(M_n))$. This implies $\Delta_{\theta}$ converges to 0 as $n \ra \ff$. This implies that $B_n = o(\delta_n(\theta_x,M_n))\to0$ and hence 
\begin{align*}
    \Lambda_n^*(x;-M_n) =& \Lambda_U^*(x)- \delta_n(\theta_x,M_n)(1+o(1)) = \Lambda_U^*(x) + g_n(M_n)C_n(\theta_x)(1+o(1)).
\end{align*}

\subsection{Prefactor stability lemma}

\begin{lem}[Prefactor stability under $O(1/n)$ level shifts]\label{lem:prefactor-stability}
Assume {\rm(\ref{assumption1})-(\ref{assumption3})} and {\rm\ref{asp:L2}-\ref{asp:L3}}. Fix $x>\mu_U$ and let
\[
x_n(u)\ :=\ \frac{n}{n-1}x-\frac{u}{n-1}\,.
\]
Let $A_n(\,\cdot\,)$ denote the Laplace prefactor from Theorems \ref{thm:LDP_ndZ}-\ref{thm:LDP_gen_ndZ}
, i.e.
\[
A_n(x)=\frac{H_n(\tilde M_n(x))}{\sqrt{2\pi n}\theta_x(\tilde M_n(x))\sigma_x(\tilde M_n(x))}
\exp\Big\{n\phi_n\big(\tilde M_n(x);x\big)\Big\}, \quad \text{as $n\to\ff$},
\]
with $\phi_n(z;x):=\Lambda_n^*(x;z)-\Lambda_{U}^*(x)$ and $\partial_z \phi_n(\tilde M_n(x);x)=0$.
Then there exist constants $C<\infty$ and $n_0$ not depending on Q and $n$) such that, for all $n\ge n_0$ and all $u\ge Q$,
\begin{equation}\label{eq:prefactor-ineq}
\frac{A_n(x)}{A_{n-1}\big(x_n(u)\big)}
\ \le\
1\;+\;\frac{C\,u}{n}\;+\;\frac{C\,u^2}{n^2}\,.
\end{equation}
Moreover, the following uniform expansion holds:
\begin{equation}\label{eq:prefactor-expansion}
\log\frac{A_n(x)}{A_{n-1}\big(x_n(u)\big)}
\;=\;\frac{\alpha_0+\alpha_1(u-x)}{n}\;+\;O\Big(\frac{1+u^2}{n^2}\Big),
\qquad (u\ge Q),
\end{equation}
where $\alpha_0,\alpha_1$ are bounded (not depending on $n$ and $Q$), and the $O(\cdot)$ is uniform in $u\ge Q$.
\end{lem}

\begin{proof}
Write $A_n(x)=\bar A_n(x)\,\exp\{n\phi_n(\tilde M_n(x);x)\}$ with
\[
\bar A_n(x):=\frac{H_n(\tilde M_n(x))}{\sqrt{2\pi n}\,\theta_x(\tilde M_n(x))\,\sigma_x(\tilde M_n(x))}.
\]
Set $\Delta_n(u):=x_n(u)-x=(x-u)/(n-1)$. By smoothness in $x$ for all $n$ (implicit function theorem for $\tilde M_n$ and $C^1$ dependence of $H_n,\theta_x,\sigma_x$), a first-order expansion yields
\[
\log \bar A_n(x)-\log \bar A_{n-1}(x_n(u))
=\frac{a_0}{n}+\frac{a_1}{n}\,(u-x)+O\Big(\frac{1+u^2}{n^2}\Big),
\]
with bounded $a_0,a_1$ (uniform for $u\ge Q$). For the exponential part, use
\[
\phi_n(\tilde M_n(x);x)=\Lambda_n^*(x;\tilde M_n(x))-\Lambda_{U}^*(x)
\]
and expand
\[
\begin{aligned}
&n\phi_n(\tilde M_n(x);x)-(n-1)\phi_{n-1}(\tilde M_{n-1}(x_n(u));x_n(u))\\
&\quad=\Big[n\Lambda_n^*(x;\tilde M_n(x))-(n-1)\Lambda_{n-1}^*(x_n(u),\tilde M_{n-1}(x_n(u)))\Big]
-\Big[n\Lambda_{U}^*(x)-(n-1)\Lambda_{U}^*(x_n(u))\Big].
\end{aligned}
\]
The first bracket is $c_0+\frac{b_1}{n}(u-x)+O((1+u^2)/n^2)$ by $C^1$ regularity in $x$ (uniformly in $n$) together with the first-oder expansion in  (\ref{pf:Tayexp}), while the second bracket is
\[
\Lambda_U^*(x)+\theta_x u-\theta_x x+O\Big(\frac{1+u^2}{n}\Big)
\]
by Taylor expansion and $(\Lambda^*)'(x)=\theta_x$. Combining both displays gives
\[
\log\frac{A_n(x)}{A_{n-1}(x_n(u))}
=\frac{\alpha_0+\alpha_1(u-x)}{n}+O\Big(\frac{1+u^2}{n^2}\Big),
\]
Exponentiating and using $e^y\le 1+y+C y^2$ for small $y$ yields \eqref{eq:prefactor-ineq}.
\end{proof}

\begin{lem}[Uniform compact TV for weighted weak limits]\label{lem:B2-weighted-corrected}
Let $B_*=[-Q,Q]$. Assume $U^{(n)}\Rightarrow U$ and $F_U$ has no atoms in $B_*$. Let $\phi_n,\phi:B_*\to[0,\infty)$ be bounded functions with $\|\phi_n-\phi\|_\infty\to0$ and $\phi$ continuous on $B_*$. Define finite measures
\[
\mu_n(A):=\int_{A\cap B_*}\phi_n(u)\,\mathrm{d}F_{U^{(n)}}(u),
\qquad
\mu(A):=\int_{A\cap B_*}\phi(u)\,\mathrm{d}F_U(u).
\]
Then
\[
\sup_{B\subset B_*}\big|\mu_n(B)-\mu(B)\big| \;\longrightarrow\; 0.
\]
\end{lem}

\begin{proof}
Fix $\zeta>0$ and set $M_*:=\sup_{B_*}\phi$. 

\smallskip\noindent\emph{Step 1 (uniform Kolmogorov-type control for fixed weight).}
Define
\[
H_n(x)\ :=\ \int_{(-\infty,x]\cap B_*}\phi(u)\,\mathrm{d}\,\big(F_{U^{(n)}}-F_U\big)(u),\qquad x\in\mathbb R.
\]
Since $\phi$ is bounded and continuous on $B_*$ and $U^{(n)}\Rightarrow U$, the finite measures
$\nu_n(\cdot):=\int_{\cdot\cap B_*}\phi\,\mathrm dF_{U^{(n)}}$ converge weakly to $\nu(\cdot):=\int_{\cdot\cap B_*}\phi\,\mathrm dF_U$, by the Portmanteau Theorem (see \cite{billing13}, Theorem~2.1). Moreover, $\nu$ has no atoms on $\mathbb R$ (because $F_U$ is atomless on $B_*$ and $\phi$ is continuous). Therefore, the \emph{distribution functions} $x\mapsto \nu_n((-\infty,x])$ converge to $x\mapsto \nu((-\infty,x])$ \emph{uniformly} on $\mathbb R$ (uniform convergence of cdfs to a continuous limit). Equivalently,
\[
\|H_n\|_\infty := \sup_{x\in\mathbb R}|H_n(x)| \ \longrightarrow\ 0. \tag{B.2.1}
\]

\smallskip\noindent\emph{Step 2 (fixed finite grid on $B_*$).}
Choose a partition $-Q=x_0<x_1<\dots<x_m=Q$ such that all gridpoints are $F_U$-continuity points (possible since $F_U$ is atomless on $B_*$) and
\[
\max_{1\le j\le m}\ \mu\big((x_{j-1},x_j]\big)\ \le\ \frac{\zeta}{8}.
\]
By weak convergence and $\|\phi_n-\phi\|_\infty\to0$, there exists $N_0$ such that for all $n\ge N_0$,
\[
\max_{1\le j\le m}\ \big|\mu_n\big((x_{j-1},x_j]\big)-\mu\big((x_{j-1},x_j]\big)\big|
\ \le\ \frac{\zeta}{8m}.
\tag{B.2.2}
\]
In particular, for $n\ge N_0$,
\(
\max_j \mu_n((x_{j-1},x_j]) \le \zeta/8 + \zeta/(8m) \le \zeta/4.
\)

\smallskip\noindent\emph{Step 3 (uniform control on the algebra generated by the grid).}
Let $\mathcal A_m$ be the finite algebra generated by the cells $\{(x_{j-1},x_j]\}_{j=1}^m$. For $S\in\mathcal A_m$, write it as a disjoint union of at most $m$ cells, $S=\bigcup_{j\in J}(x_{j-1},x_j]$. Using
\[
\mu_n\big((x_{j-1},x_j]\big)-\mu\big((x_{j-1},x_j]\big)
= H_n(x_j)-H_n(x_{j-1}),
\]
we get for all $n$:
\[
\big|\mu_n(S)-\mu(S)\big|
\ \le\ \sum_{j\in J} \big(|H_n(x_j)|+|H_n(x_{j-1})|\big)
\ \le\ 2m\,\|H_n\|_\infty. 
\]
Pick $N_1$ so large that $\|H_n\|_\infty \le \zeta/(8m)$ for all $n\ge N_1$ (by (B.2.1)). Then, for $n\ge N:=\max\{N_0,N_1\}$,
\[
\sup_{S\in\mathcal A_m}\big|\mu_n(S)-\mu(S)\big|\ \le\ \zeta/4. 
\]

\emph{Step 4 (Approximating $B$ without an $m$–factor leak).}
Fix $\zeta>0$. Because $\mu$ is a finite Borel measure on $\mathbb{R}$ and $B\subset B^*$, there exist
a closed $F\subset B$ and an open $O\supset B$ with $\mu(O\setminus F)\le \zeta/4$. Moreover, by
regularity on $\mathbb{R}$ we may take $F$ and $O$ to be finite unions of disjoint closed/open intervals
with all endpoints chosen to be $F_U$–continuity points. Let $\{x_0<\dots<x_m\}$ be the ordered
collection of these endpoints together with $\{-Q,Q\}$ and form the algebra $\mathcal{A}_m$ of finite
unions of half–open cells $(x_{j-1},x_j]$.

Define $S_B\in\mathcal{A}_m$ so that $F\subset S_B\subset O$ (e.g., take $S_B$ to be the union of all
cells contained in $O$). Then
\[
\mu(B\triangle S_B)\ \le\ \mu(O\setminus F)\ \le\ \zeta/4.
\]
By Step~3 (applied to this $\mathcal{A}_m$) there exists $N_1$ such that for all $n\ge N_1$,
\[
\sup_{S\in\mathcal{A}_m}|\mu_n(S)-\mu(S)|\ \le\ \zeta/4.
\]
Hence, for $n\ge N_1$,
\[
\begin{aligned}
|\mu_n(B)-\mu(B)|
&\le |\mu_n(S_B)-\mu(S_B)|+\mu_n(B\triangle S_B)+\mu(B\triangle S_B)\\
&\le \zeta/4\ +\ \mu_n(B\triangle S_B)+\zeta/4.
\end{aligned}
\]
Finally, by (B.2.2) and the uniform control of cell masses (Step~2) we have
$\mu_n(B\triangle S_B)\le \mu(B\triangle S_B)+\zeta/8\le \zeta/4+\zeta/8$. Combining the displays,
$|\mu_n(B)-\mu(B)|< 2\zeta$ for all $n\ge N_1$. As $\zeta>0$ is arbitrary, the claim follows.
\end{proof}

\begin{cor}[Uniform compact TV for weighted weak limits]\label{cor:B2-weighted}
Let $B_*=[-Q,Q]$. Let $(U^{(n)})$ be real random variables with $U^{(n)}\Rightarrow U$ and $F_U$ atomless on $B_*$ (no point masses in $B_*$).
Let $\phi_n,\phi: B_*\to[0,\infty)$ be bounded with $\|\phi_n-\phi\|_\infty\to 0$ and $\phi$ be continuous on $B_*$. Define finite measures on $\mathbb B(\mathbb R)$ by
\[
\mu_n(A):=\int_{A\cap B_*} \phi_n(u)\, \mathrm dF_{U^{(n)}}(u),\qquad
\mu(A):=\int_{A\cap B_*} \phi(u)\, \mathrm dF_{U}(u).
\]
Then
\[
\sup_{B\subset B_*}\big|\mu_n(B)-\mu(B)\big|\ \longrightarrow\ 0.
\]
\end{cor}

\begin{proof}
Apply Lemma~\ref{lem:B2-weighted-corrected} with $B_*=[-Q,Q]$ and the weights
$\tilde{\phi}_n=\phi_n\mathbf 1_{B_*}$, $\tilde{\phi}=\phi\mathbf 1_{B_*}$; the assumptions
$\| \phi_n-\phi\|_\infty\to0$, continuity of $\phi$ on $B_*$, and $U^{(n)}\Rightarrow U$
ensure its hypotheses. This gives
$\sup_{B\subset B_*}|\mu_n(B)-\mu(B)|\to 0$.
\end{proof}

\begin{rem}[Direct application to Theorem~2.5]\label{rem:B2-apply}
In the proof of Theorem~2.5, set 
\[
\phi_n(u)=e^{\theta_{x,n}(z)u-\Lambda_{U^{(n)}}(\theta_{x,n}(z))}\,
\]
which converges uniformly to $\phi(u)=e^{\theta_x u-\Lambda_U(\theta_x)}$ on $B_*$. Apply Corollary~\ref{cor:B2-weighted} with $B_*=[-Q,Q]$ to conclude
\[
\sup_{B_1\subset B_*}\big|\nu_n(B_1)-\nu_{\theta_x}(B_1)\big|\ \to\ 0.
\]
The same conclusion holds with $n$ replaced by $(n-k)$ for any fixed $k$, since by \ref{asp:L2} and \ref{asp:L3} and strict convexity, we have $\sup_{z \le z_0}|\theta_{x, n-k}(z)-\theta_{x,n}(z)| \ra 0$ as $n \to \ff$ (implicit-function bound) and hence the weights differ by $o(1)$ uniformly on compact $u-$sets.
\end{rem}

\begin{lem}[Uniform tilted tail bound]\label{lem:uniform-tilted-tail}
Under Assumption {\rm\ref{asp:L2}} and {\rm\ref{asp:L3}} there exist $t>0$ and $C<\infty$ such that, for all $Q>0$,
\[
\sup_{n}\ \int_{u>Q} \exp\{\theta_x u-\Lambda_U(\theta_x)\}\,
\Big(1+\tfrac{C\,u}{n}+\tfrac{C\,u^2}{n^2}\Big)\,dF_{U^{(n)}}(u)
\ \le\ C\,e^{-tQ}.
\]
\emph{Consequently}, $\sup_n \nu_n(U^{(n)}_1>Q)\ \le\ C\,e^{-tQ}\to 0$ as $Q\to\infty$.
\end{lem}

\begin{proof}
By Assumption \ref{asp:L2}, there exists $\theta_1 >\theta_x$ such that 
\begin{align}{\label{asp:Uni_bound}}
\sup_{n \ge 1}\bE{e^{\theta\,U^{(n)}}} <\ff \quad \text{and} \quad  \sup_{n \ge 1}\bE{|U^{(n)}|^k\,e^{\theta\,U^{(n)}}} <\ff ~~ \text{for all}~~ \ta < \ta_1, k=1,2.
\end{align}
Let $t >0$ be such that $\theta_x+t < \theta_1$. Now, by Chernoff, 
\[
\int_{u>Q} u^k e^{\theta_x u}\,dF_{U^{(n)}}(u)
\ \le\ e^{-tQ}\, \bm{E}\big[U^{(n) k} e^{(\theta_x+t)U^{(n)}}\big]\quad(k=0,1,2).
\]
The upper bound in the Lemma now follows from (\ref{asp:Uni_bound}). The consequently part in the lemma is the tail term in the main proof with the explicit factor from the $\La^*_{n-1}(x_n(u))/\La^*_n(x)$ ratio multiplied by Lemma~\ref{lem:prefactor-stability}.
\end{proof}

\section{Appendix}\label{app:C}

\subsection{Conditional Central Limit Theorem}

This appendix states the conditional central limit theorem and the conditional Bahadur-Rao estimate used in Section~{\ref{sec:proof}}, together with auxiliary weak–limit arguments. Proposition~\ref{thm:CCLT} (conditional CLT) provides the Gaussian control that underpins the conditional Bahadur-Rao estimate, Theorem~\ref{thm:CBahadurRao}; the latter is the local prefactor input in the Laplace evaluations in Sections~\ref{sec:5.1}–\ref{sec:5.2}. 
Proposition~\ref{thm:wlln_1type} records the marginal LLN/CLT baseline and is used to contrast the large–deviation regime with the near–mean regime (see Section~3 and Appendix~E).

\begin{prop} \label{thm:CCLT}
    $L_n=\sum\limits_{j=1}^{n} T_j^{(n)}$, where $T_j^{(n)}|\mathcal{Z}_n$ are independent. $0<(\sigma_{T_j|\mathcal{Z}_n}^{(n)})^2=\mathbf{Var}[T_j^{(n)}|\mathcal{Z}_n]<\infty$ holds (a.e. w.r.t. the probability measure associated with $\mathcal{Z}_n$) uniformly for all $j,n$, and $\mathbf{E}[T_j^{(n)}|\mathcal{Z}_n]=0$, $s_n^2=\sum\limits_{j=1}^n (\sigma_{T_j|\mathcal{Z}_n}^{(n)})^2$. If 
\begin{align*}
    \lim\limits_{n\to\infty} \frac{\sum_{j=1}^n\mathbf{E}\left[(T_j^{(n)})^2 \cdot\mathbf{1}(|T_j^{(n)}|>\delta s_n) |\mathcal{Z}_n \right]}{s_n^2}\overset{p}{=}0
\end{align*}
holds for all $\delta>0$, then $\lim\limits_{n\to\ff}\bP{\frac{L_n}{s_n}\leq x |\mathcal{Z}_n}=\Phi(x)$, where $\Phi(x)$ is the cdf of Gaussian distribution.
\end{prop}

\begin{proof}
Following \cite{sweeting1989conditional},to show $\frac{L_n}{s_n}|\mathcal{Z}_n \xrightarrow{d} \mathcal{N}(0, 1)$, it is sufficient to show moment generating functions converge; that is,
\begin{align*}
    \left|\mathbf{E}[e^{t\frac{L_n}{s_n}}|\mathcal{Z}_{n}]-e^{\frac{t^2}{2}}\right| \xrightarrow{p} 0.
\end{align*}
We proceed in multiple steps. First, note that $T_j^{(n)} = T_j^{(n)}\cdot \mathbf{1}(|T_j^{(n)}|<\delta s_n) + T_j^{(n)}\cdot \mathbf{1}(|T_j^{(n)}|\geq\delta s_n)$. Since, $\mathbf{E}[T_j^{(n)}|\mathcal{Z}_n]=0$, it follows that $\xi_j^{(n)}(\mathcal{Z}_n) = \mathbf{E}[T_j^{(n)}\cdot \mathbf{1}(|T_j^{(n)}|<\delta s_n)|\mathcal{Z}_n] = -\mathbf{E}[T_j^{(n)}\cdot \mathbf{1}(|T_j^{(n)}|\geq\delta s_n)|\mathcal{Z}_n]$. Now, define
$V_{n,j}=T_j^{(n)}\cdot \mathbf{1}(|T_i^{(n)}|<\delta s_n)-\xi_j^{(n)}(\mathcal{Z}_n)$ and $
W_{n,j}=T_j^{(n)}\cdot \mathbf{1}(|T_i^{(n)}|\geq\delta s_n)+\xi_j^{(n)}(\mathcal{Z}_n)$.
Notice that $\mathbf{E}[V_{n,j}|\mathcal{Z}_n]=\mathbf{E}[W_{n,j}|\mathcal{Z}_n]=0$, and $T_j^{(n)}=V_{n,j}+W_{n,j}$. We will now show that $\frac{\sum_{j=1}^n V_{n,j}}{s_n}|\mathcal{Z}_n \xrightarrow{d} \mathcal{N}(0, 1)$ by verifying that
\begin{align*}
    \left|\mathbf{E}[e^{t\frac{\sum_{j=1}^n V_{n,j}}{s_n}}|\mathcal{Z}_{n}]-e^{\frac{t^2}{2}}\right| \xrightarrow{p} 0.
\end{align*}
Note that $\frac{V_{n,j}}{s_n}$ is a bounded random variable, since
\begin{align*}
    \left| \frac{V_{n,j}}{s_n} \right| =& \left| \frac{T_j^{(n)}\cdot \mathbf{1}(|T_i^{(n)}|<\delta s_n)-\xi_j^{(n)}(\mathcal{Z}_n)}{s_n} \right| = \left| \frac{T_j^{(n)}\cdot \mathbf{1}(|T_i^{(n)}|<\delta s_n)-\mathbf{E}[T_j^{(n)}\cdot \mathbf{1}(|T_j^{(n)}|<\delta s_n)|\mathcal{Z}_n]}{s_n} \right| \leq 2\delta.
\end{align*}
Then conditional moment generating function of $\frac{V_{n,j}}{s_n}$ exists, and is given by
\begin{align*}
    \mathbf{E}[e^{t\frac{V_{n,j}}{s_n}}|\mathcal{Z}_{n}] =& \mathbf{E}\left[ 1+ t\cdot V_{n,j}+ \frac{t^2}{2}\left(\frac{V_{n,j}}{s_n}\right)^2 + \frac{t^3}{6} \left(\frac{V_{n,j}}{s_n}\right)^3 e^{\xi\frac{V_{n,j}}{s_n}}
    \Big|\mathcal{Z}_{n}\right]\\
    =& 1+\frac{t^2}{2}\frac{\mathbf{V}[V_{n,j}|\mathcal{Z}_{n}]}{s_n^2}+ \underbrace{\frac{t^3}{6}\mathbf{E}\left[ \left(\frac{V_{n,j}}{s_n}\right)^3 e^{\xi\frac{V_{n,j}}{s_n}}
    \Big|\mathcal{Z}_{n} \right]}_{R_{n,j}(t,\mathcal{Z}_n)}
\end{align*}
where $\xi\in(0,t)$ if $t>0$ and $\xi\in(t,0)$ if $t<0$. Note that, since
\begin{align*}
    \mathbf{Var}[V_{n,j}|\mathcal{Z}_{n}]\leq \mathbf{V}[T_j^{(n)}|\mathcal{Z}_{n}],
\end{align*}
we have
\begin{align*}
    \left|R_{n,j}(t,\mathcal{Z}_n)\right| =& \left|\frac{t^3}{6}\mathbf{E}\left[ \left(\frac{V_{n,j}}{s_n}\right)^3 e^{\xi\frac{V_{n,j}}{s_n}}
    \Big|\mathcal{Z}_{n} \right]\right| \leq \frac{|t|^3}{6} \mathbf{E}\left[ \left|\frac{V_{n,j}}{s_n}\right|^3 e^{|\xi| \left|\frac{V_{n,j}}{s_n}\right|}
    \Big|\mathcal{Z}_{n} \right]\\
    =& \frac{|t|^3}{6} \mathbf{E}\left[ \left|\frac{V_{n,j}}{s_n}\right| \cdot \left(\frac{V_{n,j}}{s_n}\right)^2 \cdot e^{|\xi| \left|\frac{V_{n,j}}{s_n}\right|}
    \Big|\mathcal{Z}_{n} \right]\\
    & \text{(using $\left|\frac{V_{n,j}}{s_n}\right|\leq 2\delta$)}\\
    \leq& \frac{|t|^3}{6}\cdot 2\delta \cdot e^{|t|\cdot 2\delta} \cdot \mathbf{E}\left[ \left(\frac{V_{n,j}}{s_n}\right)^2
    \Big|\mathcal{Z}_{n} \right]
    \leq \frac{|t|^3}{6}\cdot 2\delta \cdot e^{|t|\cdot 2\delta} \cdot \frac{\mathbf{Var}[T_j^{(n)}|\mathcal{Z}_{n}]}{s_n^2}.
\end{align*}

Now returning to the conditional moment generating function of $\frac{\sum_{j=1}^n V_{n,j}}{s_n}$, since $V_{n,j}$ are independent given $\mathcal{Z}_n$,
\begin{align*}
    \mathbf{E}[e^{t\frac{\sum_{j=1}^n V_{n,j}}{s_n}}|\mathcal{Z}_{n}] = \prod_{j=1}^{n} \mathbf{E}[e^{t\frac{V_{n,j}}{s_n}}|\mathcal{Z}_{n}] =\prod_{j=1}^{n}\left[ 1+\frac{t^2}{2}\frac{\mathbf{V}[V_{n,j}|\mathcal{Z}_{n}]}{s_n^2}+ R_{n,j}(t,\mathcal{Z}_n) \right].
\end{align*}
To show 
\begin{align*}
    \mathbf{E}[e^{t\frac{\sum_{j=1}^n V_{n,j}}{s_n}}|\mathcal{Z}_{n}] \xrightarrow{p} e^{\frac{1}{2}t^2},
\end{align*}
is equivalent to show
\begin{align*}
    \prod_{j=1}^{n}\left[ 1+\frac{t^2}{2}\frac{\mathbf{Var}[V_{n,j}|\mathcal{Z}_{n}]}{s_n^2}+ R_{n,j}(t,\mathcal{Z}_n) \right] \xrightarrow{p} e^{\frac{1}{2}t^2},
\end{align*}
only need to show 
\begin{align*}
    \sum_{j=1}^{n}\log\left[ 1+ \underbrace{\frac{t^2}{2}\frac{\mathbf{Var}[V_{n,j}|\mathcal{Z}_{n}]}{s_n^2}+ R_{n,j}(t,\mathcal{Z}_n)}_{Y_{n,j}(t,\mathcal{Z}_n)} \right] \xrightarrow{p} \frac{1}{2}t^2.
\end{align*}

write $Y_{n,j}(t,\mathcal{Z}_n)=\frac{t^2}{2}\frac{\mathbf{Var}[V_{n,j}|\mathcal{Z}_{n}]}{s_n^2}+ R_{n,j}(t,\mathcal{Z}_n)$. Then need to show
\begin{align*}
    \sum_{j=1}^{n}\log\left[ 1+ {Y_{n,j}(t,\mathcal{Z}_n)} \right] \xrightarrow{p} \frac{1}{2}t^2.
\end{align*}
Note that for fixed $t$, $Y_{n,j}(t,\mathcal{Z}_n)\xrightarrow{a.s.}0$ for all $j,\mathcal{Z}_n$. Using Taylor expansion of $\log(1+x)=x + x^2 K(x)$, where $K(x)=-\frac{1}{2}+\frac{x}{3}-\frac{x^2}{4}+\cdots$, and for $|x|\leq\frac{1}{2}$, $|K(x)|\leq \frac{1}{2} + \frac{1}{3}(\frac{1}{2})+\frac{1}{4}(\frac{1}{2})^2+\cdots \leq (\frac{1}{2})^0+(\frac{1}{2})^1 + (\frac{1}{2})^2 + \cdots\leq 2$.
\begin{align*}
    \log\left[ 1+ {Y_{n,j}(t,\mathcal{Z}_n)} \right] = {Y_{n,j}(t,\mathcal{Z}_n)} + {Y_{n,j}^2(t,\mathcal{Z}_n)}\cdot K({Y_{n,j}(t,\mathcal{Z}_n)}).
\end{align*}
We will now show that the convergence of
\begin{align*}
    \sum_{j=1}^{n}\left[ {Y_{n,j}(t,\mathcal{Z}_n)} + {Y_{n,j}^2(t,\mathcal{Z}_n)}\cdot K({Y_{n,j}(t,\mathcal{Z}_n)}) \right] \xrightarrow{p} \frac{1}{2}t^2.
\end{align*}
This is equivalent to
\begin{align*}
    &\sum_{j=1}^{n} {Y_{n,j}(t,\mathcal{Z}_n)} + \sum_{j=1}^{n} {Y_{n,j}^2(t,\mathcal{Z}_n)}\cdot K({Y_{n,j}(t,\mathcal{Z}_n)}) \xrightarrow{p} \frac{1}{2}t^2\\
    &\underbrace{\sum_{j=1}^{n} \frac{t^2}{2}\frac{\mathbf{V}[V_{n,j}|\mathcal{Z}_{n}]}{s_n^2}}_{H_1(n,\mathcal{Z}_n)}+ \underbrace{\sum_{j=1}^{n} R_{n,j}(t,\mathcal{Z}_n)}_{H_2(n,\mathcal{Z}_n)} + \underbrace{\sum_{j=1}^{n} {Y_{n,j}^2(t,\mathcal{Z}_n)}\cdot K({Y_{n,j}(t,\mathcal{Z}_n)})}_{H_3(n,\mathcal{Z}_n)} \xrightarrow{p} \frac{1}{2}t^2.
\end{align*}
First notice that $H_1(n,\mathcal{Z}_n)\xrightarrow{p}\frac{1}{2}t^2$ follows directly from the Lindeberg assumption; that is,
\begin{align*}
H_1(n,\mathcal{Z}_n)=\sum_{j=1}^{n} \frac{t^2}{2}\frac{\mathbf{Var}[V_{n,j}|\mathcal{Z}_{n}]}{s_n^2} = \frac{t^2}{2}\frac{\sum_{j=1}^{n} \mathbf{E}\left[(T_j^{(n)})^2 \cdot\mathbf{1}(|T_j^{(n)}|\leq\delta s_n) |\mathcal{Z}_n \right]}{s_n^2} \xrightarrow{p} \frac{1}{2}t^2
\end{align*}
We will now show that $H_2(n,\mathcal{Z}_n)\xrightarrow{p}0$. To this end, note that
\begin{align*}
\left|H_2(n,\mathcal{Z}_n)\right| =& \left|\sum_{j=1}^{n} R_{n,j}(t,\mathcal{Z}_n)\right| \leq \sum_{j=1}^{n} \left|R_{n,j}(t,\mathcal{Z}_n)\right| \leq \sum_{j=1}^{n} \frac{|t|^3}{6}\cdot 2\delta \cdot e^{|t|\cdot 2\delta} \cdot \frac{\mathbf{Var}[T_j^{(n)}|\mathcal{Z}_{n}]}{s_n^2}\\
    =& \frac{|t|^3}{6}\cdot 2\delta \cdot e^{|t|\cdot 2\delta} \cdot \frac{\sum_{j=1}^{n} \mathbf{Var}[T_j^{(n)}|\mathcal{Z}_{n}]}{s_n^2}.
\end{align*}
The result follows from the Lindeberg assumption, and  $\delta$
being arbitrarily small. Thus, we have established that $\sum_{j=1}^{n} {Y_{n,j}(t,\mathcal{Z}_n)}\xrightarrow{p} \frac{1}{2}t^2$. We will now show that $H_3(n,\mathcal{Z}_n)\xrightarrow{p}0$. Notice that 
\begin{align*}
    |Y_{n,j}(t,\mathcal{Z}_n)| =& \left|\frac{t^2}{2}\frac{\mathbf{Var}[V_{n,j}|\mathcal{Z}_{n}]}{s_n^2}+ R_{n,j}(t,\mathcal{Z}_n)\right| 
    \leq \frac{t^2}{2}\frac{\mathbf{Var}[V_{n,j}|\mathcal{Z}_{n}]}{s_n^2} + \left| R_{n,j}(t,\mathcal{Z}_n)\right|\\
    \leq& \frac{\mathbf{Var}[V_{n,j}|\mathcal{Z}_{n}]}{s_n^2} (\frac{1}{2}t^2 + \frac{|t|^3}{6}\cdot 2\delta \cdot e^{|t|\cdot 2\delta}) 
    = (\frac{1}{2}t^2 + \frac{|t|^3}{6}\cdot 2\delta \cdot e^{|t|\cdot 2\delta}) \cdot \frac{\mathbf{Var}[V_{n,j}|\mathcal{Z}_{n}]}{s_n^2}  \\
    =&  (\frac{1}{2}t^2 + \frac{|t|^3}{6}\cdot 2\delta \cdot e^{|t|\cdot 2\delta}) \frac{\mathbf{E}\left[(T_j^{(n)})^2 \cdot\mathbf{1}(|T_j^{(n)}|\leq \delta s_n) |\mathcal{Z}_n \right]}{s_n^2}   
    \leq \delta^2 (\frac{1}{2}t^2 + \frac{|t|^3}{6}\cdot 2\delta \cdot e^{|t|\cdot 2\delta})
\end{align*}
For fixed $t$, we can find small enough $\delta$ such that $|Y_{n,j}(t,\mathcal{Z}_n)|\leq\frac{1}{2}$, and then $|K\left(Y_{n,j}(t,\mathcal{Z}_n)\right)|\leq 2$. Thus,
\begin{align*}
    |H_3(n,\mathcal{Z}_n)| =& \left|\sum_{j=1}^{n} Y_{n,j}^2(t,\mathcal{Z}_n)\cdot K({Y_{n,j}(t,\mathcal{Z}_n)})\right| \leq 2\cdot \max_{1\leq j\leq n} \{|Y_{n,j}(t,\mathcal{Z}_n)|\} \cdot \sum_{j=1}^{n} |Y_{n,j}(t,\mathcal{Z}_n)|\\
    \leq& 2\delta^2 (\frac{1}{2}t^2 + \frac{|t|^3}{6}\cdot 2\delta \cdot e^{|t|\cdot 2\delta})\cdot \sum_{j=1}^{n} |Y_{n,j}(t,\mathcal{Z}_n)|\\
    \leq& 2\delta^2 (\frac{1}{2}t^2 + \frac{|t|^3}{6}\cdot 2\delta \cdot e^{|t|\cdot 2\delta}) \cdot (|H_1(n,\mathcal{Z}_n)|+|H_2(n,\mathcal{Z}_n)|)
    \xrightarrow{p} 0.
\end{align*}
Combining the above steps we have proved that
\begin{align*}
    \sum_{j=1}^{n}\log\left[ 1+ {Y_{n,j}(t,\mathcal{Z}_n)} \right] \xrightarrow{p} \frac{1}{2}t^2.
\end{align*}
Hence,
\begin{align*}
    \mathbf{E}[e^{t\frac{\sum_{j=1}^n V_{n,j}}{s_n}}|\mathcal{Z}_{n}] \xrightarrow{p} e^{\frac{1}{2}t^2},
\end{align*}
yielding $\frac{\sum_{j=1}^n V_{n,j}}{s_n}|\mathcal{Z}_n \xrightarrow{d} \mathcal{N}(0, 1)$.
The final step consists in showing that $\frac{\sum_{j=1}^n W_{n,j}}{s_n}|\mathcal{Z}_n \xrightarrow{p} 0$. By Chebyshev's inequality, and that $W_{n,j}$ are independent conditioned on $\mathcal{Z}_n$, and Lindeberg assumption,
\begin{align*}
    \mathbf{P}\left(\left| \frac{\sum_{j=1}^n W_{n,j}}{s_n} \right|\geq \epsilon | \mathcal{Z}_n \right) \leq& \frac{\sum_{j=1}^n \mathbf{Var}[W_{n,j}|\mathcal{Z}_n]}{s_n^2\epsilon^2}
    = \frac{1}{\epsilon^2}\cdot \frac{\sum_{j=1}^n\mathbf{E}\left[(T_j^{(n)})^2 \cdot\mathbf{1}(|T_j^{(n)}|>\delta s_n) |\mathcal{Z}_n \right]}{s_n^2}
    \to 0,
\end{align*}
therefore $\frac{\sum_{j=1}^n W_{n,j}}{s_n}|\mathcal{Z}_n \xrightarrow{p} 0$.
Combining the above steps we get $\lim\limits_{n\to\ff}\bP{\frac{L_n}{s_n}\leq x |\mathcal{Z}_n}=\Phi(x)$, where $\Phi(x)$ is the cdf of Gaussian distribution.
\end{proof}

\begin{cor}[Uniform Berry-Esseen under the tilt]\label{cor:uniform-BE}
Fix a compact $K\subset(\mu_U,\infty)$ and $z_0\in\mathbb R$.
Let $\theta_{x,n}(z)$ solve $\partial_\theta\Lambda_n(\theta,z)=x$ and set
$v^2_{x,n}(z):=\partial_\theta^2\Lambda_n(\theta_{x,n}(z),z)$.
Define the standardized sum under the exponential tilt
\[
T_n(x,z)\ :=\ \frac{S_n(z)-n x}{v_{x,n}(z)\,\sqrt{n}},\qquad
S_n(z)=\sum_{i=1}^n U_i^{(n)}X_i^{(n)}(z).
\]
Under \textnormal{(L1)–(L2)},
there exists $C<\infty$ such that for all large $n$,
\[
\sup_{x\in K}\ \sup_{z\le z_0}\ \sup_{t\in\mathbb R}\,
\big|\mathbf P_{\theta_{x,n}(z)}\big(T_n(x,z)\le t\big)-\Phi(t)\big|
\ \le\ \frac{C}{\sqrt{n}}.
\]
The same bound holds uniformly for $y$ in a shrinking window $\lvert y-x\rvert\le c/n$,
with $\theta_{y,n}(z)$ and $v_{y,n}(z)$ in place of $\theta_{x,n}(z),v_{x,n}(z)$.
\end{cor}

\begin{rem}
The conditional CLT (Proposition~2.1) already provides (i) uniform bounds
$0<c\le v_{x,n}(z)\le C$ and $0<c\le \theta_{x,n}(z)\le C$ on $K\times(-\infty,z_0]$ and
(ii) a uniform Lindeberg control under the tilt.
Keeping the next (cubic) term in the cumulant/Taylor expansion of $\log\mathbf E_\theta e^{it(Y-x)/v}$,
and applying Esseen's smoothing inequality, yields
\[
\sup_{t}\big|\mathbf P_\theta(T_n\le t)-\Phi(t)\big|
\ \le\ \frac{C_{\mathrm{BE}}}{\sqrt{n}}\,
\frac{\mathbf E_\theta\lvert Y-x\rvert^3}{(\mathrm{Var}_\theta(Y))^{3/2}}
\ \le\ \frac{C}{\sqrt{n}},
\]
with constants uniform in $(x,z)\in K\times(-\infty,z_0]$ by \textnormal{(L2)}.
\end{rem}

\subsection{Marginal Limit Distribution}

\begin{prop} \label{thm:wlln_1type}
Assume that (\ref{assumption1})-(\ref{assumption3}) hold. Then,   $n^{-1}L_n$ converges in probability to  $\mu_U F_{\ep}\left(\frac{v-\mathcal{Z}}{b}\right)$, as $n \ra \ff$. Furthermore, under the additional assumption that $\mathcal{Z}_n\overset{a.s.}{\to}\mathcal{Z}$, and $\mathbf{E}[(U_{1}^{(n)}-\mu_{U^{(n)}})^{2+\delta}]\leq C_1<\infty$ for a fixed $\delta>0$, then $n^{-1}L_n$ converges almost surely to  $\mu F_{\ep}\left(\frac{v-\mathcal{Z}}{b}\right)$, as $n \ra \ff$. Furthermore, as $n\to\infty$,
\begin{align*}
    \mathbf{P}\left( \frac{L_n- n \mu_{U^{(n)}} F_{\epsilon}\left( \frac{v-\mathcal{Z}_n}{b^{(n)}} \right)}{\sqrt{n}} \leq x \right) \to \mathbf{E}\left[\Phi\left(\frac{x}{\sigma_{T_1|\mathcal{Z}}^{(\infty)}}\right)\right],
\end{align*}
where $\Phi(x)$ is CDF of standard normal distribution and 
\begin{align*}
\sigma_{T_1|\mathcal{Z}}^{(\infty)}=\sqrt{\sigma_{U}^2 F_{\epsilon}\left( \frac{v-\mathcal{Z}}{b} \right) + \mu_{U}^2 F_{\epsilon}\left(\frac{v-\mathcal{Z}}{b}\right)\cdot \left[ 1-F_{\epsilon}\left(\frac{v-\mathcal{Z}}{b}\right)\right]}.
\end{align*}
\end{prop}

\begin{proof}
Let $\mathcal{F}_n$ denote the sigma-field generated by $\mathcal{Z}_n$.  Let $\mu_U:=\mathbf{E}[U]$ and set $\theta\coloneqq \mu_U F_{\ep}\left(\frac{v-\mathcal{Z}}{b}\right)$. First notice that, $n^{-1}L_n -\ta = T_{n,1}+T_{n,2}$, where $T_{n,1}= \left(n^{-1}L_n - \mathbf{E}\left[n^{-1}L_n|\mathcal{F}_n\right]\right)$ and $T_{n,2}= \left(\mathbf{E}\left[n^{-1}L_n|\mathcal{F}_n\right]-\ta \right)$.
We will first show that $T_{n,2}$ converges to zero in probability. To this end,  observe that using the independence of $U_{j}^{(n)}$ and $X_j^{(n)}$ and noticing that  $\mathcal{Z}_n$ is independent of $\{\ep_j\}$, we have that
$\mathbf{E}[\mathbf{1}(Y_{i}^{(n)}\leq v)|\mathcal{Z}_n] = F_{\epsilon}\left( \frac{v-\mathcal{Z}_n}{b^{(n)}} \right)$ with probability one (w.p.1). Hence, w.p.1,
\begin{align*}
\mathbf{E}[U_j^{(n)}X_j^{(n)}|{\mathcal Z}_n] =& \mathbf{E}[U_j^{(n)}\mathbf{E}[X_j^{(n)}|\mathcal{Z}_n]] = \mu_{U^{(n)}}  F_{\epsilon}\left( \frac{v-\mathcal{Z}_n}{b^{(n)}} \right).
\end{align*}
Now, using (\ref{assumption1})-(\ref{assumption3}), it follows using the continuity of $F_{\ep}(\cdot)$ that
$\mathbf{E}\left[n^{-1}L_n|\mathcal{Z}_n \right]$ converges in probability to $\mu F_{\ep}\left(\frac{v-\mathcal{Z}}{b}\right)$ as $n \ra \ff$. Thus, to complete the proof, it is sufficient to show that $T_{n,1}$ converges to zero in probability as $n \ra \ff$. Towards this, note that 
\begin{align}
\nonumber T_{n,1} =  \frac{1}{n}\sum_{j=1}^{n}U_{j}^{(n)}X_{j}^{(n)} - \frac{1}{n}\sum_{j=1}^{n}\mathbf{E}[U_{j}^{(n)}X_{j}^{(n)}|\mathcal{Z}_n]
\coloneqq T_{n,3} + T_{n,4},
\end{align}
where
\begin{align*}
    T_{n,3} =& \frac{1}{n}\sum_{j=1}^{n}(U_{j}^{(n)}-\mu_{U^{(n)}})X_{j}^{(n)} \quad \text{and} \quad T_{n,4} = \frac{\mu_{U^{(n)}}}{n}\sum_{j=1}^{n}(X_{j}^{(n)}-\mathbf{E}[X_{j}^{(n)}|\mathcal{Z}_n]).
\end{align*}
Also, note that 
\begin{align*}
\mathbf{Var}[T_{n,3}] = \frac{1}{n^2}\mathbf{Var}\left[\mathbf{E}\left[\sum_{j=1}^{n}(U_{j}^{(n)}-\mu_{U^{(n)}})X_{j}^{(n)}|\mathcal{Z}_n\right]\right] + \frac{1}{n^2}\mathbf{E}\left[\mathbf{Var}\left[\sum_{j=1}^{n}(U_{j}^{(n)}-\mu_{U^{(n)}})X_{j}^{(n)}|\mathcal{Z}_n\right]\right].
\end{align*}
Now, using the independence of $\{U_{j}^{(n)}\}$ and $\{X_{j}^{(n)}\}$, the first term of the above expression is zero and the second term reduces to
\begin{align}
\nonumber\mathbf{E}\left[\mathbf{Var}\left[\sum_{j=1}^{n}(U_{j}^{(n)}-\mu_{U^{(n)}})X_{j}^{(n)}|\mathcal{Z}_n\right]\right] =&  \sum_{j=1}^{n}\mathbf{E}\left[\mathbf{Var}\left[(U_{j}^{(n)}-\mu_{U^{(n)}})X_{j}^{(n)}|\mathcal{Z}_n\right]\right]\\
\nonumber =& n \mathbf{E}\left[(U_{1}^{(n)}-\mu_{U^{(n)}})^2\right]\cdot \mathbf{E}\left[(X_{1}^{(n)})^2|\mathcal{Z}_n\right]\\
\le& n (\sigma^{(n)})^2. \label{eq:lln_var}
\end{align}
Now, using Chebyshev's inequality and (\ref{eq:lln_var}), it follows that $T_{n,3}$ converges in probability to zero. Turning to $T_{n,4}$, using similar calculation it follows that $\mathbf{Var}[T_{n,4}]$ is bounded above by $n^{-1}(\mu_{U^{(n)}})^2$, verifying that $T_{n,4}$ converges in probability to zero. Turning to the almost sure convergence, we follow the same notation and methods as above to obtain $T_{n,2}\overset{a.s.}{\to}0$. Turning to $T_{n,1}$, notice that it is  enough to show that for any $\ep>0$,
\begin{align} \label{eq:uc}
P\left(\lim_{n\to\infty}|T_{n,1}|>\ep \right) = \mathbf{E}\left[P\left(\lim_{n\to\infty}|T_{n,1}|>\ep |\mathcal{Z}_n\right)\right] =0.
\end{align}
Decomposing $T_{n,1}=T_{n,3}+T_{n,4}$, where
\begin{align*}
    T_{n,3} =& \frac{1}{n}\sum_{j=1}^{n}(U_{j}^{(n)}-\mu_{U^{(n)}})X_{j}^{(n)} \quad \text{and} \quad T_{n,4} = \frac{\mu_{U^{(n)}}}{n}\sum_{j=1}^{n}(X_{j}^{(n)}-\mathbf{E}[X_{j}^{(n)}|\mathcal{Z}_n].
\end{align*}
we will show each of the terms converges to zero almost surely. To this end,  first note that by Markov's inequality 
\begin{align*}
    \mathbf{P}(|T_{n,3}|>\ep)\leq \frac{\mathbf{E}[|T_{n,3}|^{2+\delta}]}{\ep^{2+\delta}},\quad \mathbf{P}(|T_{n,4}|>\ep)\leq \frac{\mathbf{E}[|T_{n,4}|^{2+\delta}]}{\ep^{2+\delta}}.
\end{align*}
Let $A_{j}^{(n)}=(U_{j}^{(n)}-\mu_{U^{(n)}})X_{j}^{(n)}$. Then $T_{n,3}=\frac{1}{n}\sum_{j=1}^{n}A_{j}^{(n)}$. We first calculate the conditional expectation and then take the expectation on both sides. Applying the martingale version of the Marcinkiewicz–Zygmund inequality (see \cite{chow2012probability}), by conditioning on $\mathcal{Z}_n$ and taking expectations and using Minkowski's inequality, we obtain
\begin{align*}
    \mathbf{E}[|T_{n,3}|^{2+\delta}] \leq& \left(\frac{1}{n}\right)^{2+\delta} B_{2+\delta}\mathbf{E}\left[\left( \sum_{j=1}^{n}(A_{j}^{(n)})^2 \right)^{\frac{2+\delta}{2}}\right] = \left(\frac{1}{n}\right)^{2+\delta} B_{2+\delta}\left(\left\Vert \sum_{j=1}^{n}(A_{j}^{(n)})^2 \right\Vert_{\frac{2+\delta}{2}}\right)^{\frac{2+\delta}{2}}\\
    \leq& \left(\frac{1}{n}\right)^{2+\delta} B_{2+\delta}\left(\sum_{j=1}^{n}\left\Vert (A_{j}^{(n)})^2 \right\Vert_{\frac{2+\delta}{2}}\right)^{\frac{2+\delta}{2}} =  \left(\frac{1}{n}\right)^{2+\delta} B_{2+\delta}\left(n \left(\mathbf{E}\left[ (A_{j}^{(n)})^{2+\delta}\right]\right)^{\frac{2}{2+\delta}}\right)^{\frac{2+\delta}{\delta}}.
\end{align*}
The RHS is $n^{-(1+\delta/2)} B_{2+\delta}\mathbf{E}\left[ (A_{j}^{(n)})^{2+\delta}\right]$,
where $B_{2+\delta}\in(0,\infty)$ only depends on $\delta$. Next, using the independence of $U_{1}^{(n)}$ and $X_{1}^{(n)}$, we obtain
\begin{align*}
    \mathbf{E}[(A_{1}^{(n)})^{2+\delta}|\mathcal{Z}_n] =& \mathbf{E}[(U_{1}^{(n)}-\mu_{U^{(n)}})^{2+\delta}]\cdot \mathbf{E}[(X_{1}^{(n)})^{2+\delta}|\mathcal{Z}_n]\\
    \leq&  \mathbf{E}[(U_{1}^{(n)}-\mu_{U^{(n)}})^{2+\delta}] \leq C_1.
\end{align*}
Now, plugging into $\mathbf{E}[|T_{n,3}|^{2+\delta}|\mathcal{Z}_n]$, we get
\begin{align*}
    \mathbf{E}[|T_{n,3}|^4|\mathcal{Z}_n] \leq \left(\frac{1}{n}\right)^{1+\delta/2} B_{2+\delta} C_1.
\end{align*}
Noticing that $\mathbf{E}[|T_{n,3}|^{2+\delta}|\mathcal{Z}_n]>0$ and taking expectation on both side, we obtain
\begin{align*}
    \mathbf{E}[|T_{n,3}|^{2+\delta}] \leq \left(\frac{1}{n}\right)^{1+\delta/2} B_{2+\delta} C_1.
\end{align*}
Hence 
\begin{align*}
    \sum_{n=1}^{\infty}\mathbf{P}(|T_{n,3}|>\ep)\leq B_{2+\delta} C_1 \sum_{n=1}^{\infty}\left(\frac{1}{n}\right)^{1+\delta/2}< \infty.
\end{align*}
By Borel–Cantelli lemma, it follows that $T_{n,3}\overset{a.s.}{\to}0$. Similar calculation also yields that $T_{n,4}\overset{a.s.}{\to}0$. Combining we conclude that $T_{n,1}\overset{a.s.}{\to}0$.
Turning to the CLT part of the proposition, define
\begin{align*}
    S_n =& \sum_{j=1}^{n} \left[ U_j^{(n)}X_j^{(n)} - \mu_{U^{(n)}}F_{\epsilon}\left(\frac{v-\mathcal{Z}_n}{b^{(n)}}\right) \right] \coloneqq \sum_{j=1}^{n} T_j^{(n)}.
\end{align*}
Notice that
\begin{align*}
    \mathbf{Var}[T_j^{(n)}|\mathcal{Z}_n] =& \mathbf{Var}[U_j^{(n)}X_j^{(n)}|\mathcal{Z}_n]\\
    =&  \mathbf{E}\Big[\mathbf{Var}[U_j^{(n)}X_j^{(n)} | \mathcal{Z}_n, \epsilon_j] |\mathcal{Z}_n \Big] + \mathbf{Var}\Big[\mathbf{E}[U_j^{(n)}X_j^{(n)} | \mathcal{Z}_n, \epsilon_i] |\mathcal{Z}_n \Big]\\
    =&  \mathbf{E}\Big[X_j^{(n)}\mathbf{Var}(U_j^{(n)})| \mathcal{Z}_n \Big] + \mathbf{Var}\Big[ X_j^{(n)} \mathbf{E}[U_j^{(n)}] |\mathcal{Z}_n  \Big] \\
    =&  (\sigma^{(n)})^2 \mathbf{E}\Big[X_j^{(n)}| \mathcal{Z}_n \Big] + (\mu_{U^{(n)}})^2\mathbf{Var}\Big[ X_j^{(n)} |\mathcal{Z}_n \Big] \\
    =& (\sigma^{(n)})^2F_{\epsilon}\left( \frac{v-\mathcal{Z}_n}{b^{(n)}} \right) + (\mu_{U^{(n)}})^2 (\sigma_{X_j^{(n)}|\mathcal{Z}_n}^{(n)})^2,
\end{align*}
where $(\sigma_{X_{j}|\mathcal{Z}_n}^{(n)})^2 = \mathbf{V}[X_j^{(n)}|\mathcal{Z}_n] = F_{\epsilon}\left(\frac{v-\mathcal{Z}_n}{b^{(n)}}\right)\cdot \left( 1-F_{\epsilon}\left(\frac{v-\mathcal{Z}_n}{b^{(n)}}\right)\right)$.
Denote $\mathbf{Var}[T_j^{(n)}|\mathcal{Z}_n]$ by $(\sigma_{T_i|\mathcal{Z}_n}^{(n)})^2$, which is finite since $X_j^{(n)}$ is an indicator. Then  $\mathbf{Var}[L_n|\mathcal{Z}_n]$ is given by $s_n^2$, where
\begin{align*}
    s_n^2 = \sum_{j=1}^{n} (\sigma_{T_j|\mathcal{Z}_n}^{(n)})^2 = n \cdot (\sigma_{T_1|\mathcal{Z}_n}^{(n)})^2.
\end{align*}
Now we show that, conditionally on $\mathcal{Z}_n$, $\frac{S_n}{\sqrt{n}}$ satisfies the Lindeberg conditions. Notice that $\mathbf{E}[S_n|\mathcal{Z}_n]= 0$, and for any $\delta>0$,
\begin{align*} 
    \frac{\sum_{j=1}^{n} \mathbf{E}\left[(T_j^{(n)})^2\cdot \mathbf{1}(|T_j^{(n)}|>\delta s_n)|\mathcal{Z}_n\right]}{s_n^2}
    =& \frac{\mathbf{E}\left[(T_1^{(n)})^2\cdot \mathbf{1}(|T_1^{(n)}|>\delta s_n)|\mathcal{Z}_n\right]}{(\sigma_{T_1|\mathcal{Z}_n}^{(n)})^2}.
\end{align*}
Next  we show that the numerator on the RHS converges to 0. To this end, note that
\begin{align*}
\mathbf{E}\left[(T_1^{(n)})^2\cdot \mathbf{1}(|T_1^{(n)}|>\delta s_n)|\mathcal{Z}_n\right] =& \mathbf{E}\left[\left(U_1^{(n)}X_1^{(n)}-\mu_{U^{(n)}}F_{\ep}\left(\frac{v-\mathcal{Z}_n}{b^{(n)}}\right)\right)^2\cdot \mathbf{1}(|T_1^{(n)}|>\delta s_n)|\mathcal{Z}_n\right]\\
    \leq& 2\mathbf{E}\left[\left(U_1^{(n)}X_1^{(n)}\right)^2\cdot \mathbf{1}(|T_1^{(n)}|>\delta s_n)|\mathcal{Z}_n\right] \\
    +& 2\mathbf{E}\left[\left(\mu_{U^{(n)}}F_{\ep}\left(\frac{v-\mathcal{Z}_n}{b^{(n)}}\right)\right)^2\cdot \mathbf{1}(|T_1^{(n)}|>\delta s_n)|\mathcal{Z}_n\right]\\
    =& 2(\mu_{U^{(n)}})^2\mathbf{E}\left[\left(X_1^{(n)}\right)^2\cdot \mathbf{1}(|T_1^{(n)}|>\delta s_n)|\mathcal{Z}_n\right]\\
    + &  2\left(\mu_{U^{(n)}}F_{\ep}\left(\frac{v-\mathcal{Z}_n}{b^{(n)}}\right)\right)^2\mathbf{E}\left[\mathbf{1}(|T_1^{(n)}|>\delta s_n)|\mathcal{Z}_n\right]
\end{align*}
Now using $|X_1^{(n)}|\leq1$ and $F_{\ep}\left(\frac{v-\mathcal{Z}_n}{b^{(n)}}\right)\leq1$, it follows that the above is bounded by
\begin{align*}
2(\mu_{U^{(n)}})^2\mathbf{E}\left[\mathbf{1}(|T_1^{(n)}|>\delta s_n)|\mathcal{Z}_n\right] + 2\left(\mu_{U^{(n)}}\right)^2\mathbf{E}\left[\mathbf{1}(|T_1^{(n)}|>\delta s_n)|\mathcal{Z}_n\right] \to 0.
\end{align*}
since $\mathbf{E}\left[\mathbf{1}(|T_1^{(n)}|>\delta s_n)|\mathcal{Z}_n\right]\to 0$ as $n\to\infty$.
Thus,
\begin{align*}
\lim_{n\to\infty}\frac{\sum_{j=1}^{n} \mathbf{E}\left[(T_j^{(n)})^2\cdot \mathbf{1}(|T_j^{(n)}|>\delta s_n)|\mathcal{Z}_n\right]}{s_n^2} =& \frac{\lim_{n\to\infty}\mathbf{E}\left[(T_1^{(n)})^2\cdot \mathbf{1}(|T_1^{(n)}|>\delta s_n)|\mathcal{Z}_n\right]}{\lim_{n\to\infty}(\sigma_{T_1|\mathcal{Z}_n}^{(n)})^2} =0.
\end{align*}
And also $\frac{s_n}{\sqrt{n}}=\sigma_{T_1|\mathcal{Z}_n}^{(n)} \overset{P}{\to}\sigma_{T_1|\mathcal{Z}}^{(\infty)}$ by (\ref{assumption3}). Hence, using the CLT in Appendix \ref{app:C} (Proposition \ref{thm:CCLT}) we obtain
\begin{align*}
    \mathbf{P}\left( \frac{L_n- n \mu_{U^{(n)}} F_{\epsilon}\left( \frac{v-\mathcal{Z}_n}{b^{(n)}} \right)}{\sqrt{n}} \leq x \right) =& \mathbf{E}\left[ \mathbf{P}\left( \frac{L_n-n \mu_{U^{(n)}} F_{\epsilon}\left( \frac{v-\mathcal{Z}_n}{b^{(n)}} \right)}{\sqrt{n}} \leq x | \mathcal{Z}_n \right) \right]\\
    =& \mathbf{E}\left[ \mathbf{P}\left( \frac{S_n}{\sqrt{n}} \leq x |\mathcal{Z}_n \right)  \right] \to \mathbf{E}\left[\Phi\left(\frac{x}{\sigma_{T_1|\mathcal{Z}}^{(\infty)}}\right)\right],
\end{align*}
where $\Phi(x)$ is CDF of standard normal distribution and $$\sigma_{T_1|\mathcal{Z}}^{(\infty)}=\sqrt{\sigma_{U}^2 F_{\epsilon}\left( \frac{v-\mathcal{Z}}{b} \right) + \mu_{U}^2 F_{\epsilon}\left(\frac{v-\mathcal{Z}}{b}\right)\cdot \left[ 1-F_{\epsilon}\left(\frac{v-\mathcal{Z}}{b}\right)\right]}.$$ 
This completes the proof.
\end{proof}

\subsection{Sharp Conditional Large Deviations}
The Key ingredient to the  proofs of Theorem \ref{thm:LDP_ndZ} and Theorem \ref{thm:LDP_kappa} is the conditional sharp large deviation Theorem and an  evaluation of the behavior of the conditional rate function in the tails of the distribution of the common factors. This involves a careful decomposition of the integral and requires identification of an optimal point similar to the Laplace method for exponential integrals. We start with conditional sharp large deviations, whose proof is standard and hence omitted. Below, $\Lambda_{n}^{*'}(\cdot, z)$ and $\Lambda_{n}^{*''}(\cdot, z)$ are  derivatives for a fixed $z$. The proof  follows from Proposition~\ref{thm:CCLT} via Berry-Esseen bounds for the conditional triangular array.

\begin{thm}[Conditional Bahadur-Rao LDP] \label{thm:CBahadurRao}
Assume that conditions (\ref{assumption1})-(\ref{assumption3}), {\ref{asp:L2}-\ref{asp:L3}}, (and additional conditions in Theorem \ref{thm:LDP_kappa} for degenerate case) hold. Then for any $x>\mu_U$ ($x>q_{\kappa}$ for the degenerate case)
\begin{align*}
    \mathbf{P}(L_n>nx|\mathcal{Z}_n=z)= n^{-\frac{1}{2}}e^{-n \Lambda_{n}^*(x;z)}\frac{1}{\sqrt{2\pi}\theta_{x,n}(z)\sigma_{x,n}(z)}(1+r_n(z)),
\end{align*}
where $\Lambda_{n}^*(x;z) = \sup_{\theta}\{ \theta x - \Lambda_n(\theta;z) \}$ and $\Lambda_n(\theta;z)$ is the logarithmic moment generating function of $U^{(n)}X^{(n)}$ conditioned on $\mathcal{Z}_n=z$; that is
\begin{align*}
     \Lambda_n(\theta;z) =& \log \bE{\exp(\theta U^{(n)}X^{(n)})|\mathcal{Z}_n=z} = \log\left[ \lambda_{U^{(n)}}(\theta) F_{\epsilon}\left(\frac{v-z}{b^{(n)}}\right) + 1- F_{\epsilon}\left(\frac{v-z}{b^{(n)}}\right) \right].
\end{align*}
Also, $\sigma_{x,n}(z)= [\Lambda_{n}^{*''}(x;z)]^{-\frac{1}{2}}$ and $\theta_{x,n}(z) = \Lambda_{n}^{*'}(x;z)$. Furthermore, for any $z_0$ positive, $\sup_{z\leq z_0}r_n(z)\to0$ as $n\to\infty$ and $\theta_{x,n}(z)$ and $\sigma_{x,n}(z)$ are continuous in $z$ and $x$ and converge to $\theta_{x}(z)$ and $\sigma_x(z)$ respectively, as $n\to\ff$.
\end{thm}

\begin{proof}

Fix $x$ in a compact $K\subset(\mu_U,\infty)$ and $z\le z_0$.
Let $S_n(z)=\sum_{i=1}^n Y_i^{(n)}(z)$ with $Y_i^{(n)}(z)=U_i^{(n)}X_i^{(n)}(z)$, and let
$\theta=\theta_{x,n}(z)$ solve $\partial_\theta\Lambda_n(\theta,z)=x$.
Write $v^2=\sigma_{x,n}^2(z)=\partial_\theta^2\Lambda_n(\theta,z)$.
Perform the Cram\'er tilt at $\theta$: under $\mathbf P_\theta$,
$\mathbf E_\theta Y_i^{(n)}(z)=x$ and $\mathrm{Var}_\theta(Y_i^{(n)}(z))=v^2$.
Then
\[
\mathbf P(S_n(z)\ge nx)
= e^{-n(\theta x-\Lambda_n(\theta;z))}\,
\mathbf E_\theta\left[e^{-\theta\,(S_n(z)-nx)}\,\mathbf 1\{S_n(z)\ge nx\}\right].
\]
Set $T_n:=\frac{S_n(z)-nx}{v\sqrt{n}}$ and $a_n:=\theta v\sqrt{n}$.
The bracket equals $\mathbf E_\theta[e^{-a_n T_n}\mathbf 1\{T_n\ge 0\}]$. Next, we invoke the \emph{uniform CLT} from Proposition \ref{thm:CCLT}. By the uniform tilted Berry-Esseen bound (Prop.~C.1) and (L2),
\[
\sup_{t\in\mathbb R}\big|\mathbf P_\theta(T_n\le t)-\Phi(t)\big| \le \frac{C_1}{\sqrt{n}},
\]
with $C_1$ independent of $(x,z)\in K\times(-\infty,z_0]$.
By (L1)–(L2) and strict convexity, there exist $0<c\le C<\infty$ such that
$c\le \theta\le C$ and $c\le v\le C$ uniformly on $K\times(-\infty,z_0]$; hence
$a_n=\theta v\sqrt{n}\asymp\sqrt{n}$ uniformly.
The last step is the evaluation of the \emph{Laplace transform of the positive part}. Let $F_n$ be the cdf of $T_n$ under $\mathbf{P}_\theta$ and write
\[
\mathbf{E}_\theta[e^{-a_n T_n}\mathbf 1\{T_n\ge 0\}]
= \int_0^\infty e^{-a_n t}\,dF_n(t)
= \int_0^\infty e^{-a_n t}\,d\Phi(t)\ +\ \int_0^\infty e^{-a_n t}\,d(F_n-\Phi)(t).
\]
For the Gaussian term, a direct computation gives
\[
\int_0^\infty e^{-a t}\,d\Phi(t)=\frac{1}{a\sqrt{2\pi}}\Big(1+O(a^{-2})\Big),
\qquad a\to\infty.
\]
For the error term, integrate by parts and use the Berry-Esseen bound:
\[
\Big|\int_0^\infty e^{-a_n t}\,d(F_n-\Phi)(t)\Big|
\ \le\ a_n\int_0^\infty e^{-a_n t}\,\sup_{s}|F_n(s)-\Phi(s)|\,dt
\ \le\ \frac{C_1}{\sqrt{n}}.
\]
Since $a_n\asymp\sqrt{n}$, the $O(a_n^{-2})$ from the Gaussian term is $O(n^{-1})$ and thus dominated by $C_1/\sqrt{n}$.
Therefore,
\[
\mathbf E_\theta[e^{-a_n T_n}\mathbf 1\{T_n\ge 0\}]
= \frac{1}{\theta v\sqrt{2\pi n}}\Big(1+O(n^{-1/2})\Big),
\]
with constants uniform on $(x,z)\in K\times(-\infty,z_0]$.
Finally, combining with the tilt identity,
\[
\mathbf P(S_n(z)\ge nx)
= \frac{e^{-n\Lambda_n^*(x,z)}}{\sqrt{2\pi n}\,\theta\,v}\,\Big(1+O(n^{-1/2})\Big),
\]
which is the claimed formula with $v=\sigma_{x,n}(z)$ and uniform remainder.
The nonlattice assumption in (L2) guarantees the tilted Berry-Esseen bound; if a lattice correction is needed, the same argument holds with the usual continuity correction, and the remainder remains $O(n^{-1/2})$.
\end{proof}

\section{Appendix}\label{app:D}
In this appendix, we provide a brief description of self-neglecting functions.
\begin{defn}[Self-neglecting at infinity]\label{def:SN}
A measurable function $f:[x_0,\infty)\to(0,\infty)$ is called \emph{self-neglecting} (abbreviated \ $f\in\mathrm{SN}$) if
\[
\lim_{x\to\infty}\ \sup_{|t|\le T}\ \left|\frac{f\big(x+t\,f(x)\big)}{f(x)}-1\right|=0
\qquad\text{for every }T<\infty,
\]
and, in addition, $f(x)=o(x)$ as $x\to\infty$.
We will also use the equivalent “little-$o$ shift” notation
\[
\frac{f\big(x+o(f(x))\big)}{f(x)}\ \longrightarrow\ 1\qquad (x\to\infty),
\]
which follows from the uniform formulation above.
\end{defn}

\begin{prop}[Equivalent characterizations of SN]\label{prop:SN-equiv}
Let $f(\cdot)$ be eventually positive and measurable. The following are equivalent:
\begin{enumerate}\itemsep2pt
\item $f\in\mathrm{SN}$ in the sense of Definition~\ref{def:SN}.
\item $f(x+o(f(x)))/f(x)\to 1$ as $x\to\infty$.
\item For any sequence $x_n\to\infty$ and any bounded sequence $(t_n)$, one has
$\displaystyle f\big(x_n+t_n a(x_n)\big)/f(x_n)\to 1$.
\end{enumerate}
\emph{Proof(Sketch).} (1)$\Rightarrow$(2) is immediate; (2)$\Rightarrow$(3) by choosing $o(f(x_n))=t_n f(x_n)$; (3)$\Rightarrow$(1) follows by a standard diagonal/compactness argument in $t$.
\end{prop}

\begin{lem}[A convenient sufficient condition]\label{lem:SN-derivative}
If $f(\cdot)$ is eventually $C^1$ with $f(x)=o(x)$ and $f'(x)\to 0$ as $x\to\infty$, then $f(\cdot)\in\mathrm{SN}$.
In particular, if $Q\in C^2(\Real)$ with  $Q'(x)\to\infty$ and $Q''(x)=o\big(Q'(x)^2\big)$, then
\[
f(x):=\frac{1}{Q'(x)}\ \in\ \mathrm{SN}.
\]
\emph{Proof sketch.} By the mean value theorem,
$f(x+t f(x)) = f(x)+t\,f(x)\,f'(\xi_x)$ for some $\xi_x$ between $x$ and $x+t f(x)$; since $f'(\xi_x)\to 0$
uniformly on $|t|\le T$, the ratio tends to $1$. For the particular case, $f'(x)=-Q''(x)/Q'(x)^2\to 0$ and $f(x)=o(x)$ as $Q'\to\infty$.
\end{lem}

\begin{rem}[Time change and window scaling]\label{rem:SN-timechange}
If $f(\cdot)\in\mathrm{SN}$ and $S(x):=\int_{x_0}^x \frac{dt}{f(t)}$, then $S(x)\to\infty$ and
\[
S\big(x+t\,f(x)\big)-S(x)\ \longrightarrow\ t
\quad\text{as }x\to\infty,\ \text{uniformly for }|t|\le T.
\]
This identity underlies the width of the Laplace window (of order $1/\sqrt{n|\tilde h_n''(z_*)|}$) in our sharp Large Deviation analysis.
\end{rem}

\begin{rem}[Relation to slow variation]\label{rem:SN-slow}
(a) If $L$ is slowly varying in the Karamata sense, then for any $f(\cdot)\in\mathrm{SN}$,
$\displaystyle L\big(x+t\,f(x)\big)/f(x)\to 1$ uniformly on compact $t$ (``Beurling slow variation'' relative to $f(\cdot)$).\\
(b) If $f(x)=x^\alpha L_0(x)$ with $0\le\alpha<1$ and $L_0$ slowly varying, then $f(\cdot)\in\mathrm{SN}$.\\
(c) The functions $f(x)=(\log x)^\beta$ ($\beta>0$); $f(x)=x^\alpha$ ($0<\alpha<1$); $f(x)=x/\log x$ are standard Beurling slow variation examples.
\end{rem}

\begin{rem}[How we use \texorpdfstring{$\mathrm{SN}$}{SN} here]\label{rem:SN-here}
In assumption \textup{\ref{asp:D1}} we set $a_\ep(t):=1/Q'_\ep(t)$. The condition $Q''_\ep=o\big((Q'_\ep)^2\big)$ implies $a'_\ep\to 0$,
hence $a_\ep\in \mathrm{SN}$ by Lemma~\ref{lem:SN-derivative}. This justifies the notation
$a_\ep\big(t+o(a_\ep(t))\big)\sim a_\ep(t)$ used in the Laplace localization.
\end{rem}

Next, we provide some examples in the log-smooth class.

\smallskip
\noindent\emph{(a) Generalized gamma and Weibull-type (right tail for $\epsilon$).}
If $g_\ep(t)=1-F_\ep(t)\sim c_\ep(t)\exp\{-Q_\ep(t)\}$ with 
\[
Q_\ep(t)=\xi_\ep\,t^{m_\ep}\quad(m_\ep>1,\ \xi_\ep>0),
\]
then $Q'_\ep(t)=\xi_\ep m_\ep t^{m_\ep-1}\uparrow\infty$ and 
$Q''_\ep(t)/\{Q'_\ep(t)\}^2=O(t^{-m_\ep})\to0$. Thus \ref{asp:D1} holds with $a_\ep(t)=1/Q'_\ep(t)$ self-neglecting.
This covers the classical Weibull tail and, more generally, the \emph{generalized gamma} law on $\R_+$ 
(density $\propto t^{k_\ep-1}\exp\{-(t/\theta_\ep)^{m_\ep}\}$) by absorbing polynomial factors into $c_\ep(t)$.

\smallskip
\noindent\emph{(b) Generalized normal.}
For GN$(\beta,\xi,\gamma)$, the right tail has $Q_\ep(t)=\xi\,t^\gamma$. 
When $\gamma>1$, $Q'_\ep(t)=\xi\gamma t^{\gamma-1}\uparrow\infty$ and $Q''_\ep/Q^{\prime 2}\to0$, so \ref{asp:D1} holds. 
The borderline case $\gamma=1$ (Laplace/exponential) does \emph{not} satisfy $Q'_\ep\uparrow\infty$, 
but it is already covered by the class~$\mathcal C$ results in Theorem \ref{thm:LDP_ndZ} and Proposition \ref{prop:LDP_ndZ}.

\smallskip
\noindent\emph{(c) Log-smooth left tails for the common factor $\mathcal{Z}_n$.}
If $f_{\mathcal{Z}_n}(-w)\sim c_{\mathcal{Z}_n}(w)\exp\{-R_{\mathcal{Z}_n}(w)\}$ with
\[
R_{\mathcal{Z}_n}(w)=\xi_{\mathcal{Z}_n}\,w^{m_{\mathcal{Z}_n}}\quad(m_{\mathcal{Z}_n}>1,\ \xi_{\mathcal{Z}_n}>0),
\]
then $r_{\mathcal{Z}_n}(w):=R_{\mathcal{Z}_n}'(w)=\xi_{\mathcal{Z}_n} m_{\mathcal{Z}_n} w^{m_{\mathcal{Z}_n}-1}\uparrow\infty$ and 
$R_{\mathcal{Z}_n}''(w)/\{r_{\mathcal{Z}_n}(w)\}^2=O(w^{-m_{\mathcal{Z}_n}})\to0$, so \ref{asp:D2} holds. This includes Weibull-type left tails and 
two-sided exponential-power and GN models with $\gamma>1$.

\smallskip
\noindent\emph{(d) Mixed log-smooth pairs.}
\ref{asp:D1}–\ref{asp:D2} do not require matching shapes: e.g., $\epsilon$ Weibull ($m_\ep>1$) and ${\mathcal{Z}_n}$ generalized gamma 
($m_{\mathcal{Z}_n}>1$) are admissible; all arguments go through with the balance condition \ref{asp:B1} selecting $M_n$.

\smallskip
\noindent\emph{(e) Borderline and other cases.}
Exponential ($m=1$) and regularly varying tails fall outside \ref{asp:D1}–\ref{asp:D2} but are already treated by 
the class~$\mathcal C$ theorems (Theorem \ref{thm:LDP_ndZ} and Proposition \ref{prop:LDP_ndZ} for generalized normal and regularly varying distributions), ensuring that our results cover both Weibull and generalized-gamma–type light tails and heavy tails under a unified presentation.

\begin{rem}[Connection to extreme-value theory]
The log-smooth assumptions {\rm\ref{asp:D1}-\ref{asp:D2}} are of von~Mises type and guarantee that the tails of
$\epsilon$ and $\mathcal{Z}$ lie in the \emph{maximum domain of attraction (MDA)} of the Gumbel law.
In extreme-value theory (EVT), the MDA of a limiting distribution refers to the class of distributions
whose suitably normalized maxima converge to that limit law. In particular, distributions with
\[
\bar F(x)=c(x)\,e^{-Q(x)}, \quad Q'(x)\uparrow\infty,\quad Q''(x)/Q'(x)^2\to 0,
\]
are in the Gumbel MDA. Thus, the generalized gamma (Weibull-type with shape $m>1$), generalized
normal/exponential-power with $\gamma>1$, and many other light-tailed models fall into our
framework automatically. This situates our log-smooth setting within the well-established Gumbel
domain of attraction in EVT.
\end{rem}

\section{Appendix} \label{app:E}
\setcounter{table}{0}
\renewcommand{\thetable}{E.\arabic{table}}

We now illustrate the consequences of the sharp large deviation results by deriving asymptotic expansions for two standard portfolio risk measures: Value-at-Risk (VaR) and Expected Shortfall (ES). 
These results demonstrate how our prefactor refinements translate into practical risk quantification, and clarify when portfolios operate in the large-deviation regime rather than in a central-limit or near-critical regime. For orientation, Proposition~\ref{thm:wlln_1type} provides the near-mean LLN/CLT baseline; the expansions below quantify the correction when the portfolio operates in the large-deviation regime of Theorems~\ref{thm:LDP_ndZ}–\ref{thm:LDP_gen_ndZ}.
 Recall that assumptions \ref{asp:L2} and \ref{asp:L3} hold.

\subsection{Value-at-Risk} \label{sec:var}
For $\alpha\in(0,1)$, the Value-at-Risk at level $\alpha$ is the quantile $x_{\alpha,n}:=\mathrm{VaR}_\alpha(L_n/n)$ defined by
\[
\mathbf{P}\left(\frac{L_n}{n}\ge x_{\alpha,n}\right)=1-\alpha,
\]
where the quantile is understood in the usual sense (e.g., via the infimum definition, or under continuity of the law of $L_n/n$).
Applying Theorem \ref{thm:LDP_ndZ} yields
\[
n^{-1/2} e^{-n\Lambda^*_U(x_{\alpha,n})}
\,[H_n(-\tilde M_n)]^{-1} e^{-n\phi_n(-\tilde M_n)} C
= 1-\alpha + o(1), \qquad n\to\infty,
\]
where $C=\psi_{\ff}$, from which a second-order expansion of $x_{\alpha,n}$ follows. Tables \ref{table:GG} and \ref{table:GP} in Appendix \ref{app:E} contain numerical values of $\mathrm{VaR}_{\al}$ for $\alpha=0.95$, $\alpha=0.99$, and $\alpha=0.999$ under Gaussian and Pareto tails.

Let $\mu_U = {\bf E}[U]$ and $\sigma_U^2 = \text{Var}(U).$
To extract explicit expansions, we expand $\Lambda_U^*(x)$ around $\mu_U$. Setting $x_{\alpha,n} = \mu_U + y_n$
and applying a Taylor expansion up to third order around $\mu_U$, we obtain
\begin{align}\label{VaR-1}
\Lambda_U^*(x_{\alpha,n}) = \Lambda_U^*(\mu_U + y_n) = \Lambda_U^*(\mu_U) + (\Lambda_U^*)'(\mu_U) y_n + \frac{1}{2}(\Lambda_U^*)''(\mu_U) y_n^2 + O(y_n^3),
\end{align}
Noticing that $\Lambda_U^*(\mu_U) = 0$, $(\Lambda_U^*)'(\mu_U) = 0$, and $(\Lambda_U^*)''(\mu_U) = 1/\sigma_U^2$
we obtain, by
substituting into (\ref{VaR-1}), that
\begin{eqnarray*}
\frac{1}{n}\left(-\log(1-\alpha)-\log H_n(-\tilde M_n) -\frac{1}{2}{\log n} -n\phi_n(-\tilde M_n) +\log  C\right) =  \frac{1}{2}(\Lambda_U^*)''(\mu_U) y_n^2 +o(1), \quad \text{as}~ n \to \ff 
\end{eqnarray*}
which yields
\begin{eqnarray*}
y_n = \left(\frac{2\Lambda_U^*(x_{\alpha,n})}{(\Lambda_U^*)''(\mu_U)}\right)^{\frac{1}{2}} +o(1) \quad \text{as}~ n \to \ff.
\end{eqnarray*}
Now using the definition of $y_n$, we obtain that as $n\to\ff$
\begin{eqnarray*}
x_{\alpha,n} =
 \mu_U + \left( \frac{2 \, \sigma_U^2[-\log(1 - \alpha) -\log H_n(-\tilde M_n) -\frac{1}{2}\log n - n\phi_n(-\tilde M_n) + \log C]}{n} \right)^{1/2} +o(1).
\end{eqnarray*}
where we have used that $(\Lambda_U^*)''(\mu_U) = 1/\sigma_U^2$
in the last step.
In the degenerate case, one can follow the same idea to obtain for $n\to\ff$, with $C=\psi_{\ff}(\kappa)(2\pi)^{-\frac{1}{2}}$
\begin{eqnarray*}
    x_{\alpha,n} = \mu_U + \left(\frac{2\sigma_U^2[-\log(1-\alpha)-\frac{1}{2}\log n + \log C]}{n}\right)^{\frac{1}{2}} + o(1).
\end{eqnarray*}
For large $n$, the approximation of $x_{\alpha,n}$ is meaningful only when $\alpha$ is close to one, since our large deviation results concern the extremal behavior, where the tails are dominated by the deviations from $U$. For other cases, one can expand upon this idea and then follow the results in \cite{collamore2022sharp} to obtain a corresponding estimate.

\subsection{Expected Shortfall} \label{sec:es}
Expected Shortfall at level $\alpha \in (0,1)$ is
\[
\mathrm{ES}_\alpha\left(\frac{L_n}{n}\right) 
= \frac{1}{1-\alpha}\, \mathbf{E}\left[ \frac{L_n}{n}\,\mathbf{1}\left\{\frac{L_n}{n} \ge x_{\alpha,n}\right\}\right].
\]
Notice that from Section \ref{sec:var}, for large $n$, $x_{\alpha,n}>\mathbf{E}[U]$ is close to $\mathbf{E}[U]$, and $\Lambda_U^*(x)$ is strictly increasing in $x\in[x_{\alpha,n},\infty)$ (see Section \ref{sec:preliminary}). Now applying the Laplace method (\cite{olver1997asymptotics,fukuda2025higher}), we obtain for $n\to\ff$,
\begin{eqnarray*}
\int_{x_{\alpha,n}}^{\infty} e^{-n\Lambda_U^*(x)}\mathrm{d}x = \frac{e^{-n\Lambda_U^*(x_{\alpha,n})}}{n (\Lambda_U^*)'(x_{\alpha,n})} + o(1).
\end{eqnarray*}
Thus as $n\to\ff$, with $C=\psi_{\ff}$,
\begin{eqnarray*}
\text{ES}_{\alpha}\left(\frac{L_n}{n}\right) = x_{\alpha,n}+ \frac{C e^{-n\phi_n(-\tilde M_n)} n^{-\frac{1}{2}} [H_n(-\tilde M_n)]^{-1} }{1-\alpha}\cdot \frac{e^{-n\Lambda_U^*(x_{\alpha,n})}}{n (\Lambda_U^*)'(x_{\alpha,n})} + o(1) = x_{\alpha,n}+ \frac{1}{n \theta_{x_{\alpha,n}}} + o(1).
\end{eqnarray*}
Turning to the degenerate case, using Theorem \ref{thm:LDP_kappa} and similar calculations, with $C=\psi_\ff(\kappa)(2\pi)^{-\frac{1}{2}}$, we have for $n\to\ff$,
\begin{eqnarray*}    
\text{ES}_{\alpha}\left(\frac{L_n}{n}\right) = x_{\alpha,n}+ \frac{C n^{-\frac{1}{2}}}{1-\alpha}\cdot \frac{e^{-n\Lambda_U^*(x_{\alpha,n})}}{n (\Lambda_U^*)'(x_{\alpha,n})} + o(1) = x_{\alpha,n}+ \frac{1}{n \theta_{x_{\alpha,n}}} + o(1).
\end{eqnarray*}

\subsection{Numerical experiments}
We now provide numerical illustrations of $\mathrm{VaR}_{\alpha}$ and $\mathrm{ES}_{\alpha}$ under different distributional assumptions for the idiosyncratic factor $\epsilon$ and the common factor $\mathcal{Z}$. In all examples, we fix $U\sim U(0,1)$, $b=0.5$, and $v=0$.

\smallskip

\noindent\textbf{Case 1: Gaussian $\epsilon$ and Gaussian $\mathcal{Z}$.}  
Table~\ref{table:GG} reports VaR and ES estimates for the setting $\epsilon\sim N(0,1)$ and $\mathcal{Z}\sim N(0,1)$. 

\begin{table}[htbp]
\centering
\begin{tabular}{llccccc}
\hline
 & & $n=10$ & $n=50$ & $n=100$ & $n=500$ & $n=1000$ \\ 
\hline
\multirow{3}{*}{VaR} 
  & $\alpha=0.95$  & 0.565 & 0.522 & 0.514 & 0.505 & 0.503 \\ 
  & $\alpha=0.99$  & 0.630 & 0.550 & 0.534 & 0.513 & 0.509 \\ 
  & $\alpha=0.999$ & 0.711 & 0.589 & 0.561 & 0.526 & 0.518  \\ 
\hline
\multirow{3}{*}{ES} 
  & $\alpha=0.95$  & 0.692 & 0.597 & 0.572 & 0.537 & 0.528 \\ 
  & $\alpha=0.99$  & 0.692 & 0.583 & 0.558 & 0.526 & 0.518 \\
  & $\alpha=0.999$ & 0.745 & 0.607 & 0.574 & 0.532 & 0.522 \\ 
\hline
\end{tabular}
\caption{$\mathrm{VaR}_{\alpha}$ and $\mathrm{ES}_{\alpha}$ under Gaussian $\epsilon$ and Gaussian $\mathcal{Z}$.}
\label{table:GG}
\end{table}
\smallskip

\noindent\textbf{Case 2: Gaussian $\epsilon$ and Pareto $\mathcal{Z}$.}  
Table~\ref{table:GP} presents the results when $\epsilon\sim N(0,1)$ and $\mathcal{Z}\sim \text{Pareto}(x_m=1,\alpha=2)$. Compared with the Gaussian case, the heavy-tailed distribution of $\mathcal{Z}$ yields systematically larger values of both $\mathrm{VaR}_{\alpha}$ and $\mathrm{ES}_{\alpha}$, reflecting the heavier tail risk.

\begin{table}[htbp]
\centering
\begin{tabular}{llccccc}
\hline
 & & $n=10$ & $n=50$ & $n=100$ & $n=500$ & $n=1000$ \\ 
\hline
\multirow{3}{*}{VaR} 
  & $\alpha=0.95$  & 0.664 & 0.560 & 0.539 & 0.514 & 0.510 \\ 
  & $\alpha=0.99$  & 0.718 & 0.587 & 0.559 & 0.524 & 0.516 \\ 
  & $\alpha=0.999$ & 0.777 & 0.619 & 0.582 & 0.535 & 0.524 \\ 
\hline
\multirow{3}{*}{ES} 
  & $\alpha=0.95$  & 0.712 & 0.587 & 0.561 & 0.526 & 0.518 \\ 
  & $\alpha=0.99$  & 0.752 & 0.606 & 0.573 & 0.531 & 0.521 \\
  & $\alpha=0.999$ & 0.801 & 0.632 & 0.592 & 0.540 & 0.528 \\ 
\hline
\end{tabular}
\caption{VaR and ES under Gaussian $\epsilon$ and Pareto $\mathcal{Z}$.}
\label{table:GP}
\end{table}

For comparison, the central limit approximation (cf.\ Proposition~\ref{thm:wlln_1type}) would suggest $\mathrm{VaR}_\alpha \approx 0.5$ uniformly in $\alpha$, missing the systematic inflation of tail risk. Sharp large deviation estimates capture this effect, with the correction especially pronounced in the Pareto case.

\section{Appendix}\label{app:F}
\begin{rem}[Allowing $U$ to depend on $X$]\label{rem:U-dep-X}
All results extend when the distribution of $U$ depends on $X\in\{0,1\}$. Conditional on $\mathcal Z_n$, let $(U_i^{(n)},X_i^{(n)})$ be i.i.d., with
\[
X_i^{(n)}=\mathbf 1\{\mathcal Z_n+b^{(n)}\epsilon_i^{(n)}\le v\},\qquad
U_i^{(n)}\mid\{X_i^{(n)}=x\}\sim P_{U^{(n)}\mid X=x},\quad x\in\{0,1\}.
\]
Since $e^{\theta U X}\equiv 1$ when $X=0$, the conditional moment generating function in all formulas is replaced by
\[
\Lambda_n(\theta;z)=\log\Big(p_n(z)\,\lambda^{(n)}_1(\theta)+1-p_n(z)\Big),\qquad
\lambda^{(n)}_1(\theta):=\mathbf{ E}\big[e^{\theta U_1^{(n)}}\mid X_1^{(n)}=1\big],
\]
and every occurrence of $\lambda_{U_1^{(n)}}$ (resp.\ $\lambda_{U}$) is replaced by $\lambda^{(n)}_1$ (resp.\ $\lambda_1$), with
\[
\lambda_1(\theta):=\mathbf {E}[e^{\theta U}\mid X=1],\quad
\Lambda_1(\theta):=\log\lambda_1(\theta),\quad
\theta_x\ \text{solving}\ \Lambda_1'(\theta_x)=x.
\]
Assumption~\ref{asp:L3} (third tilted moments and nonlattice) is imposed for $U^{(n)}\mid X=1$ (uniform in $n$), which suffices for the CLT, Berry-Esseen, and Bahadur-Rao steps. Turning to \emph{Gibbs limits}, in the unbounded case ($\kappa=-\infty$), the product limit in the theorem holds with $P_{\theta_x}$ replaced by the tilt of $U\mid X=1$:
\[
\mathbf{P}^{\theta_x}_{U\mid X=1}(\cdot)=e^{\theta_x u-\Lambda_1(\theta_x)}\,\mathbf{ P}_{U\mid X=1}(\cdot),\qquad X\equiv1.
\]
In the boundary case ($\kappa>- \infty$), use the same one-step formula as in the theorem, but with $U$ replaced by $U\mid X$ and
\[
\Lambda(\theta;\kappa)=\log\big((1-p_{\kappa})+p_{\kappa}\lambda_1(\theta)\big),\qquad
p_{\kappa}=F_\epsilon\Big(\frac{v-\kappa}{b}\Big),\quad
\bar\theta_{x,\kappa}:\ \partial_\theta\Lambda(\bar\theta_{x,\kappa},\kappa)=x.
\]
All total-variation conclusions (for fixed $k$ and for $k_n=o(n)$) are unchanged.   
\end{rem}

\bibliographystyle{plainnat} 

\end{document}